\documentclass[11pt,a4paper]{article}
\usepackage{amssymb}
\usepackage{amsfonts}
\usepackage{amsthm}
\usepackage{amsmath}
\usepackage{bm}	
\usepackage{bbm}
\usepackage{hyperref}
\usepackage{enumerate}
\usepackage{color}
\usepackage[
    backend=bibtex, 
    natbib=true,
    maxbibnames=99,
    style=alphabetic,
    sorting=nyt,
]{biblatex}
\usepackage{verbatim}
\usepackage[cal=boondoxo]{mathalfa} 
\usepackage[normalem]{ulem}

\numberwithin{equation}{section}

\allowdisplaybreaks

\newcommand{\ud}{\,\mathrm{d}}

\newcommand{\mc}[1]{\mathcal{#1}}
\newcommand{\mf}[1]{\mathfrak{#1}}
\newcommand{\half}[0]{\frac{1}{2}}

\newcommand{\ol}[1]{\overline{#1}}

\newcommand{\mb}[1]{\mathbb{#1}}

\newcommand{\const}{\mathrm{const\,}}
\newcommand{\eps}{\varepsilon}
\newcommand{\str}{\operatorname{Str}\,}
\newcommand{\sdet}{\operatorname{Sdet}\,}
\newcommand{\Ber}{\operatorname{Ber}}

\newcommand\tr{\operatorname{Tr}}

\newcommand\supp{\operatorname{supp}}

\theoremstyle{plain}
\newtheorem{thm}{Theorem}
\newtheorem{defin}{Definition}[section]
\newtheorem{prop}[defin]{Proposition}
\newtheorem{cor}[defin]{Corollary}
\newtheorem{lemma}[defin]{Lemma}
\newtheorem*{thm*}{Theorem}

\theoremstyle{remark}
\newtheorem{rem}[defin]{Remark}

\begin{document}

\title{\bf The density of states of 1D random band matrices via a supersymmetric transfer operator}
\author{ Margherita Disertori\textsuperscript{1}, 
Martin Lohmann\textsuperscript{2}, 
Sasha Sodin\textsuperscript{3}
}

\date{\today}

\maketitle

\footnotetext[1]{Institute for Applied Mathematics \& Hausdorff Center for Mathematics, 
University of Bonn,
Endenicher Allee 60,
D-53115 Bonn, Germany.
E-mail: disertori@iam.uni-bonn.de.}
\footnotetext[2]{
Department of Mathematics, University of British Columbia, Vancouver, BC, Canada V6T 1Z2. 
\mbox{E-mail}: marlohmann@math.ubc.ca.}
\footnotetext[3]{School of Mathematical Sciences, Queen Mary University of London, London E1 4NS, United Kingdom \& School of Mathematical Sciences, Tel Aviv University, Tel Aviv, 69978, Israel. 
Email: a.sodin@qmul.ac.uk.}

\begin{abstract}
Recently, T.\ and M.\ Shcherbina proved a pointwise semicircle law for the density of
states of one-dimensional Gaussian band matrices of large bandwidth. The main step 
of their proof is a new method to study the spectral properties of non-self-adjoint 
operators in the semiclassical regime. The method is applied to a transfer operator 
constructed from the supersymmetric integral representation for the density
of states.

We present a simpler proof of a slightly upgraded version of the semicircle law, 
which requires only standard semiclassical arguments and some peculiar 
elementary computations. The simplification is due to the use of supersymmetry, 
which manifests itself in the commutation between the transfer operator and a family of transformations of
superspace, and was applied earlier in the context of band matrices by Constantinescu. Other versions of this supersymmetry
have been a crucial ingredient in the study of the localization--delocalization transition by theoretical physicists.

 \end{abstract}

\tableofcontents

\section{Introduction}

\paragraph{Band operators and band matrices} Random band operators are popular toy models of disordered systems in theoretical 
physics. Their properties depend on a large parameter, called the bandwidth and 
denoted $W$. Informally, the large matrix elements lie in a band of width $W$ about
the main diagonal. One natural example is an Hermitian random operator $H$, represented
by the biinfinite Gaussian random matrix with covariance
\begin{equation}\label{eq:band1} \mathbb{E} H_{x, y} \overline{H_{x', y'}} = 
\frac{1}{W} \delta_{x, x'} \delta_{y, y'} \mathbbm{1}_{|x - y| \leq W}~,\quad
x,y \in \mathbb{Z}~.
\end{equation}
In this paper we mostly focus on a different example, the Hermitian Gaussian operator
with covariance
\begin{equation}\label{eq:band2} \mathbb{E} H_{x, y} \overline{H_{x', y'}} = 
  \delta_{x, x'} \delta_{y, y'} (-W^2 \Delta + \mathbbm{1})^{-1}_{xy}~, 
\end{equation}
where $\Delta$ is the one-dimensional discrete Laplacian and $\mathbbm{1}$
is the biinfinte identity matrix. 

Along with the infinite-volume operators $H$, one considers their finite-volume 
versions $H_N$ of dimension $N \times N$, called random band matrices. Again, we single out  Gaussian band matrices, and among them -- those with the covariances
\begin{equation}\label{eq:band1'} \mathbb{E} H_{x, y} \overline{H_{x', y'}} = 
\frac{1}{W} \delta_{x, x'} \delta_{y, y'} \mathbbm{1}_{|x - y| \leq W}~. 
\end{equation}
and
\begin{equation}\label{eq:band2'} \mathbb{E} H_{x, y} \overline{H_{x', y'}} = 
\delta_{x, x'} \delta_{y, y'} (-W^2 \Delta_N + \mathbbm{1}_N)^{-1}_{xy}~, 
\end{equation}
where 
\[ \Delta_N = \left(\begin{array}{cccccc}
-1 & 1 & 0 & \cdots & 0 & 0 \\
1 &  -2 & 1 & \cdots & 0 & 0\\
0 & 1 & -2 & \cdots & 0 & 0\\
\cdots &\cdots&\cdots & \cdots&\cdots&\cdots \\
0 & 0 & 0 &\cdots&1&-1
\end{array}\right) \]
is the Neumann Laplacian on $\{1, \cdots, N\}$. We regard (\ref{eq:band1'}) and (\ref{eq:band2'}) as finite-volume versions
of (\ref{eq:band1}) and (\ref{eq:band2}), respectively.

\paragraph{Localization length}
By the general theory of one-dimensional random operators \cite{PF}, finite-range band operators (including  (\ref{eq:band1}) ) exhibit localization for
any value of $W$, manifesting itself in pure point spectrum with exponentially
decaying eigenfunctions. The rate of exponential decay of the eigenfunctions
is known as the localization length and denoted $L_{\operatorname{loc}}$.
An essentially equivalent quantity is the minimal value of $N$ such that
$9/10$ of the $\ell_2$ mass of the eigenvectors is concentrated in a 
subinterval of length $N/10$. Anderson localization also occurs for long-range random operators with sufficiently fast decay of the off-diagonal elements (such as \ref{eq:band2});  see \cite{jakvsic-molchanov}.

A long-standing question is to determine the asymptotic dependence
of the localization length on the bandwidth. 
It is widely believed that $L_{\operatorname{loc}}$ scales as $W^2$ for large $W$,
for eigenvectors corresponding to energies $|E| < 2$.
This belief is supported by various convincing albeit not mathematically rigorous
arguments \cite{Casati:1990p3447,Casati:1993p35,FM1,FM2}; see further below.

On the rigorous side, Schenker proved \cite{schenker2009eigenvector} that $L_{\operatorname{loc}} \leq C  W^8$ for a class of band matrices
including (\ref{eq:band1}). His argument is reminiscent of 
the Mermin--Wagner theorem in statistical mechanics. A combination of 
the result of  \cite{schenker2009eigenvector}  with the recent Wegner estimate from \cite{PSSS} yields a 
slight improvement  $L_{\operatorname{loc}} \leq C W^7$. 

{ As to lower bounds, the results of \cite{ErdKn1,ErdKn2} pertaining to the quantum evolution imply a weak delocalization result for $W \gg N^{\frac67}$. A  stronger form of delocalization was established in \cite{erdHos2013delocalization} for $W\gg N^{\frac 45}$, and in \cite{he2018diffusion}, the constraint was relaxed to $W \gg N^{\frac79}$.   A genuine delocalization result was obtained for $W\gg N^{\frac 67}$ by Bao and Erd\H{o}s  \cite{bao2015delocalization}. The argument of  \cite{bao2015delocalization} combines the methods developed 
for the proof of universality for Wigner matrices ($W\sim N$, see \cite{erdos2012universality} for a review, as well as the monograph \cite{erdosyau}) with a supersymmetric analysis incorporating  superbosonization formulas \cite{bunder2007superbosonization} and the asymptotic method of \cite{ShchBlocks}. In the recent series of works \cite{bourgade2018random,bourgade2018random2,yang2018random} delocalization and the convergence of local eigenvalue statistics in the bulk to the sine process was established for $W \gg N^{\frac34}$; see  \cite{bourgade2018random3} for a review.}

We mention that band matrices and band operators admit a generalization
to higher spatial dimension.  If the spatial structure of the band is $d$ dimensional with $d\geq 3$, they are expected to exhibit an Anderson-type spectral phase transition similar to the conjectural metal-insulator transition in realistic solid state models. The dimension
$d = 2$ is critical, and localization is expected for all values of the band width. See further the review \cite{Sp_banded}.

\paragraph{Density of states}
The behaviour of the eigenvectors of $H$ is controlled by the quantity  $\langle \vert G_{xy}(E+i\eps)\vert^2\rangle $, where 
\[ G_{xy}(z) = G_{xy}[H](z) = (z- H)^{-1}_{xy} \]
 is the Green's function of the random matrix $H$,  $E \in \mathbb{R}$ is a spectral parameter, and $\langle \cdot \rangle$ denotes averaging over the disorder. In infinite volume, $\epsilon$ should be sent to zero, while $W$ is large but fixed. In finite volume, $\epsilon$
should be taken of order $1/N$. The quantity  $\langle \vert G_{xy}(E+i\eps)\vert^2\rangle $ also controls the properties
of the  (quenched) spectral measure $\mu_H$ of $H$, which is defined by
\[ G_{00}(z) = \int (z-\lambda)^{-1} d\mu_H(\lambda)~, \quad z \in \mathbb{C} \setminus \mathbb{H}~.  \]

In this paper we focus on a simpler quantity, the average Green function   $\langle G_{xx}(E+i\eps)\rangle$, and the related quantity $\langle G_{xy}(E+i\eps)G_{yx}(E+i\eps)\rangle$. The former one controls
the behaviour of the average density of states $\rho =   \rho[H]$, which can be defined as the 
Radon derivative of the disorder average $\langle \mu_H \rangle$ of $\mu_H$. We also consider the
average density of states in finite volume, $\rho_N =  \rho[H_N]$.

For Gaussian
random band matrices including both (\ref{eq:band1}) and (\ref{eq:band2}), 
the density of states exists in both finite and infinite volume by a  general argument of Wegner 
 \cite{Wegner_DOS}. The average density of states is related to $\langle G_{00} \rangle$:
 \[\begin{split} 
 \int  (z-E)^{-1} \rho(E) dE &=  \langle G_{00}[H](z) \rangle \\
 \int  (z-E)^{-1} \rho_N(E) dE &= \frac{1}{N} \sum_{x=1}^N \langle G_{xx}[H_N](z) \rangle \\
\end{split}\]
and (consequently)
\begin{equation}\label{eq:defdos}
\begin{split} 
\rho(E) &=  -\frac 1{\pi } \lim_{\eps\searrow 0} \Im \, \big\langle G_{00}[H](E+i\eps)\big\rangle\\
   \rho_N(E) &=  -\frac 1{\pi N} \lim_{\eps\searrow 0} \Im \, \big\langle \tr G[H_N](E+i\eps)\big\rangle~.\end{split}
\end{equation}
 In finite volume, $\rho_N$ can be also written as
\[
\rho_N(E)  
= \partial_E \frac{1}{N} \langle \# \left\{ \text{eigenvalues of $H_N$ in $(-\infty, E]$} \right\} \rangle
\]
justifying its name. Also, one has: $\lim_{N \to \infty} \rho_N(E) = \rho(E)$ (see the proof of
Proposition~ \ref{prop:viaI}).
\medskip
It is believed that the average $\langle G_{xy}(E+i\eps) G_{yx}(E+i\eps) \rangle$ and the average density of states do not reflect the localization properties
of the eigenvectors and the spectral type of $H$: see e.g.\ \cite{Wegner_DOS}. Still, the former quantities are of intrinsic interest. 

In the works \cite{BMP, KMP}, it was proved that as $W \to \infty$ the densities
$\rho(\lambda)$ converge weakly to the semicircle density
\[ \rho_{\mathrm{s.c.}}(E) = \frac{1}{2\pi} \sqrt{4-E^2} \,\, \mathbbm{1}_{|E|\leq2}~, \]
meaning that 
\begin{equation}\label{eq:tosc} \lim_{W \to \infty}\int \phi(E) \rho(E) dE = \int \phi(E) \rho_{\mathrm{s.c.}}(E) dE
\end{equation}
for any bounded continuous test function $\phi$. The arguments in these works apply to for a wide class of band matrices including (\ref{eq:band1}), (\ref{eq:band2}).

The available pointwise results are much less general, and we list them
below following a brief discussion of supersymmetry.

\paragraph{Supersymmetry} 
One of the powerful methods for the study of random operators is the supersymmetric formalism. First introduced by Wegner and  Sch\"afer and developed in the works of Efetov, it allows to rewrite the disorder averages of various
observables as high-dimensional superintegrals.  A general introduction
may be found in the monographs \cite{Wegbook,efetov1999supersymmetry}.

In the context of random band matrices, the supersymmetric approach was applied by Fyodorov and Mirlin \cite{FM1,FM2}, who confirmed the dependence of 
the localisation length on the bandwidth and also described the crossover
occurring as $W \asymp \sqrt{N}$. 
For Gaussian random band matrices, the average $\langle \vert G_{xy}(E+i\eps)\vert^2\rangle $ corresponds, through Berezin integration and superbosonization or certain formal versions of the Hubbard--Stratonovich transformation \cite{zirnbauer2004supersymmetry}, to a high dimensional super-integral dominated by a complicated saddle \emph{manifold}. The average density of states $\rho_N$ and
the average $\langle G_{xy}(E+i\eps) G_{yx}(E+i\eps) \rangle$ lead, in the same way, to an integral dominated by saddle \emph{points}. 

The four main steps in the works \cite{FM1,FM2} are the derivation of  the supersymmetric
integral representation; the $\sigma$-model approximation, in which  the integration domain is restricted to the saddle manifold; the continuum limit; and a semiclassical
analysis of the infinitesimal transfer operator. All these steps, and particularly the
last three, have so far not been put on firm mathematical ground. (See \cite{Shch2sigma} for recent progress on the second step.)

A version of the SUSY formalism that uses similar algebraic identities (Berezin integration), but simpler supersymmetries than those involved in superbosonization, had been used early on in rigorous investigations, e.g., of localization in $d=1$ random 
Schr\"odinger operators \cite{campanino1986supersymmetric,klein1986rigorous}.

\paragraph{Pointwise estimates}

The models (\ref{eq:band2}) and their counterparts in higher dimension are
especially convenient for supersymmetric analysis, since  the dual supersymmetric 
model  has nearest neighbour coupling (see Proposition~\ref{propalg}; in classical
statistical mechanics, the idea
of duality between carefully chosen long-range models and nearest neighbour
models goes back  to the work of Mark Kac \cite[Section 9]{kac1959partition}). 

This is why most of the pointwise results established to date pertain
to this class of operators. One exception is the upper bound 
\[ \rho(E) \leq C~, \quad \max_N \rho_N(E)  \leq C\] 
valid \cite{PSSS} for a reasonably wide class of Gaussian random band matrices including both (\ref{eq:band1}) and (\ref{eq:band2}), and their counterparts  in arbitrary dimension.

Much more is known for models of the form (\ref{eq:band2}). In dimension $d=3$, it was proved in \cite{disertori2002density}  that 
\[ \forall n \,\, \forall W \geq W_0(n) \,\, \forall |E| \leq 1.8\,\,\,\, \partial_E^n \rho(E)   \leq C_n \]
and that
\[ \forall |E| \leq 1.8 \,\,\,  \left|   \rho(E)   - \rho_{\mathrm{s.c.}}(E) \right| \leq C W^{-2}~.  \]
Corresponding results were also proved in finite volume. The argument of  \cite{disertori2002density}  relies on a cluster expansion  similar to the one used for the Wegner orbital model in \cite{constantinescu1987analyticity}.
Recently, a parallel result was also proved for $d=2$ in \cite{disertori2016density}.

In $d=1$, the cluster expansion methods of \cite{disertori2002density,disertori2016density} run into
difficulties. Instead, the method of transfer matrices can be used. While the method of transfer matrices has been successfully applied in the 
physical literature for this and more involved problems \cite{FM1,FM2}, a mathematical
justification is far from straightforward since the transfer matrix is not
self-adjoint. One possible strategy to perform semi-classical analysis for non-self-adjoint
operators was suggested in \cite{DS}; for now, this strategy was only implemented
for toy operators much simpler than the one considered here.

Recently, Mariya and Tatyana Shcherbina developed a different and very general 
 method of semi-classical analysis for  non-self-adjoint
operators \cite{shcherbina2016characteristic}. 
In the work \cite{shcherbina2016transfer}, they applied the method to the 
problem under discussion and proved
\begin{thm*}[\cite{shcherbina2016transfer}]
If $N\geq C W\log W $ and $\vert E\vert \leq \sqrt{\frac {32}9} $, then $\big\vert  \rho_N(E) - \rho_{sc}(E)\big\vert \leq \frac {C'}W$.
\end{thm*}

The main goal of the current paper is to provide an alternative and arguably simpler proof for the result of \cite{shcherbina2016transfer}. Most of the work in \cite{shcherbina2016transfer} is devoted to proving the spectral gap for the transfer operator and some properties of its top eigenfunction (Theorem 4.2 therein). The argument is based on methods developed in the non-supersymmetric framework in \cite{shcherbina2016characteristic}. Our proof, by contrast, exploits the supersymmetry of the problem. Apart from important simplification and reorganization of the algebra, supersymmetry manifestly fixes the top eigenvalue to $1$, implies a simple equation for the top eigenfunction, and reduces the proof of the spectral gap to an elementary spectral bound.  We also emphasize that, similarly to the classical applications of transfer 
operators in statistical mechanics, the original problem with two large parameters, $N$
and $W$, is reduced to a bunch of asymptotic problems containing only
one parameter.

Our method is non-applicable to problems lacking
supersymmetry, such as the correlation length of characteristic polynomials 
considered in \cite{shcherbina2016characteristic},  but has the advantage of simplicity.
It also allows to obtain stronger results, summarized in the theorems below. 
\begin{thm}\label{thm} For  $\vert E\vert < 2$ and $N\geq C(E) W\log W$  
 \begin{equation}\label{eq:thm1}
 \big\vert  \rho_N(E) - \rho(E)\big\vert \leq \frac {C(E)}{N} ~. 
 \end{equation}
 The  infinite volume density is well approximated by the semi-circle law:
\begin{equation}\label{eq:distfromsc}  
\big\vert  \rho(E) - \rho_{\mathrm{s.c.}}(E) \big\vert \leq \frac {C(E)}{W^{2}}~.
\end{equation}
\end{thm}
\begin{rem}The estimate  (\ref{eq:thm1}) together with (\ref{eq:distfromsc}) yields
\begin{align}\label{finitevolsc}
\big\vert  \rho_N(E) - \rho_{\mathrm{s.c.}}(E)\big\vert \leq \frac {C(E)}{W \log W} ~,\quad N \geq C(E) W \log W~. 
\end{align}
Note that, for large band width $N\sim   W\log W $, the finite volume deviation (\ref{eq:thm1}) is substantially bigger than the accuracy (\ref{eq:distfromsc}) of the semicircle law.
\end{rem}

\begin{thm}\label{thm2}  Assume that $\vert E\vert < 2$. 
For $N \geq C(E) W \log W$ the following hold. 
\begin{enumerate}
\item For any $1 \leq y, y' \leq N$,
\begin{align}\label{eq:th2eq1}
&\vert\langle G_{yy'}[H_N](E) G_{y'y}[H_N](E)\rangle\vert \leq 
\frac{C(E)}{W} \exp\Big(-\frac{c(E)\vert y'-y\vert}{W}\Big)
\end{align}
\item The finite volume density  $\rho_N$ is smooth and for all $n\geq 1$
\begin{equation}\label{eq:th2eq2}
|\partial_E^{n} \rho_N(E)| \leq    C(E)^{n} n!\  W^{n-1}.
\end{equation}
\end{enumerate}
\end{thm}

\begin{rem}
From \eqref{eq:thm1} and \eqref{eq:th2eq2} we obtain that the densities $\rho_N(E)$ and $\rho(E)$ admit an analytic extension to a domain of complex energies of the form
\[ \left\{ E + i \epsilon \, \mid \, |E| < 2~, \, |\epsilon| \leq c(E) / W \right\}~. \]
\end{rem}

\begin{rem}
Throughout the paper,  expressions of the form  $C(E)$, $C'(E)$, $c(E)$ are 
bounded from above and below on compact subintervals of $(-2, 2)$.
The values of $C(E)$, $C'(E)$, $c(E)$ and of the numerical constants 
$C, C', c$, et cet.\ 
may change from line to line.
\end{rem}

\paragraph{Plan of the paper} In Section \ref{alg}, we recall the representation of the mean density of states $\rho_N(E)$, as well as of the left-hand sides of (\ref{eq:th2eq1})--(\ref{eq:th2eq2}), via Berezin integration, and set up a supersymmetric transfer operator $\mf T$, (\ref{eq:Tsusydef}). 
Then we discuss a certain supersymmetry, that has already been described in \cite{constantinescu1988supersymmetric},
and is the supermatrix analogue of rotational symmetry and polar coordinates. This supersymmetry
allows to reduce the transfer operator (\ref{eq:Tsusydef}) to the much simpler one, $\mc T$ of (\ref{eq:defTf}), acting on functions of two real variables.  

In Section~\ref{ana} we analyse the transfer operator $\mc T$. The two main steps are (a) the construction of an approximate top eigenfunction $u_0$, $\mc T u_0 \approx u_0$, and (b) a bound on the restriction of $\mc T$ to a complement of the top eigenfunction. For a weakened version of Theorem~\ref{thm} (with a worse error term), it would suffice to use the simple ansatz
\[ u_0^{(0)}(\lambda) = \exp(-\alpha W (\lambda_1^2 + \lambda_2^2))\]
with an appropriately chosen $\alpha$. In order to prove the results as stated, we develop, in Section~\ref{seceigenfn}, a systematic construction (``supersymmetric WKB") which allows to find an approximate solution to an arbitrary order in $W^{-1/2}$ (for the purposes of the current paper, we use this construction up to order $5$). As for step (b), the key fact is Proposition~\ref{specbound}, proved via a  simple semiclassical argument. Corollary~\ref{cor:specbound} and Lemma~\ref{le:additionalprop} contain refinements needed to obtain the improved error term in Theorem~\ref{thm} and for Theorem~\ref{thm2}.

In Section \ref{prf}  we prove Theorems \ref{thm} and \ref{thm2}, for $|E| \leq \sqrt\frac{32}9$. The relatively simple Proposition~\ref{recon} would suffice to prove a  version of Theorem~\ref{thm} with  a worse error term; we work a bit more to prove the results as stated.

In Section~\ref{S:deform} we extend the argument to all $|E| < 2$, using a contour deformation and applying the general strategy of the Section~\ref{prf} to a power of the transfer operator (where the  power is chosen depending on the energy $E$).

With the exception of the Berezin integral representation (well explained elsewhere) and undergraduate
analysis, our proof is  self contained.

\section{Berezin algebra for the DOS and SUSY}\label{alg}

\subsection{Supersymmetric integral representation}

The Berezin integral representation of $ \rho_N(E)  $
involves the use of Grassmann variables. These will be denoted by the symbols $\rho,\ol\rho,\eta,\ol\eta $ and $\xi,\ol\xi $,
possibly with an index. All Grassmann variables anticommute with one another. The Grassmann algebra $\mathfrak G$
($\equiv $ free anticommutative) generated (over $ \mb C$) by all Grassmann variables admits a
$\mathbb{Z} / 2\mathbb{Z}$ grading, in which a monomial
in the generators is even or odd according to the parity of the number of symbols.

On the Grassmann  algebra  we define complex conjugation by $\rho\to\ol\rho,\ol\rho\to -\rho $ (etc.) on the generators,
and ordinary complex conjugation on the coefficients\footnote{In the terminology of \cite[\S 6.1]{Wegbook},
this is the conjugate of the second kind; see further the discussion
in \cite[\S 6.2]{Wegbook}.} . An element $f$ in the Grassmann algebra
generated by  the family $\{\rho_{x},\ol\rho_{x}\}_{x=1}^{N}$ has a unique decomposition  as
\begin{equation}\label{eq:grassfunc}
f=f (\{\rho_{x},\ol\rho_{x}\}_{x}) =\sum_{I,J\subset\{1,\dotsc ,N \} } c_{IJ}\  \rho_{I}\ol\rho_{J}
\end{equation}
where $c_{IJ}\in \mathbb{C},$ $ \rho_{I}=\prod_{i\in I}\rho_{i}$ and  $\ol\rho_{J}:=\prod_{j\in J} \ol\rho_{j}$
and the products are understood under some fixed arbitrary order on the set $\{1,\dotsc ,N \}.$ In particular, every  element can be
written as $f= f_{0}+n$ where $f_{0}\in \mathbb{C}$ and $n$ is nilpotent, i.e.\ there exists $k \geq 1$ such that $n^{k}=0.$
The Grassmann algebra has a natural decomposition into even and odd elements. 

Having a Grassmann algebra $\mathfrak G$ at hand, we consider the algebra of $C^\infty$- smooth functions $f: \mathbb C^m \to \mathfrak G$. A general element of this larger algebra has the form
\begin{equation}\label{eq:superfunc}
f (z_{1},\dotsc ,z_{m},\{\rho_{x},\ol\rho_{x}\}_{x \in \{1,\cdots,N\}})=
\sum_{I,J\subset\{1,\dotsc ,N \} } c_{IJ} (z_{1},\dotsc ,z_{m})\  \rho_{I}\ol\rho_{J}.
\end{equation}
Smooth multivariate functions $(z_{1},\dotsc ,z_{m})\mapsto f (z_{1},\dotsc ,z_{m})$ can be extended to functions of the form (\ref{eq:superfunc}) taking values
in the Grassmann algebra by replacing  the variables by some even elements of the  algebra $z_{i}\to z_{i}+n_{i}$, where $n_{i}$ is nilpotent,
and taking their Taylor series around $(z_{1},\dotsc ,z_{m})$ until it terminates by nilpotency. For example, for $m=1$ one has:
\[
f (z+n)= f (z)+ \sum_{j=1}^{k-1} \frac{f^{(j)} (z)}{j!} n^{j},\qquad n^{k}=0.
\]
The Berezin integral is the linear functional $\int $ that selects the top degree (in the generators)
coefficient in \eqref{eq:grassfunc} and multiplies it with $(2\pi)^{-\frac m2} $, where $m$ is the number of generators.
Thus, if $\rho,\ol\rho $ are the only generators,
\[
\begin{split}
\int \ud\ol\rho\ud\rho (a + b\,\rho + c\,\ol\rho + d \,\rho\ol\rho) = \frac{d}{2\pi}.
\end{split}
\]
where the differentials $\ud\rho,\ud\ol\rho$ anticommute with each other and with all Grasmann generators.

A $2m\times 2m $ supermatrix is a matrix  of the form
\[
\Big(\begin{smallmatrix} A & \Sigma  \\ \Gamma  & B  \end{smallmatrix}\Big),
\]
where $A,B$ (resp. $\Sigma,\Gamma$) are  $m\times m$ matrices whose matrix elements are even (respectively,  odd) elements of the
Grassmann algebra. 
The supertrace and superdeterminant of a supermatrix are defined as
\begin{align*}
\str \Big(\begin{smallmatrix} A & \Sigma  \\ \Gamma  & B \end{smallmatrix}\Big) &=\mbox{tr}A-\mbox{tr}B &
\sdet \Big(\begin{smallmatrix} A & \Sigma  \\ \Gamma  & B \end{smallmatrix}\Big)
&= \frac{1}{\det B}\det \big(A-\Sigma  B^{-1}\Gamma  ),
\end{align*}
where $\mbox{tr}$ and $\det$ are the ordinary trace and determinant.
We shall use the  symbol $R$ (possibly with an index) to denote
$2\times 2$ supermatrices of the form
\begin{equation}\label{eq:defR}
\begin{split}
R &= \Big(\begin{smallmatrix} a & \ol\rho \\ \rho & ib  \end{smallmatrix}\Big),
\end{split}
\end{equation}
where $a,b $ are  real numbers, and $\rho,\ol\rho $ are  generators of the
Grassmann algebra. By a superfunction $F$ of one or several variables $R$ we mean an expression of the form \eqref{eq:superfunc}
in which the $z$ variables are replaced by $a,b$.
See \cite{berezin2013introduction} for more details on these conventions and definitions.

\medskip\noindent We will use a supersymmetric representation for the average Green's
function, for which we will now introduce some notation.

Let $J:=-W^{2}\Delta_{N} +\mathbbm{1}$ the inverse of the covariance 
introduced in \eqref{eq:band2'}. Abbreviate $E_\eps = E+i\eps,\,\eps>0$. Let 
$\Gamma_a$ and $\Gamma_b$ be two horizontal lines in $\mb C$ oriented from
left to right, such that $\Gamma_a$ lies below the pole $E+i\eps\in\mb C$, 
{  while no additional constraint is needed for $\Gamma_{b}.$ }
For any suitable\footnote{It will be enough to consider
superfunctions whose coefficients $c_{IJ} (\{a_{x},b_{x} \}_{x \in \{1,\cdots,N\}})$ in the expansion \eqref{eq:superfunc}
are smooth functions of polynomial growth} superfunction 
$F= F\big((R_x)_{x=1}^{N}\big) $, introduce the SUSY average
\begin{align}\label{def:susymeasure}
\big\langle F\big\rangle_{SUSY} & =\int \prod_{x=1}^N \, \ud R_x \,  e^{-\half\str (R,JR) }
\prod_{x=1}^N \frac{1}{\sdet (E_{\eps}-R_x)}
F ((R_{x})_{x \in \{1, \cdots, N\}}) ,
\end{align}
where  $\str (R,JR) :=\sum_{xy}J_{xy}\str R_{x}R_{y}=\sum_{x}\str R_{x}^{2}+ W^{2} \sum_{x=1}^{N-1} \str (R_{x+1}-R_x)^2$
and  $E_{\eps}-R_x:=E_{\eps}\mathbbm{1}_{2}-R_x$ is a $2\times 2$ supermatrix.
Finally  $\ud R_x := \ud a_x\ud b_x\ud\ol\rho_x\ud\rho_x$, the integral in $a_x$
is taken along $\Gamma_a$, and the integral in $b_x$ is taken along $\Gamma_b$.
We will denote $\mathbf{a}= (a_{x})_{x=1}^{N}$ and $\mathbf{b}= (b_{x})_{x=1}^{N}.$

\begin{prop}\label{propalg}  In the above notation the following identities hold.
\begin{align}
1 = \big\langle1\big\rangle &= \big\langle  1\big\rangle_{SUSY} \label{eq:repr1} \\
\big\langle G_{yy}[H_N](E_\eps)\big\rangle &=
\big\langle (J\mathbf{a})_{y} \big\rangle_{SUSY} = \big\langle (iJ\mathbf{b})_{y} \big\rangle_{SUSY}
 \nonumber\\
&= \sum_{x} J_{yx} \big\langle a_{x} \big\rangle_{SUSY}=  \sum_{x} J_{yx} \big\langle (ib)_{x} \big\rangle_{SUSY}\label{eq:reprG} \\
\big\langle G_{yy'}[H_N](E_\eps)G_{y'y}[H_N](E_\eps)\big\rangle \!
 & = \! -J_{yy'} + \big\langle  (J\mathbf{a})_{y'}  [J(\mathbf{a}-i\mathbf{b})]_{y}   \big\rangle_{SUSY}\nonumber\\
&  = \! -J_{yy'} + \sum_{xx'} J_{yx} J_{y'x'}\big\langle  a_{x'} \str R_{x}   \big\rangle_{SUSY}
\label{eq:reprGG}\\
\partial_{E}^{n}\big\langle \tr G[H_N](E_\eps)\big\rangle &=N\delta_{n,1}+ (-1)^{n}
\big\langle  \big [\sum_{x} a_{x}\big ]\, \big  [\sum_{y} (a_{y}-ib_{y}) \big ]^{n}\big\rangle_{SUSY}\nonumber\\
&\hspace{-1cm}=N\delta_{n,1}+
 (-1)^{n}
\big\langle  \big  [\sum_{x} a_{x} \big ]\,  \big [\sum_{y} \str R_{y}  \big  ]^{n}\big\rangle_{SUSY},\quad n\geq 1.
\label{eq:reprGder}
\end{align}
\end{prop}
\begin{proof}
Let $\mathbf{E}= \mbox{diag} (E_{1},\dotsc ,E_{N}),
\tilde{\mathbf{E}}=\mbox{diag} (\tilde{E}_{1},\dotsc \tilde{E}_{N})\in \mathbb{R}^{N\times N}$ and
consider the generating function
\[
\mathcal{J} (\mathbf{E}, \tilde{\mathbf{E}}):=
\big\langle \frac{\det (\mathbf{E}+i\eps\mathbbm{1}_{N} -H)}{\det (\tilde{\mathbf{E}}+i\eps\mathbbm{1}_{N} -H)}
\big\rangle
\]
Then $\mathcal{J} (E\mathbbm{1}_{N}, E\mathbbm{1}_{N})=1$ for all $E\in \mathbb{R}.$ Moreover
\begin{align*}
\big\langle G_{yy}[H_N](E_\eps)\big\rangle&= -\partial_{\tilde{E}_{y}}\mathcal{J} (E\mathbbm{1}_{N}, E\mathbbm{1}_{N})
= \partial_{E_{y}}\mathcal{J} (E\mathbbm{1}_{N}, E\mathbbm{1}_{N})\\
\big\langle G_{yy'}[H_N](E_\eps)G_{y'y}[H_N](E_\eps)\big\rangle &= \partial_{\tilde{E}_{y'}}
(\partial_{\tilde{E}_{y}}+\partial_{E_{y}})
\mathcal{J} (E\mathbbm{1}_{N}, E\mathbbm{1}_{N}),
\end{align*}
and setting $E,\tilde{E}\in \mathbb{R}$
\[
\partial_{E}^{n}\big\langle \tr G[H_N](E_\eps)\big\rangle= -\partial_{\tilde{E}}
(\partial_{E}+\partial_{\tilde{E}})^{n}
\mathcal{J} (E\mathbbm{1}_{N}, \tilde{E}\mathbbm{1}_{N})_{|E=\tilde{E}}.
\]
By the same arguments as in  \cite[Sect. 3]{disertori2002density} the generating function can be written as follows
\[
\mathcal{J} (\mathbf{E}, \tilde{\mathbf{E}})= \int \prod_{x=1}^N dR_{x}\,   e^{-\half\str (R,JR) }
\prod_{x=1}^N \frac{1}{\sdet (\hat{E}_{x}+i\eps\mathbbm{1}_{2}-R_x)} 
\]
where $\hat{E}_{x}$ is the $2\times 2$ supermatrix
\[
\hat{E}_{x}:=\Big(\begin{smallmatrix} \tilde{E}_{x} & 0
\\ 0 & E_{x}  \end{smallmatrix}\Big),
\]
Setting $\mathbf{E}= \tilde{\mathbf{E}}=E\mathbbm{1}_{N}$ we obtain the measure in \eqref{def:susymeasure} whence
$1 = \mathcal{J} (E\mathbbm{1}_{N}, E\mathbbm{1}_{N}) = \big\langle  1\big\rangle_{SUSY}.$
Moreover
\begin{multline*}
 -\partial_{\tilde{E}_{y}}\mathcal{J} (E\mathbbm{1}_{N}, E\mathbbm{1}_{N})=
  \int \prod_{x=1}^N dR_{x}\,   e^{-\half\str (R,JR) }
 (-\partial_{\tilde{E}_{y}}) \prod_{x=1}^N \frac{1}{\sdet (\hat{E}_{x}+i\eps\mathbbm{1}_{2}-R_x)}
 \big |_{\mathbf{E}= \tilde{\mathbf{E}}=E\mathbbm{1}_{N}}\\
=
  \int \prod_{x=1}^N dR_{x}\,   e^{-\half\str (R,JR) }
 (\partial_{a_{y}})\prod_{x=1}^N \frac{1}{\sdet (E_{\eps}-R_x)}=
 \big\langle  \partial_{a_{y}}{\scriptstyle \frac{\str (R,JR)}{2}}  \big\rangle_{SUSY}=
 \big\langle  (J\mathbf{a})_{y}  \big\rangle_{SUSY}
\end{multline*}
where we used 
\[ (\partial_{\tilde{E}_{y}}+\partial_{a_{y}})\sdet (\hat{E}_{y}+i\eps\mathbbm{1}_{2}-R_y)=0 \] 
and we performed
integration by parts in $a_{y}.$ The same argument holds for $\partial_{E_{y}}\mathcal{J} (E\mathbbm{1}_{N}, E\mathbbm{1}_{N})$
using 
\[ (\partial_{E_{y}}-i \partial_{b_{y}})\sdet (\hat{E}_{y}+i\eps\mathbbm{1}_{2}-R_y)=0~.\]
This completes the proof of \eqref{eq:reprG}.
Performing integration by parts twice we obtain
\begin{align*}
\partial_{\tilde{E}_{y'}}
(\partial_{\tilde{E}_{y}}+\partial_{E_{y}})
\mathcal{J} (E\mathbbm{1}, E\mathbbm{1})&=
 \int \prod_{x=1}^N dR_{x}\, \prod_{x=1}^N \frac{1}{\sdet (E_{\eps}-R_x)}
 (-\partial_{a_{y'}}) [J(\mathbf{a}-i\mathbf{b})]_{y}  e^{-\half\str (R,JR) }
\\
&= -J_{yy'} + \big\langle  (J\mathbf{a})_{y'}  [J(\mathbf{a}-i\mathbf{b})]_{y}   \big\rangle_{SUSY}.
\end{align*}
This completes the proof of \eqref{eq:reprGG}. Similar arguments  yield
\begin{align*}
-\partial_{\tilde{E}}
(\partial_{E}+\partial_{\tilde{E}})^{n}
\mathcal{J} (E\mathbbm{1}_{N}, \tilde{E}\mathbbm{1}_{N})_{|E=\tilde{E}}&=
(-1)^{n+1}nN\big\langle  \big [\sum_{y} \str R_{y}\big ]^{n-1}\big\rangle_{SUSY}\\
& + (-1)^{n}
\big\langle  \big [\sum_{x} a_{x}\big ]\, \big  [\sum_{y} \str R_{y} \big ]^{n}\big\rangle_{SUSY},
\end{align*}
where we used that $(\partial_{a_{y}}-i\partial_{b_{y}}) (a_{y}-ib_{y})=0$ for any y.
The first summand can be written as
\begin{equation}\label{eq:susygamma}
\big\langle  \big [\sum_{y} \str R_{y}\big ]^{n-1}\big\rangle_{SUSY}= \partial_{\gamma}^{n-1}
\big\langle  e^{\gamma \sum_{y}\str R_{y}}\big\rangle_{SUSY|\gamma =0}=0\quad \forall n>1
\end{equation}
since by performing the translation $R\to R+\gamma \mathbbm{1}_{2}$ we get
\[
\big\langle  e^{\gamma \sum_{y}\str R_{y}}\big\rangle_{SUSY}= \mathcal{J} ((E-\gamma) \mathbbm{1}, (E-\gamma) \mathbbm{1})=1
\qquad \forall \gamma \in \mathbb{R}.
\]
This completes the proof of
\eqref{eq:reprGder}.
\end{proof}
 Identities of the form (\ref{eq:repr1}) for supersymmetric averages of supersymmetric
functions go back to the work of Parisi and Sourlas, cf.\ \cite[Chapter 15]{Wegbook}.
Note that the first two items appear also in this form, for example, in \cite{shcherbina2016transfer,disertori2016density}.

\paragraph{Translation to the saddle} Disregarding mathematical precision in our thinking about Grassmann variables,
we have the following heuristics. The factor 
\[
\exp\big(-\frac {W^2}2 \sum  \str (R_{x+1}-R_x)^2\big)
\]
forces the ``integration field'' $R_x$ to be almost constant, and the Laplace method then implies that the main
contribution is from field $R_x$ in the vicinity of the ``saddle points'' of $ e^{-\half\str R_x^2}(\sdet (E_\eps-R_x))^{-1}$.
Now we proceed with the rigorous argument. Setting the Grassmann variables equal to $0$ and neglecting the $\eps$ contribution it is easy to see that the
saddle points (i.e. the critical points of the logarithm) are at 
\begin{equation}\label{eq:absaddledef}
\begin{split}
a^\pm &= \half \Big(E\pm i\sqrt{4-E^2}\Big)\\
b^\pm &= (-i)\frac{1}{2} \Big(E\pm i\sqrt{4-E^2}\Big).
\end{split}
\end{equation}
$a^+$ is in the same half space as the pole, therefore we choose as contour 
$\Gamma_a$ the translate of the real axis with origin at $a^- $.
We will abbreviate $\mc E =a^{-}= \frac E2 - i\sqrt{1-\frac{E^2}4}$.
Note that $\mc E \bar {\mc E}=1.$
It will turn out later that $b^-$ is the dominant saddle point\footnote{If, as is the case for GUE,
$\rho_x\equiv \rho$ were a collective variable, i.e.\ a single Grassmann generator,
and similarly for $ \ol\rho$, this could be seen heuristically by evaluating
$I=e^{-\frac N2\str R^2}\sdet (E_\eps-R)^{-N} $ at
$R=R^\pm = \Big(\begin{smallmatrix} a^- & \ol\rho \\ \rho & ib^\pm \end{smallmatrix}\Big) $.
Set $I\vert_{R=R^\pm,\eps=0} = I_0\pm I^\pm \ol\rho\rho  $.
If there were no observable, the Grassmann integral $\ud\rho\ud\ol\rho $ would select $I^\pm $.
Explicit calculation gives $I^- = -N (2-\frac {E^2}2 +i E \sqrt{4-E^2} ) $, which is large (proportional to $N$),
while $I^{+} = -N(1-\vert \mc E\vert^2)=0. $
A more convincing argument would have to analyze ``fluctuations'' in $\rho,\ol\rho$.}.
We therefore choose as  contour $\Gamma_b$ the translate of the real axis with origin at $b^- .$
After this translation, the limit $\eps\searrow0 $ can be performed by dominated convergence. We introduce
the modified SUSY average
\begin{equation}\label{eq:newsusyav}
\langle F \rangle_{SUSY}' : =
\int \prod_{x=1}^N \, \ud R_x \,  e^{-\half\str (R,JR) } e^{-\mc{E} \sum_{x}\str R_{x} }
\prod_{x=1}^N \frac{1}{\sdet (\bar{\mc{E}}-R_x)}
F ((R_{x})_{x \in \{1,\cdots,N\}}),
\end{equation}
where  $\epsilon$ is now set equal to $0$ and the integrals in $a_{x},b_{x}$ are 
taken along the real axis. Moreover we used $E-\mc E=\bar{\mc E},$ 
\begin{align*}
\str (R+\mc E \mathbbm{1}_{2},J (R+\mc E \mathbbm{1}_{2}))& = \str (R,JR)+2\mc{E}  \str (R,J \mathbbm{1}_{2})+
\mc{E}^{2}  \str (\mathbbm{1}_{2},J \mathbbm{1}_{2})\\
&=
 \str (R,JR)+2 \mc{E} \sum_{x}\str R_{x}, 
\end{align*}
where the last term vanishes since $\str  \mathbbm{1}_{2}=0$. We then have
\begin{cor}\label{coralg} Let $|E| < 2$. Then
\begin{align}
\lim_{\eps\searrow 0}\big\langle G_{yy}[H_N](E_\eps)\big\rangle &= \mc E  +\langle (J\mathbf{a})_{y} \rangle_{SUSY}'=
\mc E  +\langle (J (i\mathbf{b}))_{y} \rangle_{SUSY}'
\label{eq:coralg.1}\\
\lim_{\eps\searrow 0}\langle G_{yy'}[H_N](E_\eps)G_{y'y}[H_N](E_\eps) \rangle &=
-J_{yy'} + \big\langle  (J\mathbf{a})_{y'}  [J(\mathbf{a}-i\mathbf{b})]_{y}   \big\rangle_{SUSY}'\label{eq:coralg.2}\end{align}
and for $n \geq 1$
\begin{equation}
\lim_{\eps\searrow 0}\partial_{E}^{n}\big\langle \tr G[H_N](E_\eps)\big\rangle=N\delta_{n,1}+ (-1)^{n}
\big\langle  \big [\sum_{x} a_{x}\big ]\, \big  [\sum_{y} (a_{y}-ib_{y}) \big ]^{n}\big\rangle_{SUSY}'~.
\label{eq:coralg.3}
\end{equation}
\end{cor}
\begin{proof}
\eqref{eq:coralg.1} follows from \eqref{eq:reprG}  by replacing $\mathbf{a} $ and  $i\mathbf{b}$ by
$\mathbf{a}+\mc{E}$ and $i\mathbf{b}+\mc{E}$. In the same way,  \eqref{eq:coralg.2} follows from 
 \eqref{eq:reprGG} by observing that, in addition,  \eqref{eq:coralg.1} implies 
 $\langle (J\mathbf{a})_{y} \rangle_{SUSY}'=\langle (J (i\mathbf{b}))_{y} \rangle_{SUSY}'$.
Finally \eqref{eq:coralg.3} follows from \eqref{eq:reprGder}  observing that
 $\langle \big [  \sum_{y}\str R_{y}\big]^{n} \rangle_{SUSY}'=0$ for all $n\geq 1$ by the same argument as in \eqref{eq:susygamma}.
\end{proof}
Note that $-\frac 1\pi \Im \mc E = \frac1{2\pi} \sqrt{4-E^2} $ is the semicircle density. 
Thus, for example, (\ref{finitevolsc}) amounts to showing that 
the remaining term in (\ref{eq:coralg.1}) is small for  large $W$ and $N$.

\subsection{Supersymmetric transfer operator}
The representation of Corollary \ref{coralg} can be stated using a supersymmetric transfer operator.
Recall that  $\str (R+\mc{E},J (R+\mc{E}))=\sum_{x}\str (R_{x}+\mc{E})^{2}+ W^{2} \sum_{x=1}^{N-1} \str (R_{x+1}-R_x)^2.$
We denote
\begin{equation}\label{eq:Vdef}
e^{-V(R)} := \frac{e^{-\half  \str (R+\mc E)^2}}{\sdet ( \ol{\mc E}-R)}
\end{equation}
and define
\begin{equation}\label{eq:Tsusydef}
\begin{split}
(\mf T F)(R) &:=  e^{-V(R)} \int \ud R'\,  e^{-\frac{ W^2}2 \str (R-R')^2} F(R') .
\end{split}
\end{equation}
\begin{rem}
Here we consider $e^{-V}$ as a single typographic symbol,
therefore the reader need not be concerned, for example, with the choice of 
a branch of the logarithm. Also, for now  we treat $\mf T$ as a formal operation rather than as an operator, therefore
we do not worry about the domain.
\end{rem}
Using \eqref{eq:Vdef} and \eqref{eq:Tsusydef} we have for all   $1\leq y\leq y'\leq N$
\begin{equation}\label{eq:transfa}
\begin{split}
\langle a_{y} \rangle_{SUSY}'  &=
 \int \ud R\, \big[\mf T^{y-1}e^{-V}](R)\  e^{V(R)} \,a \ \big[\mf T^{N-y}e^{-V}](R)\\
\langle a_{y} (a_{y'}-ib_{y'}) \rangle_{SUSY}' &= \langle a_{y} \str R_{y'} \rangle_{SUSY}' = I_{yy'},\qquad \mbox{where}\\
I_{yy'}&:=  \int \ud R \big[\mf T^{y-1}e^{-V}](R)\  e^{V(R)} \,a \  \big[\mf T^{y'-y}\str (\cdot)\mf T  ^{N-y'}e^{-V}](R),
\end{split}
\end{equation}
where  
$\str (\cdot)$ is the ``multiplication operator'' by $(\str R)$ and we use the convention $\mf T^{0}=\mathrm{id}.$
In the same way we get, for any $n\geq 1,$
\begin{equation}\label{eq:2.18}
\begin{split}
\big\langle  \big [\sum_{x} a_{x}\big ]\, \big  [\sum_{y} (a_{y}-ib_{y}) \big ]^{n}\big \rangle_{SUSY}'  &= \sum_{q=1}^{n} 
\sum_{\substack{n_{1},\dotsc n_{q}\geq 1 \\  n_{1}+\dotsb+ n_{q}=n}} \tfrac{n!}{n_{1}!\cdots n_{q}!}\sum_{m=0}^{n}
\sum_{\substack{1\leq y_{1}< y_{2}< \dotsb  \  y_{q}\leq N\\
   y_{m}\leq x<y_{m+1}}}
\hspace{-0.5cm}I_{x, y_{1},\dotsc ,y_{q}}^{m,n_{1},\dotsc ,n_{q}}
\end{split}
\end{equation}
where we defined $y_{0}=1,$ $y_{q+1}=N+1,$ $n_{0}=n_{q+1}=0,$ and  
\begin{equation}\label{eq:defTmultiple}
\begin{split}
&I_{x, y_{1},\dotsc ,y_{q}}^{m,n_{1},\dotsc ,n_{q}}:=
\int \ud R \big[ \mf T^{x-y_{m}}\str (\cdot)^{n_{m}} \dotsb \mf T^{y_{2}-y_{1}}\str (\cdot)^{n_{1}} \mf T^{y_{1}-1}e^{-V}](R)\\
&\qquad \cdot [e^{V(R)} \,a] \  \big[\mf T^{y_{m+1}-x} \str (\cdot)^{n_{m+1}}\mf T^{y_{m+2}-y_{m+1}} \dotsb \str (\cdot)^{n_{q}}
\mf T^{N-y_{q}}e^{-V}](R).
\end{split}
\end{equation}
Moreover, in the special case $m=0$ the first product is replaced by $\mf T^{x-1}e^{-V}$ while for $m=q$
the second product is replaced by  $\mf T^{N-x}e^{-V}.$ For $n=1$ \eqref{eq:2.18} can be simplified as follows
\[
\big\langle  \big [\sum_{x} a_{x}\big ]\, \big  [\sum_{y} (a_{y}-ib_{y}) \big ]\big \rangle_{SUSY}'  =
- \sum_{1\leq x<y\leq N} I^{0,1}_{x,y}- \sum_{1\leq y\leq x\leq N} I^{1,1}_{x,y} = -2\sum_{x<x'} I_{xx'}- I_{xx},
\]
where for $x\leq x'$ $I_{xx'}$ was defined in \eqref{eq:transfa} and we used $I^{0,1}_{xy}=I_{xy},$ $I^{1,1}_{xy}=I_{N-x+1, N-y+1}.$

We want to reduce $\mf T $ to an ordinary transfer operator, involving no Grassmann integration variables.
This could be done by hand, i.e. by expanding in the Grassmann variables and selecting top degree coefficients.
In this paper, we will instead exploit a supersymmetry that morally states that $\mf T $ commutes with
the superrotations of the supermatrix $R$.  This will yield a purely bosonic (Grassmann free)
representation of $\mf T $ in  the appropriately defined polar coordinates.

To make this intuition  precise, let $\xi,\eta$ be  two new Grassmann generators
(that may or may not depend on $\rho,\ol{\rho }$). Then the matrix
\begin{equation}\label{eq:Udef}
U_{\eta,\xi}:= \exp \Big(\begin{smallmatrix} 0 & \eta \\ \xi & 0  \end{smallmatrix}\Big)=
 \Big(\begin{smallmatrix} e^{\frac{1}{2}\eta \xi } & \eta \\ \xi & e^{-\frac{1}{2}\eta \xi }  \end{smallmatrix}\Big)
\end{equation}
is a ``superunitary'' rotation in the following sense:
$U_{\eta,\xi}^{-1}=U_{-\eta,-\xi}=
\exp \Big(\begin{smallmatrix} 0 & -\eta \\ -\xi & 0  \end{smallmatrix}\Big)$. 
It is easy to see that $\str R^{n}=\str (U_{\eta\xi}^{-1} R U_{\eta \xi})^{n}$ and
$\sdet R= \sdet (U_{\eta\xi}^{-1} R U_{\eta \xi})$ for any $n\geq 1.$

\begin{lemma}
The measure $\ud R$ is invariant under superunitary rotations: for any
Schwartz function $F (R)$ one has \begin{equation}\label{eq:inv}
\int \ud R\,  F (R)=\int \ud R\,  F (U_{\eta\xi}^{-1} R U_{\eta \xi}) \qquad \forall \eta,\xi. 
\end{equation}
\end{lemma}
\begin{proof}
This is a special case of the general framework of Berezin integration
\cite{berezin2013introduction} as applied to the conjugation supersymmetry at hand.
We will prove it here by direct computation. The function $F$ has a unique  decomposition as
\[
F (R)= F_{0} (a,b)+ \rho F_{1} (a,b)+ \ol{\rho } F_{2} (a,b)+\rho  \ol{\rho} F_{3} (a,b).
\]
Replacing in this formula
\begin{equation}\label{eq:Rrotation}
R'=U_{\eta\xi}^{-1} R U_{\eta \xi}=
\Big(\begin{smallmatrix}
a+\mathfrak n &   \ol{\rho}+\eta (a-ib)\\
\rho -\xi (a-ib) & ib + \mathfrak n  \end{smallmatrix}\Big),
\end{equation}
where $\mathfrak n:=\eta \xi (a-ib)+ ( \ol{\rho }\xi +\rho \eta) $, 
one gets
\[
F (R')= \tilde{F}_{0} + (\rho-\xi (a-ib)) \tilde{F}_{1}+
(\ol{\rho }+\eta (a-ib)) \tilde{F}_{2} + (\rho-\xi (a-ib)) (\ol{\rho }+\eta (a-ib)) \tilde{F}_{3},
\]
where we used $\mathfrak n^{2}= -2 \rho  \ol{\rho} \eta \xi$ and  $\mathfrak n^{3}=0$ and we defined 
\[
\tilde{F}_{i} := F_{i} (a+\mathfrak n,b-i\mathfrak n)= F_{i} (a,b)+\mathfrak n (\partial_{a}-i\partial_{b})F_{i} (a,b) +\frac{\mathfrak n^{2}}{2}
(\partial_{a}-i\partial_{b})^{2}F_{i} (a,b).
\]
Note that
\begin{align*}
 &(\rho-\xi (a-ib)) \mathfrak n= \rho \ol{\rho }\xi~,&    &(\rho-\xi (a-ib)) \mathfrak n^{2}=0&\\
 &(\ol{\rho }+\eta (a-ib)) \mathfrak n=-\rho \ol{\rho }\eta~,&     &(\ol{\rho }+\eta (a-ib))\mathfrak n^{2}=0.&
\end{align*}
Then 
\begin{align*}
F (R')-F (R)&=\tilde{F}_{0}-F_{0}+ (a-ib)[-\xi F_{1}+\eta F_{2}]+  \rho \ol{\rho } (\partial_{a}-i\partial_{b})
[\xi F_{1}-\eta F_{2}] \\
&+ (a-ib) [\ol{\rho }\xi  +\rho \eta+\eta \xi (a-ib) ]F_{3}-  \rho \ol{\rho }\eta \xi (a-ib)(\partial_{a}-i\partial_{b})F_{3}.
\end{align*}
Performing now integration over $\rho,\ol{\rho}$ we obtain a sum of terms of the form $\partial_{a}F_{i},\partial_{b}F_{i},$
$ (\partial_{a}-i\partial_{b})^{2}F_{i}$
or $(a-ib) (\partial_{a}-i\partial_{b})F_{i},$ hence the integral over $a$ and $b$ yields $0.$  This completes the proof.
\end{proof}

\paragraph{Polar decomposition} Every supermatrix
{ $\Big(\begin{smallmatrix} a & \ol\rho \\ \rho & ib  \end{smallmatrix}\Big)$ with $(a,b)\neq (0,0)$}
has a polar decomposition as follows
\begin{align}\label{eqpolar}
\Big(\begin{smallmatrix} a & \ol\rho \\ \rho & ib  \end{smallmatrix}\Big)
&= \left(\exp\Big(\begin{smallmatrix}  0 & -\eta \\ -\xi & 0   \end{smallmatrix}\Big) \right)
\Big(\begin{smallmatrix} \lambda_1 & 0 \\ 0 & i\lambda_2   \end{smallmatrix}\Big)
\left(\exp \Big(\begin{smallmatrix} 0 & \eta \\ \xi & 0  \end{smallmatrix}\Big)\right)=
U_{\eta \xi }^{-1} \Big(\begin{smallmatrix} \lambda_1 & 0 \\ 0 & i\lambda_2   \end{smallmatrix}\Big)U_{\eta \xi }
\end{align}
with (cf. \eqref{eq:Rrotation})
\begin{equation}\label{egvals}
\begin{aligned}
\lambda_1 &=a  +\frac{\ol\rho\rho}{a-ib}    &\quad  &
\lambda_2 &=& b   -i \frac{\ol\rho \rho}{a-ib}  \\
\eta &=  \frac{\ol\rho }{a -ib  } &\quad &
\xi &=& -\frac{\rho }{a -ib  }.
\end{aligned}
\end{equation}
We call $\lambda= (\lambda_1,\lambda_2) $ the eigenvalues of $R$, and $(\lambda,\eta,\xi)$ -- 
the polar coordinates for $R$. {  Note that $\lambda_{1}-i\lambda_{2}=a-ib\neq 0$ for $(a,b)\neq (0,0)$.}
Any function $F (R)$ invariant under superunitary rotations
depends on the eigenvalues $\lambda_{1},\lambda_{2}$ only. In this case we will write
\[
f (\lambda):=  F (\mbox{diag} (\lambda_{1},i\lambda_{2}))= F (R).
\]
In particular  we have $\str R^{n}= \lambda_{1}^{n}- (i\lambda_{2})^{n}$ $\forall n\geq 0$ and 
$\sdet R= \frac{\lambda_{1}}{i\lambda_{2}}.$ Hence the potential term in $\mc{T}$ \eqref{eq:Vdef} becomes
\begin{equation}\label{eq:Vpol}
e^{-V (R)}=\frac{\bar{\mc{E}}-i\lambda_{2}}{\bar{\mc{E}}-\lambda_{1}} 
e^{-\frac{1}{2} (\lambda_{1}+\mc{E})^{2}+ (\lambda_{2}-i\mc{E})^{2}}=
\frac{\bar{\mc{E}}-i\lambda_{2}}{\bar{\mc{E}}-\lambda_{1}} 
e^{-\half (\lambda_1^2+\lambda_2^2)- \mc E(\lambda_1-i\lambda_2)}=:e^{-V (\lambda )}.
\end{equation}
Note that $\lambda_{1},\lambda_{2}$ are \emph{even} elements of the Grassman algebra. We will see below that
inside an integral we can replace them by ordinary real variables. 

\begin{lemma}\label{le:unitaryinv}
The  transfer operator $\mf T$ preserves superunitary rotation invariance. Precisely, let $F (R)$ be such that all
coefficients $F_{I} (a,b)$ are smooth and  $|F_{I} (a,b)|\leq e^{K (a^{2}+b^{2})}$ with
$K<\frac{W^{2}}{2}.$ Assume $F$ is invariant
under superunitary rotations $F (R)=F (U_{\eta\xi}^{-1} R U_{\eta \xi})$ $\forall \eta, \xi.$
Then $\mf{T}F$ is also invariant: for all $\eta,\xi$, 
$[\mf{T}F] (R)=[\mf{T}F] (U_{\eta\xi}^{-1} R U_{\eta \xi})$.
\end{lemma}
\begin{proof}
We abbreviate $U_{\eta \xi }$ by $U.$ We have
\begin{align*}
(\mf T F)(U^{-1}RU) &=  e^{-V(U^{-1}RU)} \int \ud R'\,  e^{-\frac{ W^2}2 \str (U^{-1}RU-R')^2} F(R') \\
                   &= e^{-V(R)} \int \ud R'\,  e^{-\frac{ W^2}2 \str (R-UR'U^{-1})^2} F(UR'U^{-1})\\
                   &=  e^{-V(R)} \int \ud R'\,  e^{-\frac{ W^2}2 \str (R-R')^2} F(R')= (\mf T F)(R).
\end{align*}
where in the second line we used $V (U^{-1}RU)=V (R)$ (since both $\sdet $ and $\str $ are invariant)
and $F (R')= F(UR'U^{-1})$ (since we assumed $F$ is invariant). 
Finally in the last line we used \eqref{eq:inv} with $U$ replaced by $U^{-1}.$
\end{proof}

\begin{rem}\label{rem:trans1}
This lemma implies that we can always replace $R$ by the diagonal matrix $\mbox{diag} (\lambda_{1},i\lambda_{2})$ when
evaluating $(\mf T F)(R),$ as long as $F$ is invariant under superunitary rotations,  {  and $\str R=a-ib\neq 0.$} 
\end{rem}

\subsection{Some useful identities}
\label{sec:susyid1}
For certain functions the transformed $\mf{T}F$ can be explicitely computed.
Let  $z\in \mathbb{C},$ with $\Re (z)>0.$ We abbreviate
\begin{equation}\label{eq:defgaussmes}
d\mu_{z} (R):=  \ud R\,  e^{-\frac{z}{2 }\str R^{2} } 
\end{equation}
the supergaussian measure with variance $z^{-1}.$ By direct computation 
\begin{equation}\label{eq:susygauss}
\int d\mu_{z} (R)= \int \ud R\,  e^{-\frac{z}{2} \str R^2}= \int \ud a \ud b \ud \ol{\rho }\ud \rho\   e^{-\frac{z}{2} (a^{2}+b^{2}+2\ol{\rho }\rho )}=1.
\end{equation}

\begin{lemma}
Let $\mathbf{1} (R):=1 $ be the constant function,   
and $\Omega_{\alpha} (R):= e^{-W\alpha\str R^{2}}$ -- the gaussian with $\alpha\in \mathbb{C},$ $\Re \alpha >0.$
Then
\begin{align}\label{eq:1func}
(\mf{T}\mathbf{1}) (R)&= e^{-V (R)}\\
(\mf{T} \Omega_{\alpha }) (R)&=e^{-V (R)}\,   \Omega_{\alpha \mu} (R)\label{eq:gaussfunc}
\end{align}
where  $\mu:= \frac{W}{W+2\alpha}$. 
Moreover, for any function $F (R)$ such that $F$ is invariant under superunitary rotations,  $F (R)=f (\lambda)$, and
$f$ is a polynomial in $\lambda$ one has:
\begin{equation}\label{eq:trans}
(\mf{T} (\Omega_{\alpha}F)) (R) = e^{-V (R)}\,   \Omega_{\alpha \mu} (R) \int d\mu_{\frac{W^{2}}{\mu }} (R') F (R'+ \mu R)
\end{equation}
where we used the measure defined in \eqref{eq:defgaussmes} with $z=W^{2}/\mu$.
\end{lemma}

\begin{rem}
Since $\Re \alpha  >0$ we have   $\Re \mu >0$ and $\Re \alpha \mu   >0.$
\end{rem}

\begin{proof}
The first identity above follows by translating all variables $R'\mapsto R'+R$ and applying \eqref{eq:susygauss}. 
For the second identity  we first complete the square, then perform the translation $R'\mapsto R'+\mu R:$
\[
e^{V (R)}(\mf{T}\Omega_{\alpha}) (R)=   e^{-W\alpha\mu  \str R^2}
\int \ud R'\,  e^{-\frac{ W^2 }{2\mu} \str (R'-\mu R)^2} = e^{-W\alpha \mu \str R^2},  
\]
The second equality follows from  \eqref{eq:susygauss},  since
$\frac{ W}{\mu_{\alpha }}=W+2\alpha $ has positive real part. 
Repeating the same arguments we get \eqref{eq:trans}.  
\end{proof}

The following identities will be useful later.
\begin{lemma}\label{le:someidentities}
For any  $z\in \mathbb{C},$ with $Re (z)>0$ and supermatrix $R$ the following identities  hold:
\begin{align*}
(a)\qquad & \int d\mu_{z} (R')\ [\str (R'+R)]^{n} = (\str R)^{n},\qquad \forall n\geq 1\\
(b)\qquad  & \int d\mu_{z} (R')\ \str (R'+R)^{n} = \str R^{n},\qquad n=2,3\\
(c)\qquad &\int d\mu_{z} (R')\ \str (R'+R)^{4}= \str R^{4} + \frac{2 }{z}   (\str R)^{2} \\
(d)\qquad & \int d\mu_{z} (R')\ [\str (R'+R)^{3}]^{2} =[\str (R)^{3}]^{2} +\frac{9 }{z}   \str R^{4} + \frac{9 }{z^{2}} (\str R)^{2} 
\end{align*}
\end{lemma}
\begin{proof}
Without loss of generality we can assume $R$ to be diagonal $R=\mbox{diag} (\lambda_{1},i\lambda_{2}).$ Indeed by superunitary rotations {  any $R=\Big(\begin{smallmatrix} a & \ol\rho \\ \rho & ib  \end{smallmatrix}\Big) $ with $(a,b)\neq (0,0)$
can be reduced to a diagonal supermatrix. Finally the case $(a,b)= (0,0)$ can be recovered from  $(a,b)\neq (0,0)$ by continuity
remarking that both sides of the identities above are polynomials in the variables $a$ and $b.$}

\textbf{(a)} We have $ [\str (R'+R)]^{n}=  [\str R'+ \str R)]^{n}= \sum_{j=0}^{n} C_{j}^{n} (\str R')^{j} (\str R)^{n-j}.$
By identity \eqref{eq:localizationid} $\int d\mu_{z} (R')\  (\str R')^{j}=  (\str 0)^{j}=0$ $\forall j>0.$

\textbf{(b)}
For $n=2$ we have
\[
\int d\mu_{z} (R')\  \str (R'+R)^{2} = \str R^{2} + 2\int d\mu_{z} (R')\  \str (R'R)+ \int d\mu_{z} (R')\   \str (R')^{2}.
\]
By \eqref{eq:localizationid} $\int d\mu_{z} (R')\   \str (R')^{2}=0$ and the remaining integral
vanishes by parity  under $R'\to -R'.$
For $n=3$ we have
\[
\str (R'+R)^{3} = \str R^{3} + 3 \str (R'R^{2})+ 3  \str R(R')^{2}+ \str (R')^{3}.
\]
The second and fourth terms are  odd under $R'\to -R',$ hence the corresponding integrals vanish.
Note that
\[
(R')^{2}=
\Big(\begin{smallmatrix} {a'}^{2}+\ol{\rho }'\rho'  & (a'+ib')\ol{\rho}' \\ (a'+ib')\rho'  & -{b'}^{2}-\ol{\rho}'\rho'
\end{smallmatrix}\Big).
\]
Direct computation shows that the integral of each matrix component equals zero.

\textbf{(c)} Expanding we have
\[
\str (R'+R)^{4} = \str R^{4} + 4 \str {R'}^{3}R + 4  \str R^{3}R'+ 2\str (RR')^{2} +4 \str R^{2}{R'}^{2} + \str {R'}^{4}.
\]
The second and third term are odd, i.e.\ change sign  under $R'\to -R'$, hence the corresponding integrals vanish.
The integral of the first term vanish by  \eqref{eq:localizationid}. 
Since the integral of each matrix component in ${R'}^{2}$ equals zero also the fifth term disappears.
It remains  to consider  $\str (RR')^{2}.$ Since $R$ is assumed to be diagonal we have
\[
\str (RR')^{2}= (a'\lambda_{1})^{2}+ 2\ol{\rho }'\rho' \lambda_{1}i\lambda_{2}+ (b'i\lambda_{2})^{2}
\]
By direct computation
\[
 \int d\mu_{z} (R')\ \str (RR')^{2}= \frac{1}{z}\left[ \lambda_{1}^{2}+ (i\lambda_{2})^{2}- 2  \lambda_{1}i\lambda_{2}  \right]
= \frac{[\lambda_{1}- (i\lambda_{2})]^{2}}{z}= \frac{(\str R)^{2}}{z}.
\]

\textbf{(d)} Using $(b)$ we have $ \int d\mu_{z} (R')\ [\str (R'+R)^{3}]^{2}=  [\str R^{3}]^{2}+  \int d\mu_{z} (R')\ X$
where 
\begin{align*}
X&=  [\str (R'+R)^{3}- \str R^{3}]^{2}= [  3 \str (R'R^{2})+ 3  \str R(R')^{2}+ \str (R')^{3} ]^{2}\\
&=
 [\str (R')^{3}]^{2} + 9  [\str (R'R^{2})]^{2}+ 9 [ \str R(R')^{2}]^{2}+ 6 \str (R')^{3}  \str (R'R^{2})\\
& \quad + 6  \str (R')^{3}   \str R(R')^{2}+ 18  \str (R'R^{2})  \str R(R')^{2}
\end{align*}
The integral of the last two term vanish by parity while the integral of the first term vanishes by  \eqref{eq:localizationid}.
Since $R$ is diagonal we have $  \str (R'R^{2})= a' (\lambda_{1})^{2}- ib' (i\lambda_{2})^{2}.$
Hence
\[
 \int d\mu_{z} (R')\  [ \str R'(R)^{2}]^{2}= \frac{1}{z}[ (\lambda_{1})^{4}- (i\lambda_{2})^{4} ]= \frac{\str R^{4}}{z}
\]
Similarly $\str R(R')^{2}= {a'}^{2}\lambda_{1}+{b'}^{2} (i\lambda_{2})+\ol{\rho }'\rho' (\lambda_{1}+i\lambda_{2}).$ Hence
\[
 \int d\mu_{z} (R')\   [\str R(R')^{2})]^{2}= \frac{(\lambda_{1}-i\lambda_{2})^{2}}{z^{2}}= \frac{(\str R)^{2}}{z^{2}} 
\]
Finally $\str {R'}^{3}= {a'}^{3}- (ib')^{3}+\ol{\rho }'\rho' 3(a'+ib').$ Direct computation yields 
\[
 \int d\mu_{z} (R')\  \str (R')^{3}  \str (R'R^{2})= 0.
\]
This concludes the proof.
\end{proof}

As a direct consequence of this lemma, we can compute exactly the action of  $\mf{T}F$ for a certain functions.
This is done in the following corollary.
\begin{cor}\label{cor:exactformulas}
Let $\Re \alpha > 0$, and let  that $\Omega_{\alpha } (R)= e^{-W\alpha \str {R}^{2}}$, $\mu = W/(W + 2\alpha)$ as before. For any supermatrix $R$ 
\begin{align}
\mf{T}\Big (\Omega_{\alpha} (R') [\str R']^{n}  \Big ) (R) &= e^{-V (R)} \Omega_{\alpha \mu } (R)\ 
\mu^{n} [\str R]^{n} \qquad \forall n\geq 1\\
\mf{T}\Big (\Omega_{\alpha} (R')  \str {R'}^{n}  \Big ) (R) &=
 e^{-V (R)} \Omega_{\alpha \mu } (R)\  \mu^{n} \str R^{n} \qquad n=2,3\\
\mf{T}\Big (\Omega_{\alpha } (R')  [\str R']^{4}  \Big ) (R) &=
 e^{-V (R)} \Omega_{\alpha \mu } (R)\ 
\left[ \mu^{4} \str R^{4} + \frac{2 \mu^{3} }{W^{2}}   (\str R)^{2}   \right] \\
\mf{T}\Big (\Omega_{\alpha } (R')  [\str {R'}^{3}]^{3}  \Big ) (R) &=
 e^{-V (R)} \Omega_{\alpha \mu } (R)\ 
\left[ \mu^{6} [\str (R)^{3}]^{2} +\frac{9\mu^{5} }{W^{2}}   \str R^{4} + \frac{9 \mu^{4}}{W^{4}} (\str R)^{2}
  \right]
\end{align}
\end{cor}
\begin{proof}
Combine identity \eqref{eq:trans} with  Lemma \ref{le:someidentities}.
\end{proof}

\subsection{Transfer operator in polar coordinates}

In the following we will consider only functions $R\mapsto F (R)$ taking values  in even elements of
the Grassmann algebra and such that $F (R)$  contains no Grassmann generators
 except $\xi,\eta$ (equivalently $\bar\rho,\rho$). If such  $F$ is invariant under superunitary rotations
then  the corresponding function $f$ maps  $\mathbb{R}^{2}$ to $\mathbb{R}.$
Our goal is to write $\mf{T}F$ as an operator $ \mc{T}$ acting on $f$ instead. This requires to 
change coordinates in the integral $(a,b,\ol{\rho},\rho)\mapsto (\lambda_{1},\lambda_{2},\eta,\xi).$ The operation
is well defined only for functions that vanish at the origin:  if $F(R)$ is a ''nice''
function  (we will make this precise below) with $F(0) = 0 $, then by Berezin integration formula
(cf. \cite{berezin2013introduction})
\begin{equation*}\label{eq:polarchange}
\begin{split}
\int \ud R \, F(R)= \int \ud a\ud b\ud \ol{\rho }\ud \rho F (R)=
\int \ud\lambda_1\ud\lambda_2\ud\xi\ud\eta \frac{F\big(R(\lambda,\eta,\xi)\big)}{\Ber} ~,
\end{split}
\end{equation*}
where ${\Ber} = (\lambda_1-i\lambda_2)^2$ is the so-called Berezinian (the super-analog of the Jacobi determinant),
and $a,b,\lambda_1,\lambda_2 $ are real integration variables. The function appearing in the integral \eqref{eq:Tsusydef} defining $\mf{T}F$
is $e^{-\frac{W^{2}}{2}\str (R-R')^{2}}F (R')$ which does not necessary vanish in $0.$ To solve the problem one can decompose the integral as follows:
\[
e^{-\frac{W^{2}}{2}\str (R-R')^{2}}F (R')= \tilde{F}_{1} (R')+  \tilde{F}_{2} (R'),
\]
where the functions $\tilde{F}_{1} (R'),\tilde{F}_{2} (R')$ must satisfy: \textit{(a)} both functions are 'nice', in particular
integrable, \textit{(b)} $\tilde{F}_{1} (0)=0$ so that we can apply the change of coordinates and
 \textit{(c)} the integral $\int \ud R'\, \tilde{F}_{2} (R')$ is easy to compute without any coordinate change.
When $F$ is invariant under superunitary rotations, the following two choices are especially convenient.
\begin{align*}
e^{-\frac{W^{2}}{2}\str (R-R')^{2}}F (R')&= e^{-\frac{W^{2}}{2}\str (R-R')^{2}}[F (R')-F (0)] + e^{-\frac{W^{2}}{2}\str (R-R')^{2}}F (0),\\
e^{-\frac{W^{2}}{2}\str (R-R')^{2}}F (R')& =e^{-\frac{W^{2}}{2}\str (R-R')^{2}}\left[1-e^{-W^{2}\str RR'}\right]F (R') +
e^{-\frac{W^{2}}{2}\str (R^{2}+{R'}^{2})}F (R').
\end{align*}
In both cases the function $\tilde{F}_{1}$ vanishes at $0.$ Moreover the function $\tilde{F}_{2}$ can be evaluated
exactly
\begin{equation}\label{eq:susylocal}
\int \ud R'\, e^{-\frac{W^{2}}{2}\str (R-R')^{2}}F (0)= \int \ud R'\, e^{-\frac{W^{2}}{2}\str {R'}^{2}}F (R')=F (0),
\end{equation}
where the first integral can be performed directly as the one in \eqref{eq:1func}, while the second
integral is a consequence of the rotation symmetry (cf. Remark~\ref{rem:susyid} below).
We will see below these choices translate in two equivalent representations of the operator
 $\mc{T}$ acting on $f.$
We first need some notation.
For $f:\mathbb{R}^{2}\to \mathbb{R}$ we define
\begin{equation}\label{eq:defdeltaW}
\begin{split}
(\delta_0f)(\lambda) &:= f(0) \\
(\delta_{W^2}^\star f)(\lambda) &:=  \frac {W^2}{2\pi}
\int_{\mathbb R^2}  e^{-\half W^2  (\lambda'-\lambda)^2 }  f(\lambda')  \ud\lambda_1'\ud\lambda_2 '~.
\end{split}
\end{equation}
where 
\[
 (\lambda'-\lambda)^2 = (\lambda_{1}-\lambda_{1}')^{2}+ (\lambda_{2}-\lambda_{2}')^{2}
\]
coincides with $\str (R'-R)^{2}$ when both matrices are diagonal. Moreover  we denote by  $\Lambda$ the multiplication by
\[
\Lambda = \lambda_1-i\lambda_2=\str R 
\]
(from the $\Ber$ term in \eqref{eq:polarchange}) and
by $\bar \Lambda $  the multiplication by
$\bar \Lambda = \lambda_1+i\lambda_2$ so that
\[
(e^{-\bar \Lambda \Lambda }f) (\lambda )= e^{-\lambda^{2}}f (\lambda).
\]
Finally we denote by  $e^{-V} $ --- the  multiplication by $e^{-V (\lambda )}$ (cf. \eqref{eq:Vpol})

\begin{prop}\label{propsusy} 
 Assume $R\mapsto F (R)$ takes values in even elements of the Grassmann algebra and $F (R)$
 contains no Grassmann generators except $\xi,\eta$ (equivalently $\bar \rho, \rho$, since $\eta=\bar \rho (a-ib)^{-1}$ and
  $\xi=-\rho (a-ib)^{-1},$ cf.\ \eqref{egvals}). Assume further that  $F$ is invariant
 under superunitary rotations $F(R) = f(\lambda) $  so that
 \[
 f\in C^1(\mb R^2)~, \quad |f(\lambda)| \leq \exp(K |\lambda^2|)~,
 \]
where $K < W^2/2$. Then $\mf{T}F$ is also rotation invariant. Moreover, $\mf{T}$ can be represented as
 an operator acting  on  $C^1(\mb R^2) $ $(\mf TF) (R) = (\mc Tf)(\lambda),$ with
\begin{align}\label{eq:defTf}
 \mc{T}& := e^{-V}e^{-\frac {W^2}2\bar \Lambda \Lambda }\delta_0 + T~, \quad T := e^{-V}\Lambda \delta_{W^2}^\star \Lambda^{-1} = e^{-V}\delta_0 +e^{-V} T (\mathrm{id}-\delta_{0})
 \end{align}
\end{prop}
\begin{rem}
Note that $(\mc Tf)(0) = f(0)$, thus if $f$ vanishes at
the origin, so does $\mc Tf$. For $f$ vanishing at the origin we have then $ \mc{T}f= Tf.$
\end{rem}

\begin{rem}\label{rem:susyid}
A direct consequence of this result is the following localization identity:
\begin{equation}\label{eq:localizationid}
\int \ud R'\,  e^{-\frac{W^{2}}{2}\str {R'}^{2}} F (R')= (\mf{T}F) (0)= (\mc{T}f) (0)= f (0)=F (0).
\end{equation}
\end{rem}
The above decomposition appears already  in \cite{constantinescu1988supersymmetric} and 
is a special case of the general framework of Berezin integration
 \cite{berezin2013introduction} applied to the conjugation supersymmetry at hand.
Generalization to more complicated supersymmetries has been an important theoretical tool
in condensed matter physics. All such results are proven by an inspired sequence of elementary
applications of Stokes' theorem.
We give a proof of the present simple case for the convenience of the reader.
\begin{proof}
Let $F$ be as above. Then setting $\hat{\lambda}=\mbox{diag} (\lambda_{1},i\lambda_{2})$ we
get from Lemma \ref{le:unitaryinv}
\begin{align*}
(\mf{T}F) (R)&= (\mf{T}F) (\hat{\lambda})= e^{-V (\lambda )}
\int \ud R'\, e^{-\frac{W^2}2 \str(\hat{\lambda}-R')^2} F (R')\\
&=  e^{-V (\lambda )}
\int \ud R'\, e^{-\frac{W^2}2 \str(\hat{\lambda}-R')^2} f (\lambda'_{1},\lambda'_{2})\\
&=  e^{-V (\lambda )}
\int \ud R'\, e^{-\frac{W^2}2 \str(\hat{\lambda}-R')^2}  f\big(a'+ \tfrac{\ol\rho'\rho'}{a'-ib'},b'-i \tfrac{\ol\rho'\rho'}{a'-ib'}\big)\\
&=  e^{-V (\lambda )}
\int \ud R'\, e^{-\frac{W^2}2 \str(\hat{\lambda}-R')^2}  f\big(a'+ \tfrac{\ol\rho'\rho'}{a'-ib'},b'-i \tfrac{\ol\rho'\rho'}{a'-ib'}\big)\\
\end{align*}
The arguments of $f$ are even elements of the Grassmann algebra, hence
\[
  f\big(a'+ \tfrac{\ol\rho'\rho'}{a'-ib'},b'-i \tfrac{\ol\rho'\rho'}{a'-ib'}\big)=
    f (a',b')+  \tfrac{\ol\rho'\rho'}{a'-ib'}  (\partial_{a'}-i\partial_{b'})f\big(a',b')
\]
Moreover 
\[
e^{-\frac{W^2}2 \str(\hat{\lambda}-R')^2} = e^{-\frac{W^2}2 ( (a'-\lambda_{1})^{2}+ (b'-\lambda_{2})^{2}+2\ol{\rho}'\rho')}
= e^{-\frac{W^2}2 ( (a'-\lambda_{1})^{2}+ (b'-\lambda_{2})^{2})}(1-W^{2}\ol{\rho }'\rho')
\]
Inserting all this in the integral we obtain
\begin{align*}
e^{V (\lambda)} (\mf{T}F) (\hat{\lambda})&=
W^{2}\int \tfrac{\ud a' \ud b'}{2\pi}  e^{-\frac{W^2}2 ( (a'-\lambda_{1})^{2}+ (b'-\lambda_{2})^{2})}f (a',b')\\
&- \int \frac{\ud a' \ud b'}{2\pi}  e^{-\frac{W^2}2 ( (a'-\lambda_{1})^{2}+ (b'-\lambda_{2})^{2})}
\tfrac{1}{a'-ib'}  (\partial_{a'}-i\partial_{b'})f\big(a',b')\\
&=I_{1}-I_{2}
\end{align*}
The second integral can be reorganized as follows
\begin{align*}
I_{2}&=\int \tfrac{\ud a' \ud b'}{2\pi}\tfrac{1}{a'-ib'} (\partial_{a'}-i\partial_{b'})
\left[e^{-\frac{W^2}2 ( (a'-\lambda_{1})^{2}+ (b'-\lambda_{2})^{2})} f\big(a',b') \right]\\
& + W^{2}\int \tfrac{\ud a' \ud b'}{2\pi}\tfrac{1}{a'-ib'}
[(a'-ib')- (\lambda_{1}-i\lambda_{2})]
\left[e^{-\frac{W^2}2 ( (a'-\lambda_{1})^{2}+ (b'-\lambda_{2})^{2})} f\big(a',b') \right]\\
 &= - f (0,0)e^{-\frac{W^2}2 ( \lambda_{1}^{2}+ \lambda_{2}^{2})} + I_{1} - \left(\Lambda \delta_{W^2}^\star \Lambda^{-1}f\right) (\lambda).
\end{align*}
In the first line we use Stokes theorem in the form of the Cauchy--Pompeiu [= Cauchy--Green] formula
\begin{equation}\label{eq:stokes}
\begin{split}
\int_{|\zeta| \leq r} \frac{da \, db}{a-ib}  2 (\bar\partial g)(a,b) = - 2\pi g(0)+ \oint_{|\zeta|=r}  \frac{d\zeta}{\zeta} g(\zeta,\bar\zeta )~, \quad \zeta = a+ib~, \,\, \bar\partial = \frac{\partial_a - i \partial_b}{2}.
\end{split}
\end{equation}
where the first term on the right hand side vanishes in the limit $r \to \infty$ as long as our test function $f (\lambda)$ does not
increase too fast. This concludes the proof of the first representation for $\mc{T}.$ The second one follows from
\[
I_{1}=f (0,0)+ W^{2}\int \tfrac{\ud a' \ud b'}{2\pi}  e^{-\frac{W^2}2 ( (a'-\lambda_{1})^{2}+ (b'-\lambda_{2})^{2})}[f (a',b')-f (0,0)]=f (0,0)+\tilde{I}_{1}
\]
and $ (\partial_{a'}-i\partial_{b'})f\big(a',b')= (\partial_{a'}-i\partial_{b'})[f\big(a',b')-f (0,0)].$ 
\end{proof}

Recall that $G[H_N] (E_{\eps}):= (E+i\eps-H)^{-1}$ and that $J_{xy}$ are the matrix elements of $J = -W^2 \Delta_N + \mathbbm{1}$, and set
\begin{equation}
\begin{split}
&
I_{x, y_{1},\dotsc ,y_{q}}^{m,n_{1},\dotsc ,n_{q}}:=
\int \tfrac{\ud\lambda_1\ud\lambda_2}{2\pi }\, 
 \big[ \mc T^{x-y_{m}}\Lambda^{n_{m}}\dotsb \mc T^{y_{2}-y_{1}}\Lambda^{n_{1}}  \mc T^{y_{1}-1}e^{-V}](\lambda )\\
&\qquad \times [e^{V(\lambda )} \,\Lambda^{-1}] \  \big[\mc T^{y_{m+1}-x} \Lambda^{n_{m+1}} \mc T^{y_{m+2}-y_{m+1}} \dotsb \Lambda^{n_{q}} 
\mc T^{N-y_{q}}e^{-V}](\lambda )\end{split}
\end{equation}
where $q,m,n_{i},x,y_{i}$ are defined as in \eqref{eq:defTmultiple}.
Also set, for $x\leq x'$, 
\[
I_{xx'}:= 
\frac{1}{2\pi }
\int \frac{\ud\lambda_1\ud\lambda_2}{\lambda_1-i\lambda_2}\, e^{V(\lambda)} 
\big[\mc T^{x-1}e^{-V}\big ](\lambda) \, \big[\mc{T}^{x'-x} \Lambda \mc T^{N-x'}e^{-V}\big ] (\lambda ). \]
and for $x' < x$ set $I_{xx'} = I_{N-x+1,N-x'+1}$.
As we see from the proof below, these definitions are consistent with \eqref{eq:defTmultiple} and  \eqref{eq:transfa}.
\begin{cor}\label{corfinalrep} 
For $y,y' = 1, \cdots, N$, we have the following representation.
\begin{align}\label{finalrep}
&\lim_{\eps\searrow 0}\big\langle G_{yy}[H_N](E_\eps)\big\rangle = \mc E  +\sum_{x=1}^{N}J_{yx}
\frac{1}{2\pi }\int  \frac{\ud\lambda_1\ud\lambda_2}{\lambda_1-i\lambda_2}\, e^{V(\lambda)}
\big[\mc T^{x-1}e^{-V}](\lambda)\, \big[\mc T  ^{N-x}e^{-V}\big ](\lambda) \\
&\lim_{\eps\searrow 0}\langle G_{yy'}[H_N](E_\eps)G_{y'y}[H_N](E_\eps) \rangle =
-J_{yy'} +\hspace{-0,2cm}\sum_{xx'=1}^{N}\hspace{-0,2cm} J_{yx} J_{y'x'}  I_{xx'}
\end{align}
For $n>1$ we have:
\begin{equation}
\begin{split}
\lim_{\eps\searrow 0}\partial_{E}^{n}\big\langle \tr G[H_N](E_\eps)\big\rangle &= (-1)^{n}
 \sum_{q=1}^{n} 
\sum_{\substack{n_{1},\dotsc n_{q}\geq 1 \\  n_{1}+\dotsb+ n_{q}=n}} \tfrac{n!}{n_{1}!\cdots n_{q}!}\sum_{m=0}^{n}
\sum_{\substack{1\leq y_{1}< y_{2}< \dotsb  < y_{q}\leq N\\
   y_{m}\leq x<y_{m+1}}}
\hspace{-0.5cm}I_{x, y_{1},\dotsc ,y_{q}}^{m,n_{1},\dotsc ,n_{q}}
\end{split}
\end{equation}

\end{cor}
\begin{proof}
From  Corollary~\ref{coralg} and relations \eqref{eq:transfa} the  proof of the first two identities is reduced 
to study
\[
I_{i}:=\int   \ud a\ud b \ud \ol{\rho }\ud \rho\  a\  f_{i}(\lambda ),\quad i=1,2
\]
with
\begin{align*}
f_{1} (\lambda)&:= e^{V (\lambda )} \big[\mc T^{x-1}e^{-V}\big ](\lambda)\ \big[\mc T  ^{N-x}e^{-V}\big ] (\lambda),\\
f_{2} (\lambda)&:= e^{V (\lambda )} \big[\mc T^{x-1}e^{-V}](\lambda)\ \big[\mc{T}^{x'-x} \Lambda \mc T^{N-x'}e^{-V}\big ] (\lambda )
\end{align*}
Then
\begin{align*}
I_{i}&:= \int   \ud a\ud b \ud \ol{\rho }\ud \rho\  a\
f_{i}\big(a+ \tfrac{\ol\rho\rho}{a-ib},b-i \tfrac{\ol\rho\rho}{a-ib}\big)\\
& = \int \ud a\ud b \ud \ol{\rho }\ud \rho \ a\,
\big[ f_{i} (a,b)+ \tfrac{\ol\rho\rho}{a-ib}  (\partial_{a}-i\partial_{b})f_{i}(a,b)\big]\\
&=  -\int   \tfrac{\ud a\ud b }{2\pi }\  a\,
\tfrac{1}{a-ib}  (\partial_{a}-i\partial_{b})f_{i}(a,b)=\int   \tfrac{\ud a\ud b }{2\pi }\ 
\tfrac{1}{a-ib}  f_{i}(a,b)
\end{align*}
where in the last line we applied again Stokes theorem \eqref{eq:stokes} and we used $af_{i} (a,b)=0$ for $a=b=0.$
The proof for the third identity is done in the same way.
\end{proof}

\section{Analysis of the transfer operator}\label{ana}

So far we treated $\mf T$ and $\mc T$ as formal operations. To apply spectral-theoretic
methods,
recall that for $f$ vanishing at the origin, $\mc T f = Tf$,
and hence 
\begin{equation}\label{eq:tr1}
\mc T^n f = T^n f = \Lambda (e^{-V} \delta_{W^2}^\star)^n \Lambda^{-1}f
=  \Lambda e^{-\half V} (e^{-\half V} \delta_{W^2}^\star e^{-\half V})^n e^{\half V} \Lambda^{-1}f~.
\end{equation}
It is safe to use expressions such as $e^{-\frac{V}{2}}$ and $V$, since $e^{-V}$ does not
vanish except if $E=0$ at $\lambda_2=1 $ (cf.\eqref{eq:Vdef}).
We will show in Sect.~\ref{sect:E=0} below this problem  is easy to deal with.
For $f$ sufficiently regular (but not necessarily vanishing at zero) the computation of 
 $\mc T^n f$  will be reduced to the following two  ingredients.

\textit{(i)} We  show in Section~\ref{s:opnorm} that the operator 
$e^{-\half V} \delta_{W^2}^\star e^{-\half V}$ is bounded in $L_p(\mathbb{R}^2)$ 
for any $p \in [1, \infty]$, and spectral theory yields
a good bound on its norm for $p\in (1,\infty).$

\textit{(ii)} In Section~\ref{seceigenfn}, we construct a solution $u$ of the eigenvalue equation
$\mc T u = u$ satifying in addition $u (0)=1.$ Unlike Section~\ref{s:opnorm}, the argument relies on explicit elementary computations
rather than on methods of operator theory, and the term `eigenfunction' is used
without specifying the construction of the operator.\footnote{A proper operator-theoretic
meaning can be given by constructing the operator
in an appropriate weighted Hilbert space; this, however, is not required for our
argument.}

We recover then $\mc T^n f$ from \textit{(i)} and \textit{(ii)} via   the decomposition
 $f=f (0)u+[f-f (0)u]$ as follows: 
\[
\mc T^n f = f(0) u +  \Lambda e^{-\half V} (e^{-\half V} \delta_{W^2}^\star e^{-\half V})^n e^{\half V} \Lambda^{-1} (f - f(0) u)~.
\]
Note that $u (0)=1$ ensures $(f - f(0) u) (0)=0.$
The results of this section mostly pertain to the case $|E| \leq \sqrt{\frac{32}9}$.
The extension to all $|E|<2$ is discussed in Section~\ref{S:deform}.

\subsection{Operator norm bound}\label{s:opnorm}

Denote by $\Vert \cdot \Vert_p $ the operator norm of an operator from   $L_p(\mathbb{R}^2) = \left\{ f: \mathbb R^2 \to \mathbb C\, \mid \, \|f\|_p < \infty\right\}$ to itself.

We recall for further use that, for an integral operator with kernel $K$,
\begin{equation}\label{eq:opnorms}
\|K\|_\infty = \sup_\lambda \int d\lambda_1' d\lambda_2' |K(\lambda, \lambda')|~, \quad
\|K\|_1 = \sup_{\lambda'} \int d\lambda_1 d\lambda_2 |K(\lambda, \lambda')|
\end{equation}
and that for all $1 < p < \infty$
\begin{equation}\label{eq:interp1}
\|K\|_p \leq \|K\|_1^{\frac{1}{p}} \|K\|_\infty^{1- \frac{1}{p}}~.
\end{equation}
The combination of the relations (\ref{eq:opnorms}) and (\ref{eq:interp1}) is known as Schur's bound.
It immediately follows from \eqref{eq:defdeltaW} and \eqref{eq:interp1}  that $\Vert \delta_{W^2}^\star\Vert_p \leq 1$ for any $p\in [1,\infty]$. It is also easy to check that, for $ \vert E\vert\leq \sqrt{\frac{32}{9}}$\footnote{At $E=\sqrt{\frac {32}9} $, $\Re V $ develops two new local minima. The value at these minima is positive until $\vert E\vert = 1.893\dots> \sqrt{\frac {32}9} = 1.885\dots$ (no analytical expression available), and the results of \cite{shcherbina2016transfer} as well as
the argument of this section actually hold until this threshold.}, $\Re V(\lambda) $ has non-degenerate global minima at $\lambda = 0$, and $ \lambda = (0,2\sqrt{1-\frac{E^2}4}) $, with value $0$. In particular, $\Re V\geq 0 $, so that 
\begin{equation}\label{eq:normleq1}
\forall |E| \leq \sqrt\frac{32}9 \,\,\forall p\in [1,\infty]   \quad \Vert e^{-V}\Vert_p\leq 1~,
\end{equation} 
where $  \Vert e^{-V}\Vert_p$ means the operator norm. In particular,
\begin{equation}\label{eq:normleq1'}
\forall |E| \leq \sqrt\frac{32}9 \,\,\forall p\in [1,\infty]   \quad \Vert \delta_{W^2}^\star e^{-V}\Vert_p~,
\Vert e^{-V} \delta_{W^2}^\star  \Vert_p~, \Vert e^{-V/2} \delta_{W^2}^\star e^{-V/2}\Vert_p \leq 1~.
\end{equation} 
The following standard semiclassical argument improves these bounds for $1<p<\infty $.
\begin{prop}\label{specbound} For any $p \in (1,\infty) $ there exists 
$c > 0$ such that
\[ \forall |E| \leq  \sqrt\frac{32}9 \,\,\,\,\,
 \Vert e^{-\half V}\delta_{W^2}^\star e^{-\half V}\Vert_p\leq 1-\frac cW~. \]
\end{prop}
\begin{rem} The estimate fails for $p=1,\infty$: indeed, 
\[ \begin{split}
&\Vert e^{-\half V}\delta_{W^2}^\star e^{-\half V}\Vert_1 = \Vert e^{-\half V}\delta_{W^2}^\star e^{-\half V}\Vert_\infty \\
&\quad = \sup_\lambda \int d\lambda_1' d\lambda_2'  \,
e^{-\half \Re V(\lambda)} \frac{W^2}{2\pi} e^{-\frac12 W^2 (\lambda - \lambda')^2}  e^{-\half \Re V(\lambda')} = 1 - \mc O(W^{-2})~. \end{split}\]
\end{rem}

\begin{rem}\label{rem:normbounded} The proposition implies  that $\Vert (e^{-V}\delta_{W^2}^\star)^k\Vert_p ,\,\Vert (\delta_{W^2}^\star e^{-V})^k\Vert_p \leq (1-c/W)^{k-1} $ for $p\in(1,\infty)$.
\end{rem}
\begin{proof}
We prove the estimate for $p=2$, in which case this is a standard semiclassical argument (see \cite{Helfferbook}), which we repeat for the convenience of the reader. The general case  follows from the case $p=2$ and (\ref{eq:normleq1'}) by Riesz--Thorin interpolation.

It suffices to show that $\Vert e^{-\half \Re V}\delta_{W^2}^\star e^{-\half \Re V}\Vert_2\leq 1-\frac cW$.  
Let $\chi_1,\chi_2 $ be smooth bump functions in a sufficiently small (but $W$-independent) neighbourhood of the
minima $0,(0,2\sqrt{1-\frac{E^2}4}) $ of $\Re V $, and let $\chi_3 $ be so that $\chi_1^2+\chi_2^2 + \chi_3^2=1 $. We have
\[\begin{split}
e^{-\half \Re V}\delta_{W^2}^\star e^{-\half \Re V} = \sum_{i=1,2,3} \chi_i e^{-\half \Re V}\delta_{W^2}^\star e^{-\half \Re V} \chi_i +  e^{-\half \Re V} A e^{-\half \Re V} ,
\end{split}\]
where $A$ has the kernel $\half\delta_{W^2}^\star(\lambda,\lambda')\sum_i [\chi_i(\lambda) - \chi_i(\lambda')]^2 $
and we used the relation $ \chi_i(\lambda)^{2} + \chi_i(\lambda')^{2}= [\chi_i(\lambda) - \chi_i(\lambda')]^2+2 \chi_i(\lambda)  \chi_i(\lambda'). $ Therefore, using $\Vert e^{-\half \Re V}\Vert_2\leq 1 $,
\[\begin{split}
\Vert  e^{-\half \Re V}\delta_{W^2}^\star e^{-\half \Re V}\Vert_2 &\leq\left\Vert   \sum_{i=1,2,3}\chi_ie^{-\half \Re V}\delta_{W^2}^\star e^{-\half \Re V}\chi_i\right\Vert_2 + \Vert   A  \Vert_2
\end{split}\]
From $\sum_i [\chi_i(\lambda) - \chi_i(\lambda')]^2 = (\lambda-\lambda')^2 \tilde\chi(\lambda,\lambda') $ with bounded $\tilde \chi $, it is easy to show with the Schur bound  \eqref{eq:opnorms}  \eqref{eq:interp1}  that $\Vert A\Vert_2 = \mc O(W^{-2}) $. Further,
\[\begin{split}
\left\Vert   \sum_{i=1,2,3}\chi_ie^{-\half \Re V}\delta_{W^2}^\star e^{-\half \Re V}\chi_i\right\Vert_2 
&= \sup_{\Vert \phi\Vert_2=1} \sum_{i=1,2,3} (\chi_i\phi,\mathbbm{1}_{\supp\chi_i}e^{-\half \Re V}\delta_{W^2}^\star e^{-\half \Re V}\mathbbm{1}_{\supp\chi_i} \chi_i\phi)\\
&\hspace{-1cm}\leq  \max_i \Vert \mathbbm{1}_{\supp\chi_i}e^{-\half \Re V}\delta_{W^2}^\star e^{-\half \Re V}\mathbbm{1}_{\supp\chi_i}\Vert_2\sup_{\Vert \phi\Vert_2=1} \sum_{i=1,2,3} \Vert \chi_i\phi\Vert_2^2 \\
&\hspace{-1cm}= \max_i \Vert \mathbbm{1}_{\supp\chi_i}e^{-\half \Re V}\delta_{W^2}^\star e^{-\half \Re V}\mathbbm{1}_{\supp\chi_i}\Vert_2.
\end{split}\]
For $i=3 $, since $  \Re V\vert_{\supp\chi_3}>0 $, we have $\Vert \mathbbm{1}_{\supp\chi_3}e^{-\half \Re V}\delta_{W^2}^\star e^{-\half \Re V}\mathbbm{1}_{\supp\chi_3}\Vert_2 \leq c <1 $, uniformly in $W$. For $ i=1,2$, since the minima of $\Re V $ are non-degenerate,
$$\Vert \mathbbm{1}_{\supp\chi_i}e^{-\half \Re V}\delta_{W^2}^\star e^{-\half \Re V}\mathbbm{1}_{\supp\chi_i} \Vert_2\leq
\Vert e^{-c'\lambda^2}\delta_{W^2}^\star e^{-c'\lambda^2}\Vert_2 $$
for some $c'>0 $. For self adjoint operators, the $L^2$ operator norm is equal to the spectral radius, and the spectrum of the harmonic Kac operator $e^{-c'\lambda^2}\delta_{W^2}^\star e^{-c'\lambda^2} $ can be computed explicitly, giving $\Vert e^{-c'\lambda^2}\delta_{W^2}^\star e^{-c'\lambda^2}\Vert_2  \leq 1-\frac{c''}W$. The claim follows for $p=2$. \end{proof}
The following corollary will be a key ingredient in the proof of our theorems.
\begin{cor}\label{cor:specbound}
Let $f\in L_{p}$ with $1<p<\infty$ be such that $\Lambda^{-1} f\in L_{p}.$ Then  we have for all $n\geq 1,$ $m\geq 0$
\begin{equation}\label{cor:Tnormest}
\begin{split}
\|\Lambda^{m}T^{n}f\|_{p}&\leq  C^{m+1}\sqrt{(m+1)!}\ 
\Big(1-\frac cW\Big)^{n-1}\Vert \Lambda^{-1}f\Vert_p,\\
\|\Lambda^{-1}T^{n}f\|_{p}&\leq   
\Big(1-\frac cW\Big)^{n-1}\Vert \Lambda^{-1}f\Vert_p.\\
\end{split}
\end{equation}
\end{cor}
\begin{proof}
For the first bound we use 
\begin{equation}
\begin{split}
&\Vert \Lambda^{m}T^n f \Vert_p =
\Vert  \Lambda^{m+1} e^{-\half V} (e^{-\half V} \delta_{W^2}^\star e^{-\half V})^{n-1}
e^{-\half V} \delta_{W^2}^\star \Lambda^{-1} f \Vert_p\\
&\quad  \leq \Vert \Lambda^{m+1} e^{-\frac{1}{2}V}\Vert_\infty  \Vert  (e^{-\half V} \delta_{W^2}^\star e^{-\half V})^{n-1}
e^{-\half V} \delta_{W^2}^\star \Lambda^{-1} f \Vert_p\\
&\quad  \leq   \Vert \Lambda^{m+1} e^{-\frac{1}{2}V}\Vert_\infty
\Big(1-\tfrac cW\Big)^{n-1}\Vert \Lambda^{-1}f\Vert_p,
\end{split}
\end{equation}
where in the first line we applied equation \eqref{eq:tr1} and in the second line equation \eqref{eq:normleq1'} and
Proposition~\ref{specbound}.
For the second bound we use
\begin{equation}
\begin{split}
&\Vert \Lambda^{-1}T^n f \Vert_p =
\Vert e^{-\half V} (e^{-\half V} \delta_{W^2}^\star e^{-\half V})^{n-1}
e^{-\half V} \delta_{W^2}^\star \Lambda^{-1} f \Vert_p\\
&\quad  \leq \Vert  e^{-\frac{1}{2}V}\Vert_\infty  \Vert  (e^{-\half V} \delta_{W^2}^\star e^{-\half V})^{n-1}
e^{-\half V} \delta_{W^2}^\star \Lambda^{-1} f \Vert_p \leq 
\Big(1-\tfrac cW\Big)^{n-1}\Vert \Lambda^{-1}f\Vert_p.
\end{split}
\end{equation}
The bound $\| \Lambda^{m+1} e^{-V}\|_\infty \leq C^{m+1} \sqrt{(m+1)!}$ follows from the quadratic growth of $\Re V$ at infinity.
\end{proof}
We conclude this subsection with a few additional properties of the $T$ operator that will prove
important later.

\begin{lemma}\label{le:additionalprop} The following holds.

\begin{itemize}
\item [(i)] Let $p\in (1,\infty), n\in \mathbb{N}$  and $1\leq q\leq p$ such that $f$ and $\Lambda^{n}f\in L_{p}$ and $\Lambda^{-1}f\in L_{q}.$
Then $Tf\in C^{\infty} (\mathbb{R}^{2})$ and there exists a  constant $C=C_{p,q}>0$ such that
\begin{equation}\label{eq:additionalprop}
\Vert e^{V}\Lambda^{n}T f \Vert_p\leq C^{n+1} \left[ \Vert \Lambda^{n}f \Vert_p+  \sqrt{(n+1)!}\ W^{-1-n} W^{\frac{2(p-q)}{pq}} 
\Vert \Lambda^{-1}f \Vert_q \right].
\end{equation}

\item [(ii)] Let $n\in \mathbb{N}$ and  $f\in L_{1}\cap L_{\infty}$ be a $C^{1}$ function such that $\Lambda^{n}f\in L_{1}\cap L_{\infty}.$ Then $\mc{T}f$ is smooth,  and for all
$p\in (1,\infty)$ $e^{V}\mc{T}f\in L_{p},$ and satisfies the bound
\[
\Vert e^{V}\Lambda^{n}\mc{T}f \Vert_p \leq C^{n+1} \left[ \Vert \Lambda^{n}f \Vert_p+  \sqrt{(n+1)!}\left\{ W^{1-\frac{2}{p}-n} 
\Vert \Lambda^{-1}f \Vert_1  +   W^{-\frac{2}{p}-n} |f(0)| \right\} \right] \]
\end{itemize}
\end{lemma}

\begin{proof}
Recall that $Tf= e^{-V}\Lambda\delta_{W^2}^\star \Lambda^{-1} f $ where $\delta_{W^2}^\star $ is the convolution with the
heat kernel in dimension $d=2$ at time $t=\frac{1}{2W^{2}}$, and hence $e^V \Lambda^n T f = \Lambda^{n+1} \delta_{W^2}^\star \frac{1}{\Lambda} f$. Smoothness follows from the regularizing effect of the
heat kernel.

\textit{Proof of (i)} We drag the multiplier $\Lambda^{n+1}$ through $\delta_{W^2}^\star$, as follows. 
The decomposition $|\Lambda^{n+1}| \leq 2^{n} (  |\Lambda -\Lambda'|^{n+1}+ |\Lambda'|^{n+1})$ yields
\[
\begin{split}
2^{-n}|e^{V}\Lambda^{n}T f (\lambda )|&= 2^{-n}| \Lambda^{n+1} \delta_{W^2}^\star \Lambda^{-1}f (\lambda )|\\
&\leq
(\delta_{W^2}^\star |{\Lambda}^{n}f|)(\lambda)  + \tfrac{W^{2}}{2\pi }\int d\lambda_{1}'d\lambda_{2}'
e^{-\frac{W^{2}}{2} |\Lambda-\Lambda'|^{2} } |\Lambda-\Lambda'|^{n+1}|\Lambda'|^{-1} |f (\lambda')|\\
& \leq   (\delta_{W^2}^\star |\Lambda^{n}f|) (\lambda )
 + C^{n+1}\ \sqrt{(n+1)!}\ W^{-1-n}   \delta_{W^2/2}^\star (|\Lambda^{-1} f|) (\lambda ),
\end{split}
\]
where in the last term we extracted a fraction of the exponential decay to bound the factor $|\Lambda-\Lambda'|.$
The result $(i)$ now  follows from
\begin{equation}\label{eq:LpL1bound'}
\| \delta_{W^2}^\star \phi \|_{p}\leq  C^{p-q}\left ( \tfrac{W^{2}}{2\pi }\right )^{\frac{p-q}{pq}} \|\phi \|_{q},\qquad 
\forall p\in (1,\infty),\qquad \forall 1\leq q\leq p.
\end{equation}

\textit{Proof of (ii)} From Definition~\eqref{eq:defTf} $\Vert e^{V}\Lambda^{n}\mc{T}f \Vert_p \leq
\Vert \Lambda^{n}e^{-\frac{W^{2}}{2}\lambda^{2}} \Vert_p |f(0)|+\Vert e^{V}\Lambda^{n}Tf \Vert_p.$
The first term is bounded by $\Vert \Lambda^{n}e^{-\frac{W^{2}}{2}\lambda^{2}} \Vert_p |f(0)| \leq C^{n+1} \sqrt{(n+1)!} W^{-n-\frac{2}{p}} |f(0)|.$
The result now follows from \eqref{eq:additionalprop} setting $q=1$. The condition $\Lambda^{-1}f\in L_{1}$ is ensured by
$f\in C_{1}\cap L_{1}$ and $f (0)=1.$ Finally $f,\Lambda^{n}f\in L_{1}\cap L_{\infty}$ ensures $f$ and $\Lambda^{n}f\in L_{p}$ for all $p\in (1,\infty).$
\end{proof}

\begin{rem} Note that while the function $e^{-V}$ is bounded,  the function $e^{V}$ develops a singularity of the form $1/(\lambda_2 - 1)$ at $\lambda_{2}=1$ for $E=0$.
The lemma above shows this causes no problem, as long as the function appears together with the operator $T.$
\end{rem}

\subsection{The top eigenfunction}\label{seceigenfn}

Using the bound of the previous paragraph, we will  construct a solution $u$ of the equation
$\mc T u = u$ normalized to $u (0)=1.$ We call this solution an eigenfunction (with eigenvalue $1$), without
specifying the operator-theoretic construction. 

The strategy is as follows. First, we guess an approximate
eigenfunction $u_0$ such that $\mc T u_0 \approx u_0$. This is done below in
Proposition~\ref{prop:approx}. Then we upgrade it to the true eigenfunction
\begin{equation}\label{eq:correctef1}
u = u_0 + \sum_{n=0}^\infty \mc T^n (\mc T u_0 - u_0)~. \end{equation}
The next proposition ensures that this procedure is justified, and yields an eigenfunction
that  is close to $u_0$.
\begin{prop}\label{recon}
Let    $u_0\in L_{1}\cap L_{\infty}$ be a $C^{1}$ function such that 
$u_0(0) =1.$ We define
\begin{equation}\label{eq:vdef}
v := \mc T u_0 - u_0. 
\end{equation}
 Then for all $p\in (1,\infty)$ it holds $v\in L_{p}$ and $\Lambda^{-1}v\in L_{p}.$ Moreover, the series
\begin{equation}\label{eq:sum}
u = u_0 + \sum_{n=0}^\infty \mc T^n v
\end{equation}
converges in $L_p$ to a solution $u$ of the equation $\mc T u = u$ which satisfies
\begin{equation}\label{eq:eigenfrecon}
\Vert u - u_0 \Vert_p\leq \Vert v\Vert_p+ \mc O(W) \Vert \Lambda^{-1}v\Vert_p~,\quad
\Vert \Lambda^{-1}(u - u_0)\Vert_p\leq \mc O(W) \Vert \Lambda^{-1}v\Vert_p
\end{equation}
and
\begin{equation}\label{eq:regudecayu}
\begin{split}
\Vert e^{V}\Lambda^{n} (u-\mc{T}u_{0}) \Vert_p &\leq C^{n+1}\Big(  \Vert \Lambda^{n} (u-u_{0}) \Vert_p
+ \frac{\sqrt{(n+1)!}}{W^{1+n}}\Vert \Lambda^{-1}( u-u_{0}) \Vert_p \Big)\\
\Vert \Lambda^{n} (u-\mc{T}u_{0}) \Vert_p &\leq 
 C^{n+1}\ \sqrt{(n+1)!}\ \Vert \Lambda^{-1}( u-u_{0})  \Vert_p.
\end{split}
\end{equation}
The limit function $u$ is independent of the initial choice for $u_{0}.$
\end{prop}
\begin{proof}
Inserting the definition \eqref{eq:defTf} of $\mc{T}$ and using $u_{0} (0)=1$ we have
\[
v=  \mc T u_0 - u_0 = e^{-V-\frac {W^2}2\lambda^2}-u_{0} +Tu_0.
\]
From $u_{0}\in C^{1}\cap L_{1}$ and $u_{0} (0)=1$ follows $\Lambda^{-1}u_{0}\in L_{1}$ and 
\begin{equation}\label{eq:LpL1bound}
\| \delta_{W^2}^\star \Lambda^{-1}u_{0}\|_{p}\leq \left ( \tfrac{W^{2}}{2\pi }\right )^{\frac{p-1}{p}} \|\Lambda^{-1}u_{0}\|_{1},\qquad 
\forall p\in (1,\infty),\end{equation}
whence 
$Tu_{0}\in L_{p}$ and $\Lambda^{-1}Tu_{0}\in L_{p}.$ Moreover $Tu_{0}\in C^{\infty},$ hence $v$ is continuous and $v\in L_{p}.$ Finally
$u_{0}\in C^{1}\cap L_{1}$ and $e^{-V (0)}=1=u_{0} (0)$ yields  
 $\Lambda^{-1}[e^{-V-\frac {W^2}2\lambda^2}-u_{0} ]\in L_{p},$  hence $\Vert \Lambda^{-1} v\Vert_p<\infty $.
Since $(\mc{T}u_{0}) (0)=u_{0} (0)$ we have $v (0)=0$ 
 hence $\mc{T}v=Tv$ and for all $ n\geq 1$ it holds
\begin{equation} \label{estim}
\Vert \mc T^n v \Vert_p = \Vert T^n v \Vert_p\leq   \Vert \Lambda e^{-\frac{1}{2}V}\Vert_\infty
\Big(1-\frac cW\Big)^{n-1}\Vert \Lambda^{-1}v\Vert_p,
\end{equation}
where we applied equation \eqref{eq:tr1} and  estimate \eqref{cor:Tnormest}.
This inequality implies convergence of the geometric series and 
the relations \eqref{eq:eigenfrecon}.
Since  $(\mc{T}^{n}v) (0)=v (0)=0$ we have  $u (0)=1$ and $u (\lambda)=\lim_{N\to \infty}\mc{T}^{N+1}u_{0},$
hence $\mc{T}u=u.$ Finally, for any two initial functions $u_{0},\tilde{u_{0}},$  $(u_{0}-\tilde{u}_{0}) (0)=0$
implies $\mc{T}^{n}(u_{0}-\tilde{u}_{0})=T^{n}(u_{0}-\tilde{u}_{0}).$ 

The estimates \ref{eq:regudecayu}) follow from  \eqref{eq:additionalprop} together with the relation
\begin{equation}\label{eq:newid}
u = \mc{T}u_{0}+ T (u-u_{0}),
\end{equation}
which is obtained from \eqref{eq:correctef1}.

The independence of the limit of $u_0$ follows from
\eqref{estim} again.
\end{proof}

\subsubsection{Approximate eigenfunction}
We need to make a good choice for $u_0$ in order for $u - u_0$ to be relatively small.
Since we expect $u$ to be also eigenfunction of the supersymmetric operator $\mf{T},$
we take as initial ansatz some function $U_{0} (R)$ invariant under superunitary rotations
hence $U_{0} (R)= u_{0} (\lambda).$ Guided from equation \eqref{eq:gaussfunc} and
Corollary \ref{cor:exactformulas} and in analogy with
semiclassical analysis we take as ansatzs
\begin{equation*}
U_{0}^{(0)} (R):= e^{-\alpha W\str R^{2}},\qquad  U_{0}^{(M)} (R):= e^{-\alpha W\str R^{2}} \Big (1+\sum_{j=1}^{M}WQ^{(j)} (R)\Big), \ M\geq 1,
 \end{equation*}
where $Q^{(j)}$ is a polynomial in $R$ of degree $j,$ consisting of sums of products of supertraces,
which we assume invariant under superunitary rotations $Q^{(j)} (R)=q^{(j)} (\lambda).$
Finally $\alpha\in \mathbb{C}$ is some constant. Then from \eqref{eq:trans}
\[
\begin{split}
(\mf{T} ( U_{0}^{(M)})) (R) &= e^{-V (R)} e^{-\alpha\mu  W\str R^{2}} \Big (1+\sum_{j=1}^{M}W\tilde{Q}^{(j)} (R)\Big )\\
& =  e^{-\alpha W\str R^{2}} \Big [ e^{\alpha (1-\mu)  W\str R^{2}-V (R)} \Big (1+\sum_{j=1}^{M} W\tilde{Q}^{(j)} (R)\Big )   \Big],
\end{split}
\]
where $\tilde{Q}^{(j)} (R)=  \int d\mu_{\frac{W^{2}}{\mu }} (R') Q^{(j)} (R'+ \mu R).$
We introduce the rotation invariant error function
\begin{equation}\label{eq:error}
\begin{split}
\mathrm{Err}^{(M)} (R)&:= e^{\alpha  W\str R^{2}} [\mf{T} ( U_{0}^{(M)})- U_{0}^{(M)}  ] (R)\\
=&\sum_{j=1}^{M} W  \Big ( \tilde{Q}^{(j)} (R)-Q^{(j)} (R)\Big)+
\Big[e^{\alpha (1-\mu)  W\str R^{2}-V (R)} -1\Big ]  \Big (1+\sum_{j=1}^{M} W\tilde{Q}^{(j)} (R)\Big)
\end{split}
\end{equation}
By abuse of notation we will use the same letter $\mathrm{Err}^{(M)} (\cdot)$ to denote the function of $R$
and the function of $\lambda.$ Then
\begin{equation}\label{eq:errorbis}
v^{(M)} (\lambda )= \mc{T}u_{0}^{(M)}-u_{0}^{(M)} (\lambda )=
 e^{-\alpha  W\lambda^{2}} \mathrm{Err}^{(M)} (\lambda )
\end{equation}
Hence
\[
\|v^{(M)}\|_{p}^{p}= \int d\lambda |v^{(M)} (\lambda ) |^{p}= \frac{1}{W}  \int d\lambda
\left| e^{-\alpha  \lambda^{2}}   \mathrm{Err}^{(M)} (\lambda/\sqrt{W} ) \right|^{p}.
\]
The precision of our approximation is therefore determined by the leading order in $W^{-\frac{1}{2}}$
from $\mathrm{Err} (R/\sqrt{W}).$ The following proposition shows we can make this error function at
least as small as  $O (W^{-3}).$
\begin{prop}\label{prop:approx}
Let $\alpha $ be the solution  of the equation $ \alpha^{2}= \frac{1}{4} (1-\mathcal{E}^{2})$ with $\Re \alpha >0.$
Set $Q^{(1)}=Q^{(2)}=0$ and
\begin{equation}
\begin{split}
Q^{(3)} (R)&:= c_{3} \str R^{3}\\
Q^{(4)} (R ) &:= c_{4}^{0}\str R^{4}+  c_{4}^{1} W^{-1}\str R^{2} +c_{4}^{2} W^{-1} (\str R)^{2}+
 c_{4}^{3} W (\str R^{3})^{2},\\
Q^{(5)} (R)&:=  c_{5}^{0} \str R^{5} + c_{5}^{1} \str R^{3} \str R^{2}+ c_{5}^{2} \str R^{3} (\str R)^{2} \\
&+ c_{5}^{3} W^{-1} \str R\str R^{2}+ c_{5}^{4}W^{-1} \str R^{3} + c_{5}^{5}W^{-2} \str R  \\
&+  c_{5}^{6} W \str R^{3}\str R^{4}+  c_{5}^{7} W^{2} (\str R^{3})^{3} ]
\end{split}
\end{equation}
where
$c_{3},c_{4}^{j},$ $j=0,\dotsc ,3$ and $c_{5}^{j}$ $j=0,\dotsc ,7$ form the solution of the following system of equations
\begin{equation}\label{eq:contants}
\begin{split}
&
3 (2\alpha ) c_{3}= \tfrac{\mathcal{E}^{3}}{3},\ \   
4(2\alpha ) c_{4}^{0}  -9c_{4}^{3}= \tfrac{\mathcal{E}^{4}}{4},\ \  c_{4}^{1} = -\alpha^{2},\ \ (2\alpha ) c_{4}^{2}-c_{4}^{0}=0,
 \ \   6 (2\alpha) c_{4}^{3}= \tfrac{\mathcal{E}^{3}}{3} c_{3}\\
& 5 (2\alpha ) c_{5}^{0} =  \tfrac{\mathcal{E}^{5}}{5}+12 c_{5}^{6},\   \  5(2\alpha )c_{5}^{1} = \tfrac{\mathcal{E}^{3}}{3}c_{4}^{1} ,\ \
5 (2\alpha) c_{5}^{2}-2c_{5}^{6}= \tfrac{\mathcal{E}^{3}}{3}c_{4}^{2} ,\\
&  3(2\alpha )  c_{5}^{3}-5c_{5}^{0}-6 c_{5}^{1}=0,\
(2\alpha )c_{5}^{4}-2c_{5}^{2}= 2 (2\alpha )^{2} c_{3},\  \
(2\alpha )c_{5}^{5}-2c_{5}^{3}=0,
\\  
& 7(2\alpha )c_{5}^{6}-3^{3} c_{5}^{7}  = \tfrac{\mathcal{E}^{3}}{3}c_{4}^{0},\ \  9(2\alpha )c_{5}^{7} = \tfrac{\mathcal{E}^{3}}{3}c_{4}^{3}
\end{split}
\end{equation}
Then for all $p\in [1,\infty]$ it holds
\begin{equation}\label{eq:vbounds}
\begin{split}
&\|v^{(0)}\|_{p}=\mc{O} (W^{-\frac{1}{p}} W^{-\frac{3}{2}}),\quad  \|\Lambda^{-1}v^{(0)}\|_{p}=\mc{O} (W^{-\frac{1}{p}} W^{-1})\\
& \|v^{(M)}\|_{p}=\mc{O} (W^{-\frac{1}{p}} W^{-\frac{M+1}{2}}),\quad \|\Lambda^{-1}v^{(M)}\|_{p}=\mc{O} (W^{-\frac{1}{p}} W^{-\frac{M}{2}}),
\  M=3,4,5.
\end{split}
\end{equation}
\end{prop}
\begin{proof}
It is more convenient to study the error term in $R$ coordinates.
After rescaling $R\to W^{-\frac{1}{2}}R$ the polynomials $Q^{3},Q^{4}$ can be written as
\[
Q^{(3)} ( W^{-\frac{1}{2}}R)= \frac{1}{\sqrt{W}^{3}} P_{3} (R),\quad Q^{(4)}  ( W^{-\frac{1}{2}}R)=  \frac{1}{\sqrt{W}^{4}} P_{4} (R)
\quad Q^{(5)}  ( W^{-\frac{1}{2}}R)=  \frac{1}{\sqrt{W}^{5}} P_{5} (R)
\]
where we defined
\[
\begin{split}
& P_{3} (R) :=   c_{3}\str R^{3},\qquad P_{4} (R):= c_{4}^{1}\str R^{4}+  c_{4}^{2}\str R^{2} +c_{4}^{3} (\str R)^{2}+
 c_{4}^{4} (\str R^{3})^{2}\\
& P_{5} (R)=  c_{5}^{0} \str R^{5} + c_{5}^{1} \str R^{3} \str R^{2}+ c_{5}^{2} \str R^{3} (\str R)^{2}  + c_{5}^{3}  \str R\str R^{2}
           + c_{5}^{4} \str R^{3}\\
&\qquad  + c_{5}^{5} \str R   +  c_{5}^{6}  \str R^{3}\str R^{4}+  c_{5}^{7}  (\str R^{3})^{3}.\
\end{split}
\]
The $W$ prefactors ensure that all terms in $P_{4}$ contribute to the same order $\sqrt{W}^{-4}$ and
all terms in $P_{5}$ contribute to the same order $\sqrt{W}^{-5}.$
It follows from Corollary~\ref{cor:exactformulas} and the additional formulas in the Appendix 
\[
\tilde{Q}^{(M)} ( W^{-\frac{1}{2}}R)=\frac{1}{\sqrt{W}^{M}} \tilde{P}_{M} (R),\qquad M=3,4,5,
\]
where
\[
\begin{split}
 \tilde{P}_{3} (R) =& P_{3} (\mu R)=  \mu^{3} c_{3}\str R^{3},\\
\tilde{P}_{4} (R)=& P_{4} (\mu R)+\frac{1}{W}\ [2  c_{4}^{1} \mu^{3} (\str R)^{2}
+9 c_{4}^{4} \mu^{5}\str R^{4} ]   + \mc{O} (W^{-3}).\\
\tilde{P}_{5} (R) =& P_{5} (\mu R) + \frac{1}{W}\Big [
 5 c_{5}^{0} \str R\str R^{2} + 6  c_{5}^{1}  \str R\str R^{2}
+6 c_{5}^{2} \str R^{3} +2 c_{5}^{3} \str R \\
&  +  c_{5}^{6} [2 (\str R)^{2}\str R^{3}+12\str R^{5}]+
  c_{5}^{7} 3^{3}   \str R^{3}\str R^{4} \Big ]  + \mc{O} (W^{-2}) 
\end{split}
\]
On the other hand
\begin{equation}\label{eq:Vexpansion}
V (R)=\frac{1-\mc{E}^{2}}{2}\str R^{2}-\sum_{q\geq 3} \frac{\mc{E}^{q}}{q!}\str R^{q}= 
\frac{1-\mc{E}^{2}}{2}(\lambda_{1}^{2}+\lambda_{2}^{2}) -
\sum_{q\geq 3} \frac{\mc{E}^{q}}{q!}\big (  \lambda_{1}^{q}- (i\lambda_{2})^{q}\big )
\end{equation}
where the sum is absolutely convergent for $|\lambda|$ small. Moreover
\[
\mu = (1+\frac{2\alpha }{W})^{-1} =
\sum_{n\geq 0} (-2\alpha )^{n} W^{-n},\quad  W(1-\mu)= 2\alpha \mu= (2\alpha ) - (2\alpha )^{2}W^{-1}+\mc{O} (W^{-2}).
\]
Rescaling $R\to W^{-\frac{1}{2}}R$ and setting  $2\alpha^{2}= \frac{1}{2} (1-\mathcal{E}^{2})$  we obtain
\[
\begin{split}
& [\alpha W(1-\mu) W^{-1}\str R^{2}-V ((RW^{-\frac{1}{2}})] =
\frac{1}{W^{\frac{3}{2}}} V_{3} (R)+ \frac{1}{W^{\frac{4}{2}}} V_{4} (R)+\frac{1}{W^{\frac{5}{2}}} V_{5} (R)+ \mc{O} (W^{-3}),
\end{split}
\]
where $V_{3} (R):= \frac{\mathcal{E}^{3}}{3}\str R^{3},$ 
$ V_{4} (R):=    \frac{\mathcal{E}^{4}}{4} \str R^{4} - 4\alpha^{3} \str R^{2}, $ and
$V_{5} (R):= \frac{\mathcal{E}^{5}}{5}\str R^{5}.$ 
Hence
\[
e^{\alpha (1-\mu)\str R^{2}-V ((RW^{-\frac{1}{2}})}= 1+ W^{-\frac{3}{2}} V_{3} (R)
+ W^{-\frac{4}{2}} V_{4} (R)+W^{-\frac{5}{2}} V_{5} (R)+ \mc{O} (W^{-3}).
\]
Inserting all this in the rescaled error function yields for $M=0$
\begin{equation}\label{eq:err0bound}
\mathrm{Err}^{(0)} (RW^{-\frac{1}{2}})=  W^{-\frac{3}{2}} V_{3} (R)+ \mc{O} (W^{-2})= \mc{O} (W^{-\frac{3}{2}}).
\end{equation}
Setting $ 6\alpha c_{3}= \frac{\mathcal{E}^{3}}{3}$ we obtain for $M=3$ 
\begin{equation}\label{eq:err3bound}
\begin{split}
\mathrm{Err}^{(3)} (RW^{-\frac{1}{2}})&= W [ \tilde{Q}^{(3)} (RW^{-\frac{1}{2}})-Q^{(3)} (RW^{-\frac{1}{2}}) ]\\
& + [1+ W\tilde{Q}^{(3)} (RW^{-\frac{1}{2}})]\ [ W^{-\frac{3}{2}} V_{3} (R)
+ \mc{O} (W^{-2})]\\
&=  \frac{1}{\sqrt{W}} (\mu^{3}-1)P_{3} (R)+  \frac{1}{\sqrt{W}^{3}} V_{3} (R) + \mc{O} (W^{-2})\\
&= W^{-\frac{3}{2}} \Big (-6\alpha P_{3} (R)+  V_{3} (R)   \Big )+ \mc{O} (W^{-2})=\mc{O} (W^{-2}).
\end{split}
\end{equation}
For $M=4,$ inserting the values \eqref{eq:contants} we get
\[
\begin{split}
&\mathrm{Err}^{(4)} (RW^{-\frac{1}{2}})= W [\tilde{Q}^{(3)} (RW^{-\frac{1}{2}})Q^{(3)} (RW^{-\frac{1}{2}})+ \tilde{Q}^{(4)} (RW^{-\frac{1}{2}})-Q^{(4)} (RW^{-\frac{1}{2}}) ]\\
&\quad  + [1+ W (\tilde{Q}^{(3)} (RW^{-\frac{1}{2}})+\tilde{Q}^{(4)} (RW^{-\frac{1}{2}}))]\ [ W^{-\frac{3}{2}} V_{3} (R)+ W^{-2}V_{4} (R)
+ \mc{O} (W^{-\frac{5}{2}})]\\
&\quad =  \frac{1}{W^{2}} \Big [ (-8c_{4}^{1}\alpha  +9c_{4}^{4}+ \frac{\mathcal{E}^{4}}{4} ) \str R^{4}
+ ( - 4c_{4}^{2}\alpha  - 4\alpha^{3}) \str R^{2}\\
&\quad + ( - 4c_{4}^{3}\alpha+2c_{4}^{1}) (\str R)^{2}+ (  \frac{\mathcal{E}^{3}}{3} c_{3} - 12 \alpha c_{4}^{4}) (\str R^{3})^{2}
\Big ]+ \mc{O} (W^{-\frac{5}{2}})= \mc{O} (W^{-\frac{5}{2}}).
\end{split}
\]
Finally, the same  arguments yield  $\mathrm{Err}^{(5)} (RW^{-\frac{1}{2}})= \mc{O} (W^{-3}).$
The result follows.
\end{proof}
\begin{rem}\label{rem:allE} Note that  Proposition~\ref{prop:approx} remains valid for all 
$|E| < 2$, since it only relies on the properties of the kernel in the vicinity of the origin. 
For the same reason, the conclusion remains valid if the kernel (or the contour) is 
deformed outside a vicinity of the origin. We shall use this in  Section~\ref{S:deform}. 
\end{rem}

A first consequence of these results is the following Corollary.

\begin{cor}\label{c:bounds} Let $u^{(0)}_{0}, u^{(3)}_{0},$ $u_{0}^{(4)}$ and $u_{0}^{(5)}$ as above. Then
for all $p \in (1, \infty)$
\[
\begin{split}
&\| u - u^{(0)}_{0} \|_p \leq \mc{O} (W^{-\frac{1}{p}-\frac{1}{2}}),\qquad \|\Lambda^{-1} (u - u^{(0)}_{0}) \|_p \leq \mc{O} (W^{-\frac{1}{p}})\\
&\| u - u^{(M)}_{0} \|_p \leq \mc{O} (W^{-\frac{1}{p}-\frac{(M-1)}{2}}),\qquad \|\Lambda^{-1} (u - u^{(M)}_{0}) \|_p \leq
\mc{O} (W^{-\frac{1}{p}-\frac{(M-2)}{2}})\quad M=3,4,\\
&\| u - u^{(5)}_{0} \|_p \leq \mc{O} (W^{-\frac{1}{p}-\frac{3}{2}}),\qquad \|\Lambda^{-1} (u - u^{(5)}_{0}) \|_p \leq
\mc{O} (W^{-\frac{1}{p}-\frac{3}{2}}).\\
\end{split}
\]
Moreover 
\[\begin{split} \| u  \|_p \leq \mc{O} (W^{-\frac{1}{p}})~, \quad 
\forall n\geq1 \,\,\, \| \Lambda^{n} u  \|_p \leq C^{n}\sqrt{n!} \mc{O} (W^{-\frac{1}{p}-\frac{1}{2}})~, \\
\forall n \geq 0 \,\,\, \| e^{V}\Lambda^{n} u  \|_p \leq C^{n+1}\sqrt{n!} \Big ( \mc{O} (W^{-\frac{1}{p}-\frac{1}{2}})+\mc{O} (W^{-\frac{1}{p}-\frac{n}{2}})  \Big)
\end{split}\]
\end{cor}

Note that combining Proposition~\ref{prop:approx} with equations~\eqref{eq:eigenfrecon}, we have:
\[
\begin{split}
&\| u - u^{(0)}_{0} \|_p \leq \mc{O} (W^{-\frac{1}{p}}),\qquad \|\Lambda^{-1} (u - u^{(0)}_{0}) \|_p \leq \mc{O} (W^{-\frac{1}{p}})\\
&\| u - u^{(M)}_{0} \|_p \leq \mc{O} (W^{-\frac{1}{p}-\frac{(M-2)}{2}}),\qquad \|\Lambda^{-1} (u - u^{(M)}_{0}) \|_p \leq
\mc{O} (W^{-\frac{1}{p}-\frac{(M-2)}{2}})\quad M=3,4,5.\\
\end{split}
\]
To improve this bounds, we need a longer argument.
\begin{proof}
First, note that $u_{0}^{(3)}-u_{0}^{(0)}=e^{-\alpha W\lambda^{2}}Wq^{(3)} (\lambda),$  and
\[ u_{0}^{(M+1)}-u_{0}^{(M)}=e^{-\alpha W\lambda^{2}}Wq^{(M+1)} (\lambda) \quad \text{for} \quad M=3,4~.\]
Then $\| u - u^{(0)}_{0} \|_p \leq \| u - u^{(3)}_{0} \|_p+\| u_{0}^{(3)} - u^{(0)}_{0} \|_p\leq  \mc{O} (W^{-\frac{1}{p}-\frac{1}{2}}).$
Similar arguments show that the norms for the cases $M=3,4$ are also improved by a factor $\frac{1}{2}.$
This argument yields no improvement on $\|\Lambda^{-1} (u - u^{(M)}_{0}) \|_p,$ because the $\Lambda^{-1}$ term generates an additional factor
$W^{\frac{1}{2}}.$
To prove the last inequalities note that  $\|u^{(M)}_{0} \|_p=\mc{O} (W^{-\frac{1}{p}})$ (for, say, $M=0$) yields
$\|u \|_p= \mc{O} ( W^{-\frac{1}{p}}).$ Moreover inserting \eqref{eq:newid} 
\[
\begin{split}
\| \Lambda^{n} u  \|_p &\leq \| \Lambda^{n} \mc{T}u_{0}^{(3)}  \|_p +\| \Lambda^{n} (u-\mc{T}u_{0}^{(3)})  \|_p
\leq  \| \Lambda^{n} \mc{T}u_{0}^{(3)}  \|_p + C^{n+1}\sqrt{n!}\ \Vert \Lambda^{-1}( u-u_{0})  \Vert_p\\
&\leq C^{n+1}\sqrt{n!} \Big ( \mc{O} (W^{-\frac{1}{p}-\frac{n}{2}})+\mc{O} (W^{-\frac{1}{p}-\frac{1}{2}})  \Big)=
 C^{n+1}\sqrt{n!} \mc{O} (W^{-\frac{1}{p}-\frac{1}{2}})
\end{split}
\]
where in the first line we used \eqref{eq:regudecayu} and in the last line we used the explicit expression
$\mc{T}u_{0}^{(3)}=e^{-V (\lambda)} e^{-\alpha W\mu \lambda^{2}} (1+W\mu^{3}q^{(3)} (\lambda)),$ together with the constraint $n\geq 1.$
The same argument yields the bound on $\| e^{V}\Lambda^{n} u  \|_p.$
\end{proof}

\begin{rem}\label{re:ubound}
In the rest of the paper we will mostly use $u_{0}^{(0)},$ $u_{0}^{(3)}$ and $u_{0}^{(5)}.$ While the latter gives a better approximation of $u$,
the first two are easier to deal with. This last feature is particularly useful in some parts of the proof.
\end{rem}
A key ingredient of our proofs will be the comparison of the function $e^{-V},$ where the operator $\mc T^{n}$
applies, with the exact eigenfunction $u$. This is done in the following lemma.
\begin{lemma}
Let $u_{0}^{(M)}$ as above and  $u$ be the solution of $\mc{T}u=u$ constructed from $u_{0}^{(M)}$ via \eqref{eq:correctef1}.
For $p\in (1,\infty),$ it holds
\begin{equation}\label{eq:uVerrorbound}
\begin{split}
\|(e^{-V}-u) \|_p &\leq  C \\
\|\Lambda^{-1} (e^{-V}-u) \|_p &\leq \ \left\{
\begin{array}{ll}
C W^{\frac{1}{2}-\frac{1}{p}} & p>2\\
C \ln W  & p=2\\
C & p<2
\end{array}
\right.\\
\end{split}
\end{equation}
\end{lemma}
\begin{proof}
The first bound follows directly from $ \|e^{-V}\|_{p}\leq \mc{O} (1)$ and Remark~\ref{re:ubound}. It suffices to prove  the second bound replacing
$u$ by $u_{0}^{(0)}.$ Indeed by Corollary~\ref{c:bounds} $\|\Lambda^{-1} (u-u_{0}^{(0)}) \|_p\leq \mc{O} ( W^{\frac{1}{2}-\frac{1}{p}})$
where $ W^{\frac{1}{2}-\frac{1}{p}}\leq \ln W$ for $p=2$ and $ W^{\frac{1}{2}-\frac{1}{p}}\leq 1$ for $p>2.$ By similar arguments we can replace
$e^{-V}$ by $e^{-\frac{1}{2}\lambda^{2}}$ in the estimate. We have
\[
\|\Lambda^{-1} (e^{-\frac{1}{2}\lambda^{2}}- e^{-W\alpha \lambda^{2}}) \|_{p}^{p}\leq C \int_{0}^{\infty} t^{1-p}\left | e^{-\frac{t^{2}}{2}}- e^{-W\alpha t^{2}} \right |^{p} dt.
\]
We use different bounds in the various integration regions. For $t\leq W^{-\frac{1}{2}}$ we have
\[
\int_{0}^{\frac{1}{\sqrt{W}}} t^{1-p}\left | e^{-\frac{t^{2}}{2}}- e^{-W\alpha t^{2}} \right |^{p} dt\leq C W^{p}\int_{0}^{\frac{1}{\sqrt{W}}} t^{p+1}=
\mc{O} (W^{\frac{p}{2}-1}) .
\]
In all other regions we estimate $e^{-\frac{t^{2}}{2}}$ and  $e^{-W\alpha t^{2}}$ separately. Direct computation give
$\int_{\frac{1}{\sqrt{W}}}^{\infty} t^{1-p}e^{-Wp \Re \alpha\, t^{2}}  dt= C   W^{\frac{p}{2}-1}$ and
$\int_{1}^{\infty} t^{1-p}e^{-\frac{1}{2}p t^{2}}  dt= \mc{O} (1).$ Finally 
\[
\begin{split}
\int_{\frac{1}{\sqrt{W}}}^{1} t^{1-p}e^{-\frac{1}{2}p t^{2}}  dt\leq \left\{\begin{array}{ll}
C & p<2\\
C \ln W & p=2\\
C W^{\frac{p}{2}-1}& p>2\\
\end{array} \right.  
\end {split}
\]
This concludes the proof.
\end{proof}

\section{Proof of Theorems \ref{thm} and \ref{thm2} away from the edges}\label{prf}

Throughout this section we assume that $|E| \leq \sqrt{\frac{32}9}$. This assumption
will be relaxed in Section~\ref{S:deform}. For technical reasons we will initially assume also $E\neq 0.$
We will explain at the end of this section how to deal with $E=0.$

\subsection{Preliminary results}\label{subs:prel}

Recalling the definition of $\rho (E)$ and $\rho_{N} (E)$ \eqref{eq:defdos},
the relevant quantities to study are
$\lim_{\eps\searrow 0}\frac{1}{N} \big\langle \tr G[H_N](E_\eps)\big\rangle$ and
$\lim_{\eps\searrow 0} \big\langle  G_{00}[H](E_\eps)\big\rangle,$ 
where $H_{N}$ is the finite-volume operator, whereas $H$ is the operator in infinite volume.
From  (\ref{finalrep}) we have
\begin{equation}\begin{split}\label{finalreprep}
\lim_{\eps\searrow 0}\frac{1}{N} \big\langle \tr G[H_N](E_\eps)\big\rangle - \mc E  &= \frac{1}{N} \sum_{y=1}^N I_{N}(y)~, \\
\text{where} \quad I_{N}(y) &=
\frac{1}{2\pi }\int \frac{\ud\lambda_1\ud\lambda_2}{\lambda_1-i\lambda_2}\ e^{V(\lambda)}\  \big[\mc T^{y-1}e^{-V}](\lambda)\ 
\big[\mc T  ^{N-y}e^{-V}](\lambda)~.
\end{split}
\end{equation}
For the individual diagonal matrix elements of the finite volume resolvent, we have a similar expression
\begin{equation}\label{finalreprep'}
\lim_{\eps\searrow 0}\big\langle G_{yy}[H_N](E_\eps)\big\rangle - \mc E = -W^2(I_{N}(y-1) + I_{N}(y+1))+ (2W^2 + 1) I_{N}(y)~, \quad y \neq 1,N~.  
\end{equation}
Recall that  $u$ is the top eigenfunction constructed in Sect.~\ref{seceigenfn} above, i.e. $\mc{T}u=u$ and
$u (0)=1.$ Inserting the decomposition $e^{-V} = u + (e^{-V}-u) $ 
we get for $k\geq 1$ (cf. \eqref{eq:tr1})
\[
\mc T^{k} e^{-V} = u + \mc T^{k}\big[e^{-V}-u\big] = u + T^{k} (e^{-V}-u)
\]
where we used  $(e^{-V}-u) (0)=0.$ Therefore $I_{N} (y)= I_1 + I_2(y-1) + I_2(N-y) + I_3(y-1) $ where
\begin{equation}\label{eq:defIj}
\begin{split}
I_{1}&:=\frac{1}{2\pi } \int   \ud\lambda_1\ud\lambda_2\  [\Lambda^{-1}e^{V}]\ u^2 \\
I_{2} (k)&:=\frac{1}{2\pi }  \int   \ud\lambda_1\ud\lambda_2\ [\Lambda^{-1}e^{V}]\ u \ [T^{k-1} (e^{-V}-u)]\qquad 1\leq k\leq N\\
I_{3} (k)&:= \frac{1}{2\pi } \int   \ud\lambda_1\ud\lambda_2\ [\Lambda^{-1}e^{V}]\  [T^{k-1} (e^{-V}-u)]\  [T^{N-k} (e^{-V}-u)].
\end{split}
\end{equation}
Note that $I_{1}$ is independent of $N$ and $y.$ The following proposition estimates the decay
of $I_{2}$ and $I_{3}$ and is a key ingredient for the proof of our results.
The estimate we obtain here for $I_{2}$ is not optimal. We will prove a stronger bound in
Proposition~\ref{prop:Iestimateimproved} below. 
\begin{prop}\label{prop:Iestimate}
Let $I_{2},I_{3}$ as above. Then for all $k\geq 1$
\begin{equation}\label{eq:Ibound}
\begin{split}
|I_{2}(k)|& \leq  C\, W^{-\frac{1}{2}} \left(1-c/W \right)^{k-1},\\
|I_{3} (k)|&\leq C\, (\ln W)^{2} \left(1-c/W \right)^{N-4}.\\
\end{split}
\end{equation}
\end{prop}
\begin{proof}
For $k\geq 2$ we write
\[
I_{2} (k)= \int   \tfrac{\ud\lambda_1\ud\lambda_2}{2\pi } \ [\delta_{W^2}^\star u] (\lambda ) \ [\Lambda^{-1}T^{k-2} (e^{-V}-u) (\lambda )]
\]
Inserting absolute values and setting $\frac{1}{q}+\frac{1}{p}=1,$ $p>2$ yields
\[
\begin{split}
2\pi\, |I_{2}(k)|& \leq  \|\delta_{W^2}^\star u\|_{q}\, \|\Lambda^{-1}T^{k-2} (e^{-V}-u)\|_{p}
\leq   \|u\|_{q}\, \|\Lambda^{-1}(e^{-V}-u)\|_{p}  \Big(1-\tfrac cW\Big)^{\max\{0,k-3\}}\\
& \leq\  C W^{-\frac{1}{q}}\, W^{\frac{1}{2}-\frac{1}{p}} \Big(1-\tfrac cW\Big)^{k-2}= C W^{-\frac{1}{2}}\Big(1-\tfrac cW\Big)^{k-2},
\end{split}
\]
where in the first line we applied eq.~\ref{cor:Tnormest}, while in the second line we used \eqref{eq:uVerrorbound} for $p>2$
and Corollary~\ref{c:bounds}.
For $k=1,$ using again  \eqref{eq:uVerrorbound} for $p>2$
and Corollary~\ref{c:bounds}, the bound reduces to 
\[
2\pi |I_{2} (1)|=|\int   \ud\lambda_1\ud\lambda_2 \ u e^{V}\Lambda^{-1}  (e^{-V}-u)|\leq \|u e^{V}\|_{q} \|\Lambda^{-1}  (e^{-V}-u) \|_{p}
\leq \mc{O} (W^{-\frac{1}{2}}).
\]
It remains to prove the decay of $I_{3} (k).$ Since $N\gg 1,$ it holds $k-1>2$ or $N-k>2.$ We assume without loss of generality $k-1>2.$
Then applying \eqref{cor:Tnormest} and  \eqref{eq:uVerrorbound} for $p=2$ we get
\[
\begin{split}
2\pi |I_{3} (k) | & = |\int   \ud\lambda_1\ud\lambda_2\ \  \delta_{W^2}^\star[\Lambda^{-1}T^{k-2} (e^{-V}-u)]\  [T^{N-k} (e^{-V}-u)]|\\
& \leq \ C \,   \|\Lambda^{-1} (e^{-V}-u) \|_{2}  \|\Lambda^{-1} (e^{-V}-u) \|_{2} \Big(1-\frac cW\Big)^{N-4}= \mc{O} ((\ln W)^{2})
\Big(1-\frac cW\Big)^{N-4}.
\end{split}
\]
\end{proof}

A first consequence of this estimates is the following representation for the infinite volume Green's
function.
\begin{prop}\label{prop:viaI}
Let $I_1, I_2, I_3$ be as in (\ref{eq:defIj}). Then
\begin{equation}\label{eq:viaI}\begin{split}
\lim_{\eps\searrow 0} \big\langle G_{00}[H](E_\eps)\big\rangle - \mc E  &= I_1 \\
\lim_{\eps\searrow 0} \left\{ \big\langle G_{00}[H](E_\eps)\big\rangle - \frac{1}{N} \big\langle \tr G[H_N](E_\eps)\big\rangle\right\} &=
\frac{1}{N} \sum_{y=0}^{N-1} \left\{ 2 I_2(y) + I_3(y) \right\}~. \end{split}\end{equation}
\end{prop}

\begin{proof}
Let $y = y(N)$ so that $\min(y, N-y) \geq N^{0.01},$  and rewrite
\[
\langle G_{00}[H](E_\eps)\rangle = 
\langle G_{y y}[H](E_\eps)\rangle = \langle G_{y  y}[H_N](E_\eps)\rangle  +
\langle (G_{y  y }[H](E_\eps) - G_{y y}[H_N](E_\eps))\rangle.
\] 
For fixed $\epsilon$, the second term vanishes in the limit $N \to \infty$. Indeed, it is equal
to a sum of several boundary terms such as  $\langle G_{y1}[H_N](E_\eps) G_{1y}[H](E_\eps) \rangle~.$
Each of these terms tends to zero: indeed, $|G_{1y}[H](E_\eps)| \leq \epsilon^{-1}$, whereas
$\langle |G_{y1}[H_N](E_\eps)| \rangle$ tends to zero by (an appropriate version of) the Combes--Thomas bound (see e.g.\ \cite{AizWar}).
Precisely let $X_{N}$ denote the event $| (H_{N})_{jk}|\leq K\ N^{\alpha } J^{1/2}_{jk},$ where
$J_{jk}= (-W^{2}\Delta_{N}+\mathrm{id})^{-1}$ is the covariance of the random matrix $H_{N}$ and decays exponentially
$J_{jk}\leq c_{W} e^{-|j-k|/W}$,
$0<\alpha <1$ and $K\gg 1$ are some fixed parameters.
Then $\sup_{x} \sum_{y} | (H_{n})_{xy} (e^{\delta (|x-y|)}-1) |\leq K' N^{\alpha } $ as long as $J_{jk}^{1/2}e^{\delta |j-k|}$ retains
some exponential decay. By Combes-Thomas
\[
\langle |G_{y1}[H_N](E_\eps)| \mathbf{1}_{X_{N}} \rangle \leq \epsilon^{-1} e^{- c|y-1| \epsilon /N^{\alpha }}\to_{N\to \infty} 0,
\]
where $c>0$ is some constant and we used $y/N^{\alpha }\to \infty$ as $N\to \infty.$
To conclude we show that $X_{N}^{c}$ has vanishing probability:
\[
\langle |G_{y1}[H_N](E_\eps)| \mathbf{1}_{X_{N}^{c}} \rangle \leq \epsilon^{-1} \mathbb{P} (X_{N}^{c})
\leq K \epsilon^{-1}\sum_{ij} e^{-c N^{2\alpha }}\leq K \epsilon^{-1} N^{2} e^{-c N^{2\alpha }}\to_{N\to \infty} 0.
\]
We obtain (recall $y=y (N)$)
\[
\lim_{\eps\searrow 0}  \langle G_{00}[H](E_\eps)\rangle =
\lim_{\eps\searrow 0} \lim_{N\to \infty} \langle G_{y y}[H_{N}](E_\eps)\rangle=  \lim_{N\to \infty}\lim_{\eps\searrow 0}
\langle G_{y y}[H_{N}](E_\eps)\rangle
\]
where in the last equality we can exchange limits since after translating to the saddle in the integral representation
for $\langle |G_{y1}[H_N](E_\eps)| \rangle,$ all integrals are bounded unformly in $\epsilon.$
The result now follows from representation \eqref{finalreprep'} and estimates \eqref{eq:Ibound}. 
\end{proof}

Using this Proposition,  \eqref{eq:distfromsc} and  \eqref{eq:distfromsc} of Theorem~\ref{thm} reduce to a study of $I_1 $ and $I_2,I_3 $,
respectively. However, to obtain the  error estimate \eqref{eq:distfromsc} we will need an improved version
of \eqref{eq:Ibound}  for $I_{2} (k)$ that requires substantial more work. This will
be done in Proposition \ref{prop:Iestimateimproved} below. 
The bound on $I_1 $ follows from the properties of $u $, in particular, its approximate symmetry
under $\lambda\to-\lambda $. The argument is given in section \ref{sectproofthm2} below.

\subsection{Proof of Theorem~\ref{thm}}\label{sectproofthm1}

\subsubsection{Convergence to $\rho$}
Our goal is to prove $ \big\vert  \rho_N(E) - \rho(E)\big\vert \leq \frac {C'(E)}{N }$
for  $\vert E\vert < \sqrt{32/9}$ (away from the edge)  and $N\geq C(E) W\log W.$
Recalling the definition of $\rho (E)$ and $\rho_{N} (E)$ \eqref{eq:defdos}, Proposition~\ref{prop:viaI}
implies
\[
 \rho_N(E) - \rho(E)= -\frac{1}{\pi } \Im [\frac{1}{N} \sum_{y=0}^{N-1} \left\{ 2 I_2(y) + I_3(y) \right\}]
\]
A direct application of \eqref{eq:Ibound} yields $ \big\vert  \rho_N(E) - \rho(E)\big\vert \leq \frac {C'(E)}{N } W^{\frac{1}{2}}.$
To extract the correct $W$ prefactor we need  the following improved version
of \eqref{eq:Ibound}.

\begin{prop}[{\bf Improved estimate on $I_{2} (k)$}]\label{prop:Iestimateimproved}
Let $I_{2}$ be as as above. Then
\begin{equation}\label{eq:Iboundimproved}
\begin{split}
|I_{2}(k)|& \leq  \mc O (W^{-1}) \left(1-c/W \right)^{k-4}\qquad k\geq 4,\\
\end{split}
\end{equation}
\end{prop}
\begin{proof}
Note that for any (regular enough) functions $f,g$ it holds
\begin{equation}\label{eq:integrationbyparts}
\int   \ud\lambda_1\ud\lambda_2\ g\  [T^{n} f] = \int   \ud\lambda_1\ud\lambda_2\   [T^{n} (\Lambda^{2}e^{-V}g)]\  [e^{V}\Lambda^{-2}f].
\end{equation}
Replacing $f= (e^{-V}-u)$ and $g=\Lambda^{-1}e^{V}u$ the integral $I_{2} (k)$ can be written as
\begin{equation}\label{eq:improved-decomp}
2\pi  I_{2} (k)= \int   \ud\lambda_1\ud\lambda_2 \ [\Lambda^{-1}e^{V}u]\  [T^{k-1} (e^{-V}-u)]=
 \int   \ud\lambda_1\ud\lambda_2 \  [\Lambda^{-2}e^{V} (e^{-V}-u)] \  [T^{k-1} (\Lambda u)]. 
\end{equation}
The proof of the Proposition relies on two main ingredients:\textit{(a)} $\Lambda u$ is an approximate eigenfunction of $T$ with eigenvalue $\mu$, i.e.
$T (\Lambda u)\simeq \mu (\Lambda u)$ and \textit{(b)} the contribution $I_{2} (0)$ from $k=0$ is smaller than expected. More precisely

\begin{prop}\label{prop:nexteigenf}
Let $u_{0}^{(0)} (\lambda )=e^{-\alpha W\lambda^{2}}$ and $u_{0}^{(3)} (\lambda )=e^{-\alpha W\lambda^{2}} (1+Wq^{(3)} (\lambda)),$ as in   Proposition~\ref{prop:approx}, with $q^{(3)} (\lambda )
=c_{3} [\lambda_{1}^{3}- (i\lambda_{2}^{3})]$. Recall that
$\mu = (1+\frac{2\alpha }{W})^{-1}$ and the definition of $v^{(M)}$ in \eqref{eq:errorbis}. Then
\begin{equation}\label{eq:nexteigenf1}
T (\Lambda u_{0}^{(3)}) = \mu (\Lambda u_{0}^{(3)})  + \mu \ \Lambda\   \mathrm{Rem} (\lambda) 
\end{equation}
where $\mathrm{Rem} := 3c_{3}\frac{\mu^{3}}{W}\bar\Lambda  ( u_{0}^{(0)}+ v^{(0)})+ v^{(3)}.$ Moreover
\begin{equation}\label{eq:nexteigenf2}
\|\Lambda^{-1}e^{V}[T^{n} (\Lambda u_{0}^{(3)})- \mu^{n} (\Lambda u_{0}^{(3)}) ]\|_{p}\leq C \ W^{-\frac{1}{2}-\frac{1}{p}} (1-\tfrac{c}{W})^{n-3}\qquad
\forall p\in (1,\infty).
\end{equation}
\end{prop}

\begin{prop}\label{prop:firstterm}
Let $I_{2,0} (k):= \tfrac{1}{2\pi } \mu^{k-1} \int   \ud\lambda_1\ud\lambda_2 \  [\Lambda^{-1}(e^{-V}-u)] \   [e^{V}u_{0}^{(3)}].$
It holds
\[
|I_{2,0} (k) |\leq C |\mu|^{k-1} W^{-\frac{3}{2}}.
\]
\end{prop}
The proofs are given below. We decompose now the integral in \eqref{eq:improved-decomp}
as $I_{2} (k)=I_{2,0} (k)+I_{2,1} (k)+I_{2,3} (k)$ where $I_{2,0}$ was defined above and
\[
\begin{split}
I_{2,1} (k) &:=  \tfrac{1}{2\pi }
\int   \ud\lambda_1\ud\lambda_2 \  [\Lambda^{-2}e^{V} (e^{-V}-u)] \  [T^{k-1} (\Lambda u_{0}^{(3)})-\mu^{k-1}\Lambda u_{0}^{(3)}]\\
I_{2,2} (k) &:=  \tfrac{1}{2\pi}
\int   \ud\lambda_1\ud\lambda_2 \  [\Lambda^{-2}e^{V} (e^{-V}-u)] \  [T^{k-1} (\Lambda (u-u_{0}^{(3)}))].  
\end{split}
\]
For the first integral we obtain 
\[
 \begin{split}
|2\pi \ I_{2,1} (k) |&\leq \|\Lambda^{-1} (e^{-V}-u)\|_{p} \| \Lambda^{-1}e^{V}[T^{k-1} (\Lambda u_{0}^{(3)})-\mu^{k-1}\Lambda u_{0}^{(3)}] \|_{q}\\
&
\leq  C\ W^{\frac{1}{2}-\frac{1}{p}} W^{-\frac{1}{2}-\frac{1}{q}}  (1-\tfrac{c}{W})^{k-4}= C W^{-1}  (1-\tfrac{c}{W})^{k-4}.
\end{split}
\]
where we applied \eqref{eq:uVerrorbound} for $p>2.$
Finally, applying Corollary~\ref{c:bounds} 
\[
\begin{split}
|2\pi \ I_{2,2} (k) |&\leq  \|\Lambda^{-1} (e^{-V}-u)\|_{p} \| \Lambda^{-1}e^{V} [T^{k-1} (\Lambda (u-u_{0}^{(3)}))] \|_{q}\\
&=  \|\Lambda^{-1} (e^{-V}-u)\|_{p}\  \| \delta_{W^{2}}^{\star}\Lambda^{-1}[T^{k-2} (\Lambda (u-u_{0}^{(3)}))] \|_{q}\\
&\leq\  C(1-\tfrac{c}{W})^{k-3} \ W^{\frac{1}{2}-\frac{1}{p}}\  \|u-u_{0}^{(3)}\|_{q}  =\mc{O} ( W^{-\frac{3}{2}})
(1-\tfrac{c}{W})^{k-3}. 
\end{split}
\]
This concludes the proof of Proposition~\ref{prop:Iestimateimproved}.  
\end{proof}

\paragraph{Proof of Proposition~\ref{prop:nexteigenf}}
 Note that both $V$ and $u_{0}^{(M)}$ 
have a representation as functions of the supermatrix $R.$ Moreover $\Lambda u_{0}^{(3)}$ vanishes at $0.$
Then we can rewrite
\[
[T (\Lambda  u_{0}^{(3)})] (\lambda ) = [\mc{T} (\Lambda  u_{0}^{(3)})] (\lambda )=
[\mf{T}[ \str (\cdot) U_{0}^{(3)}] (R)_{| R=\mbox{diag}(\lambda_{1},i\lambda_{2})}
\]
where we used $\Lambda=\str R.$
From equation \eqref{eq:trans} and Lemma~\ref{le:someidentities2} $(d)$  we find
\[
\begin{split}
\mf{T}[ \str (\cdot) U_{0}^{(3)}] (R) &= \mu \str R\Big[ \mf{T}[ U_{0}^{(3)}] (R)+ \frac{3\mu^{3} c_{3}}{W}
\frac{\str R^{2}}{\str R} \mf{T}[ U_{0}^{(0)}] (R)\Big ]=\mu\  U_{0}^{(3)} (R)\str R\\
&   + \mu  U_{0}^{(0)} (R)  \str R\Big[  \mathrm{Err}^{(3)} (R)+ \frac{3\mu^{3}c_{2} }{W}
\frac{\str R^{2}}{\str R} \Big ( \mathrm{Err}^{(0)} (R) + 1\Big )\Big ]\ \\
& = \mu\  \mathrm{id} [\str (\cdot) U_{0}^{(3)} ] (R)\ +\mu\   \str (R) \mathrm{Rem} (R) 
\end{split}
\]
where $\mathrm{Rem} (R) := U_{0}^{(0)} (R)\mathrm{Err}^{(3)} (R) + \frac{3\mu^{3}c_{3}}{W}\frac{\str R^{2}}{\str R} U_{0}^{(0)} (R) (1 + \mathrm{Err}^{(0)} (R) ),$
and $\mathrm{Err}^{(0)},\mathrm{Err}^{(3)}$ were introduced in \eqref{eq:error}. Then \eqref{eq:errorbis} yields 
\eqref{eq:nexteigenf1}. In order to prove \eqref{eq:nexteigenf2} we insert the decomposition
 $\mf{T}^{n} = \mu^{n}\ \mathrm{id}+ \sum_{l=0}^{n-1} \mu^{n-1-l} \mf{T}^{l} (\mf{T}-\mu \mathrm{id}):$ 
\[
\begin{split}
T^{n} (\Lambda u_{0}^{(3)})&= \mf{T}^{n}[\str (\cdot) U_{0}^{(3)} ](R)_{| R=\mbox{diag}(\lambda_{1},i\lambda_{2})}  \\
&= \mu^{n}\ \Lambda  u_{0}^{(3)} (\lambda )  + \sum_{l=0}^{n-1} \mu^{n-1-l} T^{l}
[\Lambda \ \mathrm{Rem} (\lambda ) ].
\end{split}
\]
Using $|\mu |\leq (1-\tfrac{c}{W})$ we get
\[
\begin{split}
&\|\Lambda^{-1}e^{V}[T^{n} (\Lambda u_{0}^{(3)})- \mu^{n} (\Lambda u_{0}^{(3)}) ]\|_{p} \leq  \sum_{l=0}^{n-1} (1-\tfrac{c}{W})^{n-1-l}
\|\Lambda^{-1}e^{V}[T^{l} (\Lambda \ \mathrm{Rem}  )  ]\|_{p} \\
&\quad \leq   (1-\tfrac{c}{W})^{n-3}\Big[ \| e^{V} \mathrm{Rem}\|_{p} \ +\  n  
\| \mathrm{Rem}\|_{p}\Big]\leq (1-\tfrac{c}{W})^{n-3}\Big[ \| e^{V} \mathrm{Rem}\|_{p} \ +\  W  
\| \mathrm{Rem}\|_{p}\Big]
\end{split}
\]
The estimates \eqref{eq:vbounds} now yield
\[
\begin{split}
\| \mathrm{Rem}\|_{p}&\leq \| v^{(3)}\|_{p}+  W^{-1} \|\bar\Lambda  u^{(0)}_{0}\|_{p} +
W^{-1} \|\bar\Lambda  v^{(0)}\|_{p}\\
&= \mc{O} (W^{-\frac{1}{p}-2})+   W^{-1} \mc{O} (W^{-\frac{1}{2}-\frac{1}{p}}) 
+W^{-1}\mc{O} (W^{-\frac{1}{p}-2})= \mc{O} (W^{-\frac{3}{2}-\frac{1}{p}}).
\end{split}
\]
The same bound holds for $\| e^{V}\mathrm{Rem}\|_{p}.$ This completes the proof.
\qed

\paragraph{Proof of Proposition~\ref{prop:firstterm}}
Inserting the decomposition $(e^{-V}-u)= (e^{-V}-u_{0}^{(3)})+ (u_{0}^{(3)}-u)$ we get
$I_{2,0} (k)=\frac{1}{2\pi }\mu^{k-1}[I_{2,0,a}+I_{2,0,b}]$ where 
\[
\begin{split}
I_{2,0,a}=& \int   \ud\lambda_1\ud\lambda_2 \  [\Lambda^{-1}(e^{-V}-u_{0}^{(3)})] \   [e^{V}u_{0}^{(3)}]\\
I_{2,0,b}=& \int   \ud\lambda_1\ud\lambda_2 \  [\Lambda^{-1}(u_{0}^{(3)}-u)] \   [e^{V}u_{0}^{(3)}].
\end{split}
\]
Decomposing further $1-e^{V}u_{0}^{(3)}= 1-u_{0}^{(3)}-(e^{V}-1)u_{0}^{(3)} $
\begin{equation}\label{eq:u3dentity}
\begin{split}
I_{2,0,a}&=\int   \ud\lambda_1\ud\lambda_2 \ \Lambda^{-1} u_{0}^{(3)}- 
\int   \ud\lambda_1\ud\lambda_2 \ \Lambda^{-1} ( u_{0}^{(3)})^{2}\\
&-\int   \ud\lambda_1\ud\lambda_2 \ \Lambda^{-1} [u_{0}^{(3)} (1-e^{-V})]\ [u_{0}^{(3)}e^{V}]
\end{split}
\end{equation}
The first two integrals vanish. This can be more easily seen by going back to $R$ coordinates
\begin{equation}\label{eq:integral1}
\begin{split}
\frac{1}{2\pi }\int   \ud\lambda_1\ud\lambda_2 \ \Lambda^{-1} u_{0}^{(3)}&= \int   \ud R\ a \    u_{0}^{(3)} (R)\\
&=  \int   \ud R\ a \  e^{-\alpha W\str R^{2}}+ Wc_{3}
 \int   \ud R\ a \  e^{-\alpha W\str R^{2}} \str R^{3}.
 \end{split}
\end{equation}
The first integral equals 0 by parity, the second by explicit computation.
In the same way
\[
\begin{split}
\frac{1}{2\pi }\int   \ud\lambda_1\ud\lambda_2 \ \Lambda^{-1} (u_{0}^{(3)})^{2}&=
\int   \ud R\ a \    (u_{0}^{(3)} (R))^{2}\\
&=  W^{2}c_{1}^{2}\int   \ud R\ a \  e^{-2\alpha W\str R^{2}} (\str R^{3})^{3}=0
 \end{split}
\]
where the last integral vanishes by parity and the other contributions vanish using \eqref{eq:integral1}.
Hence, using $1-e^{-V}=\mc{O} (\lambda^{2}),$
\[
\begin{split}
|I_{2,0,a}|&=|\int   \ud\lambda_1\ud\lambda_2 \ \Lambda^{-1} [u_{0}^{(3)} (1-e^{-V})]\ [u_{0}^{(3)}e^{V}]|\\
&\leq
\|\Lambda^{-1} [u_{0}^{(3)} (1-e^{-V})]\|_{p}  \ \| u_{0}^{(3)}e^{V}  \|_{q}=
\mc{O} (W^{\frac{1}{2}-1-\frac{1}{p}}) \mc{O} (W^{-\frac{1}{q}})=\mc{O} (W^{-\frac{3}{2}}).
\end{split}
\]
Finally, using Corollary~\ref{c:bounds} again 
\[
\begin{split}
|I_{2,0,b}|\leq & \| \Lambda^{-1}(u_{0}^{(3)}-u) \|_{p}\
\|e^{V}u_{0}^{(3)} \|_{q}\leq \mc{O} (W^{-\frac{1}{p}-\frac{1}{2}}) \mc{O} (W^{-\frac{1}{q}})=\mc{O} (W^{-\frac{3}{2}}).
\end{split}
\]
This concludes the proof.\qed

\subsubsection{Semi-circle law.}
Our goal is to prove that the infinite volume density of states $\rho (E)$ 
satisfies $\big\vert  \rho(E) - \rho_{\mathrm{s.c.}}(E) \big\vert \leq \mc{O} (W^{-2})$
According to Proposition~\ref{prop:viaI}, 
\[
\lim_{\eps\searrow 0} \big\langle \tr G_{00}[H](E_\eps)\big\rangle - \mc E  = I_1[E]
=\frac{1}{2\pi }\int   \ud\lambda_1\ud\lambda_2  \  \Lambda^{-1} e^{V} u^2,
\]
hence $ \rho(E) - \rho_{\mathrm{s.c.}}(E)= -\frac{1}{\pi } \Im I_{1}.$
Inserting the decomposition $u=u_{0}^{(5)}+ (u-u_{0}^{(5)})$  we obtain
$I_{1}= I_{1,0}+ 2I_{1,1}+I_{1,3},$ where
\[
\begin{split}
&I_{1,0} =\frac{1}{2\pi }  \int \ud\lambda_1\ud\lambda_2\  \Lambda^{-1}e^{V(\lambda)} (u_{0}^{(5)})^2,  \qquad 
I_{1,1}=\frac{1}{2\pi }  \int\ud\lambda_1\ud\lambda_2\  \Lambda^{-1}e^{V(\lambda)}u_0^{(5)} (u-u_{0}^{(5)})\\
&I_{1,2} = \frac{1}{2\pi }  \int\ud\lambda_1\ud\lambda_2\   \Lambda^{-1}e^{V(\lambda)} (u-u_{0}^{(5)})^{2}.
\end{split}\]
The second integral is easily bounded by
\[
|I_{1,1}|\leq   \|e^{V(\lambda)}u_0^{(5)} \|_{p}\|\Lambda^{-1} (u-u_{0}^{(5)})\|_{q}\leq \mc{O} (W^{-\frac{1}{p}}) \mc{O} (W^{-\frac{1}{q}-\frac{3}{2}})=
\mc{O} (W^{-2-\frac{1}{2}}).
\]
To control the $e^{V}$ factor in $I_{1,2}$ we insert the identity $u-u_{0}^{(5)}= v^{(5)}+ T (u-u_{0}^{5}):$
\[
\begin{split}
|I_{1,2} |&\leq
\  \|\Lambda^{-1}(u-u_{0}^{5})\|_{p} \| e^{V(\lambda)} v^{(5)} \|_{q}    +  \| (u-u_{0}^{5})\|_{q} \|\Lambda^{-1}e^{V}T (u-u_{0}^{5})\|_{p}\\
& \leq   \|\Lambda^{-1}(u-u_{0}^{5})\|_{p}  \Big ( \| e^{V(\lambda)} v^{(5)} \|_{q} +  \| (u-u_{0}^{5})\|_{q}  \Big )\\
& \leq \mc{O} (W^{-\frac{1}{p}} W^{-\frac{3}{2}})  \Big (  \mc{O} (W^{-\frac{1}{q}}) + \mc{O} (W^{-\frac{1}{q}-\frac{3}{2}})  \Big )
 = \mc{O} (W^{-2-\frac{1}{2}}),
\end{split}
\]
where we used the fact the $v^{(5)}=\mc{T}u_{0}^{(5)}- u_{0}^{(5)}$ has always an exponential prefactor $e^{-\alpha W\lambda^{2}}.$
Finally, to compute $I_{1,0},$ we remark that \eqref{eq:Vexpansion} yields
$V (\lambda )=\frac{1-\mc{E}^{2}}{2}\lambda^{2}+\mc{O} (\lambda^{3}).$
We decompose  $(2\pi) I_{1,0}=I_{1,0,a}+I_{1,0,b}$ with
\[
I_{1,0,a}= \int \ud\lambda_1\ud\lambda_2\  \Lambda^{-1}e^{V(\lambda)} (1- e^{\frac{1-\mc{E}^{2}}{2}\lambda^{2}-V}) (u_{0}^{(5)})^2,
\quad I_{1,0,b}=
\int \ud\lambda_1\ud\lambda_2\  \Lambda^{-1}e^{\frac{1-\mc{E}^{2}}{2}\lambda^{2}} (u_{0}^{(5)})^2.
\]
The first integral is  bounded as follows
\[
|I_{1,0,a}|\leq \|\Lambda^{-1}  (1- e^{\frac{1-\mc{E}^{2}}{2}\lambda^{2}-V}) u_{0}^{(5)}\|_{p}\  \|e^{V}u_{0}^{(5)}\|_{q}\leq
\mc{O} (W^{\frac{1}{2}-\frac{3}{2}-\frac{1}{p}}) \mc{O} (W^{-\frac{1}{q}})=\mc{O} (W^{-2}).
\]
The remaining term $I_{1,0,b}$ is estimated using $u_{0}^{(5)}=e^{-\alpha W\lambda^{2}} (1+Wq^{(3)}+Wq^{(4)}+Wq^{(5)}),$
where $q^{(3)},q^{(5)}$ are odd polynomials while $q^{(4)}$ is even (cf. Proposition~\ref{prop:approx}).
Replacing  $\alpha $ by $\tilde{\alpha }:= \alpha -\frac{1-\mc{E}^{2}}{4W},$ we get
\[
\begin{split}
I_{1,0,b}&= \int \ud\lambda_1\ud\lambda_2\ \Lambda^{-1} e^{-2\tilde{\alpha } W\lambda^{2}}\\
&\cdot 
\Big (
(1+Wq^{(3)})^{2}+ 2W (q^{(4)}+q^{(5)})+ 2W^{2}q^{(3)} (q^{(4)}+q^{(5)})+ (Wq^{(5)})^{2}
\Big )\\
&= \int \ud\lambda_1\ud\lambda_2\ \Lambda^{-1} e^{-2\tilde{\alpha } W\lambda^{2}}\ (2Wq^{(5)}+ 2W^{2} q^{(3)}q^{(4)} ),
\end{split}
\]
where the first term vanishes by the same arguments used in \ref{eq:u3dentity} and the other terms cancel by
parity. Finally
\[
|I_{1,0,b}|\leq \mc{O} (W^{-1+\frac{1}{2}} W^{1-\frac{5}{2}})+ \mc{O} (W^{-1+\frac{1}{2}}W^{2-\frac{3}{2}-\frac{4}{2}})=\mc{O} (W^{-2}).
\]
This completes the proof of Theorem~\ref{thm} away from the edges.\qed

\subsection{Proof of Theorem~\ref{thm2}}\label{sectproofthm2}

Our goal is to prove the estimate $|\partial_{E}^{n}\rho_{N} (E)|\leq \ W^{n-1}  C(E)^{n} n!,$ $n\geq 1,$
uniformly in $N.$
Using the supersymmetric representation we have seen that
\[
\begin{split}
-\pi \partial_{E}^{n}\rho_{N} (E)&= \frac{1}{N}\partial_{E}^{n} \lim_{\eps\searrow 0}\Im \big\langle \tr G[H_N](E_\eps)\big\rangle=
 \frac{1}{N}\Im \lim_{\eps\searrow 0} \partial_{E}^{n} \tr \big\langle G[H_N](E_\eps)\big\rangle\\
&=\tfrac{(-1)^{n}}{N}
 \sum_{q=1}^{n} 
\sum_{\substack{n_{1},\dotsc n_{q}\geq 1 \\  n_{1}+\dotsb+ n_{q}=n}} \tfrac{n!}{n_{1}!\cdots n_{q}!}\sum_{m=0}^{q}
\sum_{\substack{1\leq y_{1}< y_{2}< \dotsb  < y_{q}\leq N\\
   y_{m}\leq x<y_{m+1}}}
\hspace{-0.5cm}\Im I_{x, y_{1},\dotsc ,y_{q}}^{m,n_{1},\dotsc ,n_{q}}
 \end{split}
\]
where $ I_{x, y_{1},\dotsc ,y_{q}}^{m,n_{1},\dotsc ,n_{q}}$ was defined in Corollary~\ref{corfinalrep}
and we use the convention $y_{0}=1,$ $y_{q+1}=N+1$.
Using  $\mc{T} (\Lambda f)=T (\Lambda f),$ we can reorganize the integral as follows:
\begin{equation}
\begin{split}
&I_{x, y_{1},\dotsc ,y_{q}}^{m,n_{1},\dotsc ,n_{q}}=
\int \tfrac{\ud\lambda_1\ud\lambda_2}{2\pi }\, 
 \big[  T^{x-y_{m}} \Lambda^{n_{m}}  \prod_{j=m-1}^{1}\Big (T^{y_{j+1}-y_{j}}\Lambda^{n_{j}}\Big )\mc T^{y_{1}-1}e^{-V}](\lambda )\\
&\qquad \cdot [e^{V(\lambda )} \,\Lambda^{-1}] \  \big [ T^{y_{m+1}-x}\Lambda^{n_{m+1}}
\prod_{j=m+1}^{q-1}\Big (T^{y_{j+1}-y_{j}}\Lambda^{n_{j+1}}\Big )   
\mc T^{N-y_{q}}e^{-V}](\lambda ),
\end{split}
\end{equation}
where we use the convention $n_{0}=0=n_{q+1}.$ 
Note that by \eqref{cor:Tnormest} for all $m,n\geq 1$ it holds
\[
\|\Lambda^{-1} (\Lambda^{n}T^{m})f \|_{p}\leq C^{n} \sqrt{n!}\ (1-c/W)^{m-1} \|\Lambda^{-1}f\|_p.
\]
When $f=\Lambda u$ this estimate gives a factor $\|u\|_p=\mc{O} (W^{-\frac{1}{p}}).$ The following lemma
shows the bound can be improved.
\begin{lemma}\label{le:lambdaubound}
For all $n,m\geq 1$
\[
\|\Lambda^{-1} (\Lambda^{n}T^{m})\Lambda u \|_{p}\leq C^{n} \sqrt{n!}\ (1-c/W)^{m-1} W^{-\frac{1}{p}-\frac{1}{2}}.
\]
\end{lemma}
\begin{proof}
We decompose $T^{m} \Lambda u= T^{m} \Lambda (u-u_{0}^{(3)})+ \mu^{m} \Lambda u_{0}^{(3)}+ [T^{m}-\mu^{m}\mathrm{id} ]\Lambda u_{0}^{(3)},$
where we recall that $\mu= (1+\tfrac{2\alpha }{W})^{-1}$ and $|\mu|\leq  (1-\tfrac{c}{W})^{-1}.$ Then
\[
\begin{split}
\|\Lambda^{n-1}T^{m}\Lambda u \|_{p} \leq & \|\Lambda^{n-1}T^{m}\Lambda (u-u_{0}^{(3)}) \|_{p}
+|\mu |^{m} \|\Lambda^{n}u_{0}^{(3)} \|_{p}+
 \|\Lambda^{n-1}[T^{m}-\mu^{m}\mathrm{id} ]\Lambda u_{0}^{(3)} \|_{p}\\
 &\leq \  C^{n} \sqrt{n!}\ (1-\tfrac{c}{W})^{m-1}  \| (u -u_{0}^{(3)}) \|_{p}
+|\mu |^{m} \|\Lambda^{n}u_{0}^{(3)} \|_{p}\\
&\quad + \|\Lambda^{n}e^{-V} \|_{\infty}
\|\Lambda^{-1}e^{V}[T^{m}-\mu^{m}\mathrm{id} ]\Lambda u_{0}^{(3)} \|_{p}\\
& \leq  C^{n} \sqrt{n!}\ (1-\tfrac{c}{W})^{m-1} [W^{-\frac{1}{p}-1}+ W^{-\frac{1}{p}-\frac{n}{2}} +W^{-\frac{1}{p}-\frac{1}{2}} ]
\\&=
 {C'}^{n} \sqrt{n!}\ (1-\tfrac{c}{W})^{m-1}W^{-\frac{1}{p}-\frac{1}{2}},
\end{split}
\]
where in the first line we used \eqref{cor:Tnormest} and in the last line Corollary~\ref{c:bounds}, Proposition~\ref{prop:nexteigenf} and
the explicit form of $u_{0}^{(3)},$ together with the constraint $n\geq 1.$ 
\end{proof}

\subsubsection{Case $n=1$}
For $n=1$ we have
$-\pi \partial_{E}\rho_{N} (E)= \frac{1}{N}\Im\lim_{\eps\searrow 0}\sum_{yy'} \langle G_{yy'}[H_N](E_\eps)G_{y'y}[H_N](E_\eps) \rangle $.
From Corollary \ref{corfinalrep} we have the representation
$\lim_{\eps\searrow 0}\langle G_{yy'}[H_N](E_\eps)G_{y'y}[H_N](E_\eps) \rangle =
-J_{yy'} +\sum_{xx'=1}^{N} J_{yx} J_{y'x'}  I_{xx'}$ where for $x' \geq x$
\[
I_{xx'}= \frac{1}{2\pi }
\int \ud\lambda_1\ud\lambda_2 \Lambda^{-1}\, e^{V(\lambda)} 
\big[\mc T^{x-1}e^{-V}\big ](\lambda) \, \big[T^{x'-x} \Lambda \mc T^{N-x'}e^{-V}\big ] (\lambda )
\]
and for $x' < x$ we set $I_{xx'} = I_{N-x+1, N-x'+1}$. We want to prove now $|I_{xx'}|\leq CW^{-1} (1-\frac{c}{W})^{x'-x}.$ 
We insert again the decomposition $e^{-V}=u+ (e^{-V}-u),$ which yields
$I_{xx'}= I_{1} (x'-x)+ I_{2} (x-1,x'-x)+I_{2} (N-x', x'-x)+I_{3} (x-1,x'-x,N-x') $ where
\begin{equation}
\begin{split}
I_{1} (k)& := \int \tfrac{\ud\lambda_1\ud\lambda_2}{2\pi }   \ \Lambda^{-1}   e^{V} u\  [T^{k} (\Lambda u)], \hspace{2,5cm}  k\geq 0 \\
I_{2} (k,k') & := \int\tfrac{\ud\lambda_1\ud\lambda_2}{2\pi }   \Lambda^{-1}   e^{V} \  [T^{k} (e^{-V}-u)]\  [T^{k'} (\Lambda u)], \quad k,k'\geq 0  \\
I_{3} (k,k',k'')& := \int\tfrac{\ud\lambda_1\ud\lambda_2}{2\pi }   \Lambda^{-1}  e^{V}  \  [T^{k} (e^{-V}-u)] \ 
[T^{k'} \Lambda\  T^{k''} (e^{-V}-u)],\quad
\end{split}
\end{equation}
and in $I_{3}$ we have $ k+k'+k''=N-1\gg 1.$ To obtain $I_{2} (N-x', x'-x)$ we used in addition
\eqref{eq:integrationbyparts}.  The first integral is bounded by
\[
\begin{split}
2\pi |I_{1} (k)|&\leq  \| e^{V} u\|_{p}  \| \Lambda^{-1}T^{k} \Lambda u\|_{q} \ \leq \
C (1-\tfrac{c}{W})^{k} \| e^{V} u\|_{p} \|u\|_{q} \leq  \mc{O} (W^{-1})  (1-\tfrac{c}{W})^{k}.
\end{split}
\]
where we used \eqref{cor:Tnormest} and Corollary~\ref{c:bounds}.
The second integral is bounded by
\[
\begin{split}
2\pi |I_{2} (k)|&\leq  \| \Lambda^{-1} T^{k} (e^{-V}-u)\|_{p} \  \| e^{V}T^{k'} \Lambda u\|_{q} \\
&\leq \
C (1-\tfrac{c}{W})^{k+k'}  W^{+\frac{1}{2}-\frac{1}{p}} W^{-\frac{1}{q}-\frac{1}{2}}=\mc{O} (W^{-1})  (1-\tfrac{c}{W})^{k+k'}.
\end{split}
\]
where in the first term we used again \eqref{cor:Tnormest} together with \eqref{eq:uVerrorbound} for $p>2.$
In the second term we used $ \| e^{V}\Lambda u\|_{q}=\mc{O} (W^{-\frac{1}{q}-\frac{1}{2}})$ (cf. Corollary~\ref{c:bounds})
for the case $k'=0$.  When  $k'\geq 1$ we apply \eqref{eq:additionalprop} to get
$ \| e^{V}T^{k'} \Lambda u\|_{q}\leq  \| T^{k'-1} \Lambda u\|_{q}+ \frac{C}{W} \| \Lambda^{-1}T^{k'-1} \Lambda u\|_{q}.$
The estimate now follows from Lemma~\ref{le:lambdaubound} and \eqref{cor:Tnormest}.
Note that we are forced to estimate the factor $e^{V}$ together with $T^{k'} \Lambda u$ since for $k=0$ the term
$e^{V} (e^{-V}-u)$ is not integrable. Finally the constraint $k+k'+k''=N-1$ guarantee that $k\leq 1$ or $k'+k''\geq 1.$
We can assume without loss of generality $k\geq 1.$ Then using  \eqref{eq:uVerrorbound} for $p=2$
\[
\begin{split}
2\pi |I_{3} (k)|&\leq  \| \Lambda^{-1}e^{V} T^{k} (e^{-V}-u)\|_{2} \  \| T^{k'} \Lambda T^{k''} (e^{-V}-u)\|_{2} \\
&\leq \
C (1-\tfrac{c}{W})^{N-1}  (\ln W)^{2} =\mc{O} (W^{-1})  (1-\tfrac{c}{W})^{k+k'+k''}.
\end{split}
\]
if $N\geq C (E)W\ln W,$ for $C (E)>0$ some large constant. This completes the proof of
\eqref{eq:th2eq1}.
Performing the sum over $y,y'$ we obtain $|\partial_{E}\rho_{N} (E)|\leq C.$

\subsubsection{Case $n>1$}
As in the case $n=1$ we insert the decomposition $e^{-V}=u+ (e^{-V}-u),$ and reorganize the integral
(eventually  applying also \eqref{eq:integrationbyparts})
as the sum of three terms of the following form:
\begin{equation}
\begin{split}
I_{1}  :=&\int \tfrac{\ud\lambda_1\ud\lambda_2}{2\pi }\, e^{V}\Lambda^{-1}
\Big [ \prod_{j=0}^{l} (T^{m_{j}}\Lambda^{n_{j}}) u \Big ]
\Big [  \prod_{k=0}^{l'} (T^{m'_{k}}\Lambda^{n'_{k}}) u \Big ]
 \\
I_{2}   := &\int \tfrac{\ud\lambda_1\ud\lambda_2}{2\pi }\, e^{V}\Lambda^{-1}
\Big [ [\prod_{j=1}^{l} (T^{m_{j}}\Lambda^{n_{j}})  T^{m_{0}} (\Lambda^{n_{0}} u) \Big ]
\Big [  T^{m'_{0}}\Lambda^{n'_{0}} \prod_{k=1}^{l'} (T^{m'_{k}}\Lambda^{n'_{k}}) T^{\bar m'} (e^{-V}-u) \Big ]\\
I_{3}  :=& \int \tfrac{\ud\lambda_1\ud\lambda_2}{2\pi }\, e^{V}\Lambda^{-1}
\Big [ \prod_{j=0}^{l} (T^{m_{j}}\Lambda^{n_{j}}) T^{\bar m} (e^{-V}-u)  \Big ]\ 
\Big [ \prod_{k=0}^{l'} (T^{m'_{k}}\Lambda^{n'_{k}}) T^{\bar m'} (e^{-V}-u) \Big ],
\end{split} 
\end{equation} 
where $l\geq 1,$ $l'\geq 0,$  $n_{j},n'_{k}\geq 1$ and  $m_{j},m'_{k}\geq 0$ for all $j,k,$
with the constraint  $\sum_{j=0}^{l}n_{j}+\sum_{k=0}^{l'}n'_{k}=n.$
Finally $\bar m,\bar m'\geq 0 $ but
must satisfy the constraints $\sum_{j=0}^{l}m_{j}+\sum_{k=0}^{l'}m'_{k} +\bar m+\bar m'=N-1,$
The proof now works as in the case $n=1$ and yields
\[
| I_{x, y_{1},\dotsc ,y_{q}}^{m,n_{1},\dotsc ,n_{q}}|\leq \ \frac{1}{W} C^{n} \prod_{j=1}^{q} (\sqrt{n_{j}!})
\Big[ (1-\tfrac{c}{W})^{y_{max}-y_{min}}+ (1-\tfrac{c}{W})^{y_{max}-1}+ (1-\tfrac{c}{W})^{N-y_{min}}+ (1-\tfrac{c}{W})^{N}   \Big]
\]
where $y_{max}:=\max [y_{q},x],$ and  $y_{min}:=\min [y_{1},x].$ Hence
\[
\begin{split}
\pi |\partial_{E}^{n}\rho_{N} (E)|& \leq \ \frac{C^{n}n!}{W}\  \frac{1}{N} 
 \sum_{q=1}^{n}\ (N W^{q}+W^{q+1})\ \prod_{j} \Big (\sum_{n_{j}}\frac{1}{\sqrt{n_{j}!}}\Big )\ \leq 
 {C'}^{n}n! W^{n-1}.
 \end{split}
\]

\subsection{The case $E=0$}\label{sect:E=0}
At $E=0$ the factor $e^{V}$ may develop a pole. To solve the problem we use \eqref{eq:newid} to replace $u$ by
$u=\mc{T}u_{0}+ T (u-u_{0})$ before doing any other manipulation. Formulas become slightly more cumbersome, but
each factor $e^{V}$ now comes with a prefactor $e^{-V}.$

\section{Contour deformation}\label{S:deform}

To extend the proof to the entire range $E \in (-2, 2)$, the contour of integration has to be
deformed.
One possible strategy (followed in \cite{Disertori_GUE}) is to rotate the contour. The rotation angle must ensure that 
$\Re V$ has only one non-degenerate global minimum at the saddle point. This can only be achieved for a rotation angle
close to $\pi/6$ (cf. \cite[5.1.2]{Disertori_GUE}). However the corresponding transfer operator
$e^{-\Re V/2} e^{-W^{2} (\lambda-\lambda')}e^{-\Re V/2} $ is no longer longer self-adjoint.
Another strategy developed in \cite{DS} consists in  performing  a complex  rotation  
that makes the operator approximately normal. The results in \cite{DS} require 
the resulting function $e^{-\Re V}$ to have only one non degenerate
global minimum. In the present case we would need to rotate by
approximately $\pi/8,$ but then $e^{-\ Re V}$  has two minima. 

Here we therefore use the following strategy: first (Section~\ref{s:contour}), we find a 
good contour $\Gamma$ for the bosonic variable. After the contour deformation, 
the operator $\mc T$ is transformed to a new operator $\mc T_\Gamma$.
The main technical difficulty is to find a replacement for the operator
norm bound of Proposition~\ref{specbound}. We show (Proposition~\ref{p:contour})
that a similar bound holds when the operator is replaced with its $k$-th power,
where $k$ is a sufficiently large number, independent of $W$. Having this 
bound at hand, the proof follows the lines of its counterparts for $|E|<\sqrt\frac{32}9$.

We fix an energy $|E|<2$, the dependence on which is omitted from the notation. An
inspection of the argument shows that all the estimates are uniform on
compact subintervals of $(-2, 2)$.

\subsection{Choice of the contour}\label{s:contour}
Decompose
\[ e^{-V(\lambda)}= e^{-V_1(\lambda_1) - V_2(\lambda_2)}~, \quad 
e^{-V_1(\lambda_1)} = e^{-\half \lambda_1^2 - \mc E \lambda_1}\frac{1}{ \ol{\mc E} -\lambda_1}~, \quad
e^{-V_2(\lambda_2)} = e^{-\half \lambda_2^2 + i \mc E\lambda_2} \left( \ol{\mc E}-i\lambda_2\right)~. \]
\begin{lemma}\label{l:contour}
For any $|E| < 2$ there exists a contour $\Gamma$ and numbers $C_\Gamma, c_\Gamma > 0$ such that 
\begin{enumerate}[(1)]
\item $\Gamma$ contains the segments $(-\infty, -C_\Gamma]$, 
$[-c_\Gamma, c_\Gamma]$, $[C_\Gamma, \infty)$; 
\item the angle between $\Gamma$ and the real axis stays in the range $(-\frac\pi4(1-c_\Gamma), \frac\pi4(1-c_\Gamma))$;
\item $e^{-V_1}$ is analytic in $\Gamma^+ = \bigcup\limits_{\substack{a, a' \in \Gamma\\
|a-a'| < c_\Gamma}} [a, a']$;
\item $\Gamma$ is homologous to $(-\infty, \infty)$ in the domain of analyticity of $e^{-V_1}$;
\item $\min_{a \in \Gamma^+}  \Re V_1(a)$ is uniquely attained at $a=0$.
\end{enumerate}
\end{lemma}
Note that when $|E|\to 2$ we  need to take $c_{\Gamma }\to 0$ too.
The proof is an elementary verification. We reduce it to a similar verification
already performed in \cite{Disertori_GUE}.
\begin{proof}
In \cite[Section 5.1.2]{Disertori_GUE} it is proved that for any $|E| < 2$ there 
exists $\zeta$ with $|\zeta| = 1$, $|\arg \zeta|<\pi/4$ such that 
$\min_{a \in \zeta \mathbb{R}}  \Re V_1(a)$ is uniquely attained at $a=0$
and the singularity of $\Re V_1$ does not lie between $\zeta \mathbb{R}$ and $\mathbb{R}$. For
$C > c > 0$, denote by $\Gamma(c, C)$ the piecewise linear contour
going from $-\infty$ to $\infty$ via the points $-c-2\Re \zeta C$, $-c-\zeta C$, $-c$, $c$, $c + \zeta C$, $c + 2 \Re \zeta C$: 
\[ \begin{split}
\Gamma(c, C) &= (-\infty, -c-2\Re \zeta C] \, (1) + [-c-2\Re\zeta C, -c-\zeta C] \, (2) + 
[-c - \zeta C, -c] \, (3) \\
&+ [-c, c] \, (4)  + [c, c + \zeta C] \, (5) + [c+\zeta C, c + 2 \Re \zeta C] \, (6) + 
[c + 2 \Re \zeta C, \infty) \, (7)~.
\end{split} \]
We first choose a large $C>0$ and then a small $c > 0$. For sufficiently large $C$ one has $\Re V_1 > \const > 0$ in the entire domain
\[ \Big\{ |z| > C~, \quad \text{$z$ lies between $\mathbb{R}$ and $\zeta\mathbb{R}$}\Big\}~.  \]
In particular, one has $\Re V_1 > \const > 0$ on the 
four segments $(1), (2), (6), (7)$. For this value of $C$, one can choose $c>0$
sufficiently small so that, by a continuity argument,the minimum of $\Re V_1$ on the union of the remaining segments $(3), (4), (5)$ is uniquely attained at the origin.
For these values of $C$ and $c$, let $\Gamma = \Gamma(c, C)$.

Then $\Gamma$ satisfies the conditions (1)--(4), and a weakened form of (5) with
$\Gamma$ in place of $\Gamma^+$. By an additional continuity argument, (5) also holds
as stated provided that $c_\Gamma$ is chosen sufficiently small.
\end{proof}

\noindent For a contour $\Gamma$, denote by $K_\Gamma$ the integral operator
with kernel 
\[
K_\Gamma (\lambda,\lambda') = \frac{W^2}{2\pi} \exp\left\{ - \frac{V(\lambda)}{2} - \frac{W^2}{2} (\lambda - \lambda')^2 - 
\frac{V(\lambda')}{2} \right\} \]
acting on $L_p(\Gamma \times \mathbb{R})$. Here we use the convention
$\lambda^2 = \lambda_1^2 + \lambda_2^2$.
The main technical step is the following proposition, the proof of which will be 
the subject of  the next Section~\ref{s:normbd}.
\begin{prop}\label{p:contour}
Let $\Gamma$ be a contour satisfying the conclusions (1)--(5)
of Lemma~\ref{l:contour}. Then there exists $k \geq 1$ such that for any $p \in (1, \infty)$
\[ \| K_\Gamma^k \|_p \leq 1 - \frac{c_p}{W}~. \]
\end{prop}

\subsection{Proof of the operator norm bound}\label{s:normbd}

To prove Proposition~\ref{p:contour}, we study the kernel of the operator
$K_\Gamma^k$. 

\begin{lemma}\label{l:lapl} Let $\Gamma$ be a contour satisfying the conclusions (2)--(3)
of Lemma~\ref{l:contour}.  For any  $k$,
\[\begin{split} 
&K_\Gamma^k(\lambda, \lambda') = 
\exp\left\{ -\frac{W^2}{2k}(\lambda - \lambda')^2 - \frac{V(\lambda)}{2} - V(\hat\lambda^1) - 
V(\hat\lambda^2) - \cdots -  V(\hat\lambda^{k-1})  - \frac{V(\lambda')}{2}\right\} \\
&\quad\times \frac{W^2}{2\pi k} \, (1 + \mc O(W^{-2}))~, \qquad \lambda,\lambda' \in \Gamma\times\mathbb{R}~, \, |\lambda - \lambda'| < c_\Gamma/2~,\end{split}\]
where $\hat\lambda^j = \lambda + \frac{j}{k}(\lambda' - \lambda)$, and the asymptotics
is uniform on compact sets.
\end{lemma}

\begin{proof} The proof proceeds by a saddle point analysis. Consider the integral
\[
K_\Gamma^k(\lambda, \lambda') = \int_{(\Gamma \times \mathbb{R})^{k-1}}  K_\Gamma(\lambda, \lambda^1) \cdots K_\Gamma(\lambda^{k-1}, \lambda') \,
\prod_{j=1}^{k-1} \ud\lambda^j_1 \ud\lambda^j_2~.
\]
The saddle point equations 
\[
\lambda^j = \frac{\lambda^{j-1} + \lambda^{j+1}}{2} \quad \text{with} \quad
\lambda^0 = \lambda, \, \lambda^k = \lambda'
\]
have a unique solution given by
$\hat{\boldsymbol \lambda} = (\hat\lambda^1, \cdots, \hat\lambda^{k-1})$.
Extracting the saddle contribution we get
\[
K_\Gamma^k(\lambda, \lambda') = \tfrac{W^2}{2\pi k}
e^{-\frac{W^2}{2k}(\lambda - \lambda')^2}e^{- \frac{V(\lambda)}{2} - V(\hat\lambda^1) - 
V(\hat\lambda^2) - \cdots -  V(\hat\lambda^{k-1})  - \frac{V(\lambda')}{2}}\, \tilde{K}_\Gamma^k(\lambda, \lambda')
\]
where
\[
\begin{split}
 \tilde{K}_\Gamma^k(\lambda, \lambda')&= k \int_{(\Gamma \times \mathbb{R})^{k-1}}   \hspace{-0,4cm}
 e^{-W^2 (\phi(\boldsymbol \lambda_1)  + \phi(\boldsymbol \lambda_2))}
 e^{-\sum_{j=1}^{k-1} [V_\iota(\lambda_\iota^{j})- V_\iota(\hat\lambda_\iota^{j})]}
\prod_{j=1}^{k-1} \tfrac{W^{2}}{2\pi} \ud\lambda^j_1 \ud\lambda^j_2~, \\
\phi(\boldsymbol \lambda_\iota) &=
\frac12 \sum_{j=0}^{k-1} \left[ (\lambda^{j}_\iota-\lambda^{j+1}_\iota)^2 - \frac{(\lambda_\iota'-\lambda_\iota )^2}{k^2} \right]= 
\frac12 \sum_{j=0}^{k-1}\left (\lambda^{j}_\iota-\lambda^{j+1}_\iota+\tfrac{\lambda'-\lambda}{k} \right)^{2}
\end{split} 
\]
To conclude it is enough to prove that  $ \tilde{K}_\Gamma^k(\lambda, \lambda')= (1 + \mc O(W^{-2}))$
for $\lambda,\lambda' \in \Gamma\times\mathbb{R}$ and $ |\lambda - \lambda'| < c_\Gamma/2.$

Perform a contour deformation in each of the variables $\lambda^j_1$, so that $\Gamma$ is replaced
with a homologous contour $\Gamma'(\lambda_1, \lambda_1')$ which contains the straight segment 
\[
L(\lambda_1, \lambda_1') =
[\lambda_1 - 3k W^{-2/5}e^{i \arg(\lambda_1' - \lambda_1)},  \lambda_1' + 3k W^{-2/5}e^{i \arg(\lambda_1' - \lambda_1)}]
\]
and still satisfies the conclusions (2)--(3) of Lemma~\ref{l:contour}. We claim that for $\iota = 1,2$ one has 
\begin{equation}\label{eq:qfbd}
\Re \phi(\boldsymbol \lambda_\iota) \geq c_k |\boldsymbol \lambda_\iota -  {\boldsymbol {\hat \lambda}_\iota}|^2
\end{equation}
on the integration contour.   For $\iota = 2$ 
$\Re \phi(\boldsymbol{\lambda}_{2}) =\phi(\boldsymbol{\lambda}_{2})$ and the result   from the positive definiteness of the quadratic form.  Also, for  $\iota = 1$  we have 
\begin{equation}\label{eq:qfbd2} \Re \phi(\boldsymbol{\lambda}_{1}) \geq
\tilde{c} \sum_{j=0}^{k-1}
\left | (\lambda^{j}_{1}-\hat{\lambda}^{j}_{1})- (\lambda^{j+1}_{1}-\hat{\lambda}^{j+1}_{1}) \right|^{2}\end{equation}
in each of the following regions:  
\begin{equation} \operatorname{reg}_{a}=\left\{\boldsymbol{\lambda}_{1}: \quad \forall j \,  \lambda_1^j \in L(\lambda_1, \lambda_1')\right\}~, \quad \operatorname{reg}_{b}=\left\{\boldsymbol{\lambda}_{1}: \quad \max_j |\lambda_1^j| \geq C_0\right\}~, \end{equation}
when $C_0 = C_0(k)$ is chosen to be sufficiently large.  For these regions, (\ref{eq:qfbd2}) follows   from the condition (2) on the slope.  To prove (\ref{eq:qfbd2}) for the remaining values of $\boldsymbol{\lambda}_{1},$
let $(x(t), y(t))_{t \in \mathbb R}$ be a parametrisation of $\Gamma'(\lambda_1, \lambda_1')$.
Then for $j=1,\dotsc k-1$
\begin{equation}\label{eq:der}
\begin{split}&\frac{\partial}{\partial t_j}  \Re \phi(x(t_1) + i y(t_1), x(t_2) + i y(t_2), \cdots) \\
&\qquad= (x(t_j) - \frac{x(t_{j+1}) + x(t_{j-1})}{2}) x'(t_j) -  (y(t_j) - \frac{y(t_{j+1}) + y(t_{j-1})}{2}) y'(t_j)~.\end{split}
\end{equation}
Taking into account condition (2),  we obtain that  (\ref{eq:der}) has a definite sign when  $\lambda_1^j$ lies outside the curvilinear interval containing $\lambda_1^{j-1}$ to $\lambda_1^{j+1}$: indeed, if, for example, $\lambda_1^j$ lies to the right of this interval, then 
\[\begin{split} &\arg \left[ (x(t_j) - \frac{x(t_{j+1}) + x(t_{j-1})}{2}) + i (y(t_j) - \frac{y(t_{j+1}) + y(t_{j-1})}{2})\right] \in (-\frac\pi4, \frac\pi4)~, \\
& \arg(x'(t_j) - i y'(t_j)) \in (-\frac\pi4, \frac\pi4)~,
\end{split}\]
whence $\eqref{eq:der}>0$. Therefore the minimum of $\phi(\boldsymbol \lambda_1)$ is attained when $\lambda_1^j$ lies between $\lambda_1^{j-1}$ and $\lambda_1^{j+1}$ on the contour. This is true for any $j$, hence the minimum of $\phi$ in the part of the contour defined by $\max_j |\lambda_1^j| \leq C_0$ is attained  when the coordinates $\lambda_1^j$ are ordered, and in particular all lie in $L(\lambda_1, \lambda_1')$.  Hence $\min_{\Gamma'(\lambda_1, \lambda_1')\setminus \operatorname{reg}_{a} }\Re \phi(\boldsymbol{\lambda}_{1})>c>0.$ 
This completes the proof of (\ref{eq:qfbd2}) and hence also of (\ref{eq:qfbd}).

Now split the integral into two pieces,
\[ J_1 = \int_{|\hat{\boldsymbol\lambda} - \boldsymbol \lambda| < W^{-2/5}}~, \quad
 J_2 = \int_{|\hat{\boldsymbol \lambda} - \boldsymbol \lambda| \geq W^{-2/5}}~.\] 
In $J_1$
 we approximate  $e^{-\sum (V(\lambda^j) - V(\hat\lambda_j))}$ by a linear function:
 \[ e^{-\sum_{j=1}^{k-1} V(\lambda^j)} =
 e^{-\sum_{j=1}^{k-1} V(\hat{\lambda}^j)} \left\{ 1 + \sum_{j=1}^{k-1} \sum_{\iota=1,2} \left( \partial_\iota V\right) (\hat{\lambda}^j) (\lambda^j_\iota - \hat\lambda^j_\iota)
 + \mc O(|\hat{\boldsymbol\lambda} - \boldsymbol \lambda|^{2})\right\} \]
 After replacing the left-hand side with  the right-hand side, we may extend the integral to the full straight line containing $L(\lambda_1, \lambda_1')$ (at the expense of adding a negligible term); then the constant term gives the asymptotics, the integral of the 
linear term vanishes by symmetry, and the error term is $\mc O(W^{-2})$. For $J_2$, we insert absolute values and use (\ref{eq:qfbd}).
\end{proof}

\begin{proof}[Proof of Proposition~\ref{p:contour}]
By Riesz--Thorin interpolation, it suffices to show that, for sufficiently large $k$,
\begin{align}
\|K_\Gamma^k\|_1, \|K_\Gamma^k\|_\infty &\leq 1 +\mc o(W^{-1})  \label{eq:bd1}\\
\|K_\Gamma^k\|_2 &\leq 1 - \frac{c}{W} \label{eq:bd2}~.
\end{align}
To prove (\ref{eq:bd1}) we recall (\ref{eq:opnorms}), which implies that
\[\max(\|K_\Gamma^k\|_1, \|K_\Gamma^k\|_\infty)=
\sup_\lambda \int_{\Gamma \times \mathbb{R}} |d \lambda'_1| |d\lambda'_2|  |K_\Gamma^k(\lambda, \lambda')|~. \]
Consider two cases. If $|\lambda| \leq 2 C_\Gamma$, we
split
\[ \int |d \lambda'_1| |d\lambda'_2| |K_\Gamma^k(\lambda, \lambda')| 
= \int_{|\lambda' - \lambda| < c_\Gamma/3} +  \int_{|\lambda' - \lambda| \geq c_\Gamma/3} 
= J_1 + J_2~.\]
According to Lemma~\ref{l:lapl},
\[\begin{split}
|J_1| &\leq  \max_{|\lambda' - \lambda| < c_\Gamma/ 3} \sqrt{\frac{|\lambda_1-\lambda_1'|^2}{\Re (\lambda_1 - \lambda_1')^2}} + \mc o(W^{-1}) \\
&\leq 
\begin{cases} 
1 + \mc o(W^{-1})~, &|\lambda_1|<2 c_\Gamma/3 \\
\cos^{-1/2} \left(\frac\pi2 (1-c_\Gamma)\right) \exp(-k \Re V_1(c_\Gamma/3)) +
\mc o(W^{-1}),&    |\lambda_1|\geq 2 c_\Gamma/3~.
\end{cases}\end{split}\]
Choosing $k$ sufficiently large we can ensure that
\[ J_1 \leq 1 + \mc o(W^{-1})~, \quad |\lambda|\leq 2 C_\Gamma~. \]
The bound obtained by taking absolute values and bounding  
$\Re V \geq 0$ in the definition of $K_\Gamma$ suffices to see that 
$J_2$ is exponentially small in $W^2$.

In the case $|\lambda| \geq 2 C_\Gamma$ we split
\begin{equation}\label{eq:faraway} \int |d \lambda'_1| |d\lambda'_2| |K_\Gamma^k(\lambda, \lambda')| 
= \int_{|\lambda' - \lambda| < C_\Gamma/2} +  \int_{|\lambda' - \lambda| \geq C_\Gamma/2}~.\end{equation}
Arguing as before, with the bounds for the non-deformed case in place of 
Lemma~\ref{l:lapl} for the first integral, we obtain:
that the left-hand side of (\ref{eq:faraway}) is less than one and, in, fact,
decays exponentially with $k$. This concludes the proof of (\ref{eq:bd1}).

\medskip To prove (\ref{eq:bd2}) we apply semiclassical reasoning as
in the proof of Proposition~\ref{specbound}. Let $\chi_1,\chi_2 $ be smooth bump functions in a small neighbourhood of the minima $0$, $(0,\sqrt{1-\frac{E^2}4}) $ of $\Re V $
(we make sure that the radius of the neighbourhood is $\leq c_\Gamma/10$).
Let $\chi_3 $ be such that $\chi_1^2+\chi_2^2 + \chi_3^2=1 $.  Then
\[\begin{split}
|K_\Gamma^k(\lambda, \lambda')| = \sum_{i=1,2,3} \chi_i(\lambda) |K_\Gamma^k(\lambda,\lambda')| \chi_i(\lambda') + \frac12
\sum_{i=1,2,3} |K_\Gamma^k(\lambda, \lambda')|[\chi_i(\lambda) - \chi_i(\lambda')]^2~.
\end{split}\]
The norm of the second term is $\mc O(W^{-2})$. Next,
\[\begin{split}
\left\Vert    \sum_{i=1,2,3} \chi_i |K_\Gamma^k| \chi_i  \right\Vert_2 
&= \sup_{\Vert \phi\Vert_2=1} \sum_{i=1,2,3} (\chi_i\phi, |K_\Gamma^k| \chi_i\phi)_\Gamma\\
&\leq  \max_i \Vert \mathbbm{1}_{\supp \chi_i} |K_\Gamma^k| \mathbbm{1}_{\supp \chi_i}\Vert_2\sup_{\Vert \phi\Vert_2=1} \sum_{i=1,2,3} \Vert \chi_i\phi\Vert_2^2 \\
&= \max_i \Vert \mathbbm{1}_{\supp \chi_i} |K_\Gamma^k| \mathbbm{1}_{\supp \chi_i}\Vert_2~.
\end{split}\]
For $i=1,2$ we use Lemma~\ref{l:lapl} and bound 
\[ \Vert \mathbbm{1}_{\supp \chi_i} |K_\Gamma^k| \mathbbm{1}_{\supp \chi_i}\Vert_2
\leq 1 - c/W\]
as in the proof of Proposition~\ref{specbound}. For $i=3$ we argue as in 
the proof of (\ref{eq:bd1}) above and show that 
\[ \Vert \mathbbm{1}_{\supp \chi_3} |K_\Gamma^k| \mathbbm{1}_{\supp \chi_3}\Vert_2
\leq c < 1\]
for sufficiently large $k$.
\end{proof}

\subsection{Proofs of the Theorems close to the edge}

Let $\Gamma$ be a contour as in Lemma~\ref{l:contour}. Define
an operation $\mc T_\Gamma$  on functions 
$\Gamma \times \mathbb{R} \to \mathbb{C}$ via the same formula
\[ \mc T_\Gamma = e^{-V - W^2 \lambda} \delta_0 + T_\Gamma~,
\quad T_\Gamma = e^{-V} \Lambda \delta_{W^2, \Gamma}^\star \Lambda^{-1}~, \]
where for example
\[
(\delta_{W^2, \Gamma}^\star f)(\lambda) =  \frac {W^2}{2\pi}  \int_{\Gamma \times \mathbb{R}}  e^{-\half W^2  (\lambda'-\lambda)^2 }  f(\lambda')  \ud\lambda_1'\ud\lambda_2 '~.
\]
By analyticity, the formul{\ae} of Corollary~\ref{corfinalrep} remain
valid with $\mc T_\Gamma$ in place of $\mc T$.
For the same reason, in $R$ coordinates all formulas still hold replacing the integration contour for the variable $a$ by $\Gamma.$

Let us show (cf.\ Remark~\ref{rem:allE}) that the conclusion of Proposition~\ref{prop:approx}  remains valid
for the deformed operator. Indeed, take $u_0^{(M)}$ from Proposition~\ref{prop:approx}.
and denote by $u_{0,\Gamma}^{(M)}$ its analytic
continuation on $\Gamma$. By condition (1) of Lemma~\ref{l:contour} the contour $\Gamma$ goes along the real
axis in the vicinity of the origin, whereas, by condition (2)  $u_{0,\Gamma}^{(M)}$ decays
away from the origin. Therefore $u_{0,\Gamma}^{(M)}$ boasts the same
properties as the approximate eigenfunction from Proposition~\ref{prop:approx}. In particular,  setting
$v_\Gamma^{(M)} := \mc T_\Gamma u_{0,\Gamma}^{(M)} - u_{0,\Gamma}^{(M)}$, $ \|v_\Gamma^{(M)}\|_p $ and $\|\Lambda^{-1} v_\Gamma^{(M)}\|_p$
enjoy the same bounds as  $ \|v^{(M)}\|_p $ and $\|\Lambda^{-1} v^{(M)}\|_p.$
By Proposition~\ref{p:contour}, one can choose $k$ so that
$\| (e^{-V/2} \delta^\star_{W^2,\Gamma}e^{-V/2})^k \|_p  \leq 1 - c_p W^{-1}$, hence
$\| (e^{-V} \delta^\star_{W^2,\Gamma})^n\|_p \leq C^{k} (1 - c_p W^{-1})^{\lfloor \frac{n-1}{k}\rfloor}$. Therefore Proposition~\ref{recon}  also
remains valid for $\mc T_\Gamma$, that is, 
 $u_{0,\Gamma}^{(M)}$ can be upgraded to an exact solution of
$\mc T_\Gamma u_\Gamma = u_\Gamma$ with the usual bound on the error $ u_\Gamma - u_{0,\Gamma}^{(M)} .$
From this point the argument proceeds as in Section~\ref{prf}.

\begin{appendix}
\section{Some additional useful identities}

Recall the definition of $d\mu_{z} (R):=  \ud R\,  e^{-\frac{z}{2 }\str R^{2} }$ from Section~\ref{sec:susyid1},
where  $z\in \mathbb{C},$ with $Re (z)>0.$
The following two lemmas extend Lemma \ref{le:someidentities} to a larger set of polynomials. Their proof
follows the same strategy as the one of Lemma \ref{le:someidentities}.
\begin{lemma}\label{le:someidentities2}
For any  $z\in \mathbb{C},$ with $Re (z)>0$ and supermatrix $R$ the following identities  hold:
\begin{align*}
(a)\qquad & \int d\mu_{z} (R')\ \str (R'+R)^{2} \str (R'+R) = \str R^{2} \str R+ \frac{2}{z}   \str R  \\
(b)\qquad  & \int d\mu_{z} (R')\ \str (R'+R)^{2} [\str (R'+R)]^{n} =  \str R^{2}  (\str R)^{n}+ \frac{2n }{z}   (\str R)^{n} \ \forall n\geq 1 \\
(c)\qquad &\int d\mu_{z} (R')\ [ \str (R'+R)^{2}]^{2}=  [\str R^{2}]^{2} +   \frac{4 }{z}   \str R^{2})\\
(d)\qquad & \int d\mu_{z} (R')\ \str (R'+R)^{3} \str (R'+R) =\str R^{3}  \str R + \frac{3}{z}   \str R^{2} \\
(e)\qquad & \int d\mu_{z} (R')\ \str (R'+R)^{5} = \str R^{5} +\frac{5 }{z}   \str R \str R^{2} +\frac{5 }{z}   \str R  \\
(f)\qquad  & \int d\mu_{z} (R')\ \str (R'+R)^{2} [\str (R'+R)]^{3} =
 \str R^{2} (\str R)^{3} + \frac{6}{z}  (\str R)^{3}\\
(g)\qquad  & \int d\mu_{z} (R')\  \str (R'+R)^{2} [\str ( R'+R)]^{4}= \str R^{2}  (\str R)^{4} 
+  \frac{8}{z}  (\str R)^{4}\\
\end{align*}
\end{lemma}
\begin{lemma}\label{le:someidentities3}
Let $\mu = W/(W+2\alpha)$ and set $W\gg 1$. For any supermatrix $R$ the following identities  hold:
\begin{align*}
(a) & \int d\mu_{\frac{W^{2}}{\mu }} (R') \str (R'+R)^{2} \str (R'+R)^{3} = \str R^{2}  \str R^{3}
+  \frac{6\mu }{W^{2}}  \str R \str R^{2} + \mc{O}(W^{-4}R)  \\
(b) & \int d\mu_{\frac{W^{2}}{\mu }} (R') \str (R'+R)^{3} [\str (R'+R)]^{2} = \str R^{3}  (\str R)^{2}
+  \frac{6\mu  }{W^{2}}  \str R^{3}+ \mc{O}(W^{-4}R) \\
\\
(c)  & \int  d\mu_{\frac{W^{2}}{\mu }} (R')\ 
W \str (R'+R)^{3} \str ( R'+R)^{4}=\\
&\  
W\str R^{3}  \str R^{4} +  \frac{2\mu }{W}  \str R^{3} (\str R)^{2} + \frac{12 \mu }{W}  \str R^{5} 
+ \mc{O}(W^{-5}R)+ \mc{O}( W^{-3}R^{3})\\
(d)\  & \int d\mu_{\frac{W^{2}}{\mu }} (R')\ W^{2}\str (R'+W^{-\frac{1}{2}}R)^{3}]^{3} =\\
&\  
W^{2}  [\str (W^{-\frac{1}{2}}R)^{3}]^{3} +3^{3}\mu  \str (W^{-\frac{1}{2}}R)^{3}  \str (W^{-\frac{1}{2}}R)^{4}+ \mc{O}(W^{-2}\sqrt{W}^{-5})
\end{align*}
\end{lemma}

\end{appendix}

\vspace{15pt}\noindent
\textbf{Acknowledgement.} M.D. is supported in part by the DFG via CRC 1060, and acknowledges the hospitality of the Institute for Advanced Study in Princeton and the Newton Institute in Cambridge, where part of this work has been carried out. M.L. is partially supported by NSERC. S.S. is supported in part by the European Research Council starting grant 639305 (SPECTRUM) and by a Royal Society Wolfson Research Merit Award.

\renewcommand*{\bibfont}{\footnotesize}

\printbibliography

\end{document}